\documentclass[12pt]{amsart}
\usepackage[colorlinks=true, pdfstartview=FitV, linkcolor=blue, citecolor=blue, urlcolor=blue]{hyperref}

\usepackage{amssymb,amsmath, amscd}
\usepackage{times, verbatim}
\usepackage{graphicx}
\usepackage[english]{babel}
 \usepackage[usenames, dvipsnames]{color}
\usepackage{amsmath,amssymb,amsfonts}
\usepackage{enumerate}

\usepackage{tikz}
\usetikzlibrary{calc, arrows, decorations.markings} \tikzset{>=latex}
\usepackage{amsmath,amssymb,amsfonts}
\usepackage{amscd, amssymb, latexsym, amsmath, amscd}
\usepackage[all]{xy}
\usepackage{pb-diagram}
\usepackage{enumerate}
\usepackage{ulem}
\usepackage{anysize}
\marginsize{3cm}{3cm}{3cm}{3cm}
\input xy
\xyoption{all}
\usepackage{pb-diagram}
\usepackage[all]{xy}

\usepackage{xcolor}
\input xy
\xyoption{all}

\theoremstyle{plain}
\newtheorem{thm}{Theorem}

\newtheorem{lemma}{Lemma}
\newtheorem{cor}{Corollary}

\newtheorem{prop}{Proposition}

\theoremstyle{definition} \theoremstyle{definition}
\newtheorem{remark}{Remark}
\newtheorem{example}{Example}
\newtheorem{defn}{Definition}

\theoremstyle{remark}

\newcommand{\PP}{\mathbb{P}}

\newcommand{\G}{\textsc{\G}}

\newcommand{\g}{\mathfrak{g}}

\newcommand{\gk}{\mathfrak{k}}
\newcommand{\A}{\mathbb{A}}
\newcommand{\Q}{\mathbb{Q}}

\newcommand{\p}{\mathfrak{p}}

\newcommand{\U}{\mathcal{U}}
\newcommand{\Z}{\mathbb{Z}}

\newcommand{\R}{\mathbb{R}}

\newcommand{\Gm}{\mathbb{G}_m}
\newcommand{\h}{\mathfrak{h}}
\newcommand{\C}{\mathbb{C}}
\newcommand{\Si}{\mathbb{S}}

\newcommand{\wM}{\widehat{\sM}}
\newcommand{\wL}{\widehat{\sL}}
\newcommand{\wP}{\widehat{\sP}}

\newcommand{\wG}{\widehat{\sG}}
\newcommand{\wT}{\widehat{\sT}}
\newcommand{\wZ}{\widehat{\sZ}}
\newcommand{\wB}{\widehat{\sB}}

\newcommand{\qH}{\mathbb {H}}
\newcommand{\Hom}{{\rm Hom}}
\newcommand{\Ext}{{\rm Ext}}

\newcommand{\lrta}{\longrightarrow}

\def\Uk{{\mathfrak{U}}}
\def\Zk{{\mathfrak{Z}}}
\def\g{{\mathfrak{g}}}
\def\gl{{\mathfrak{g}{\ell}}}
\def\h{{\mathfrak{h}}}
\def\b{{\mathfrak{b}}}
\def\a{{\mathfrak{a}}}
\def\n{{\mathfrak{n}}}
\def\q{{\mathfrak{q}}}
\def\l{{\mathfrak{l}}}
\def\u{{\mathfrak{u}}}
\def\t{{\mathfrak{t}}}
\def\b{{\mathfrak{b}}}
\def\k{{\mathfrak{k}}}

\def\G{{\rm G}}
\def\T{{\rm T}}

\def\SL{{\rm SL}}
\def\Spin{{\rm Spin}}

\def\Sp{{\rm Sp}}

\def\SU{{\rm SU}}
\def\U{{\rm U}}

\def\sG{{\mathsf{G}}}
\def\sQ{{\mathsf{Q}}}
\def\sZ{{\mathsf{Z}}}
\def\sH{{\mathsf{H}}}
\def\sZ{{\mathsf{Z}}}
\def\sB{{\mathsf{B}}}
\def\sP{{\mathsf{P}}}
\def\sA{{\mathsf{A}}}
\def\sT{{\mathsf{T}}}
\def\sL{{\mathsf{L}}}
\def\sM{{\mathsf{M}}}

\def\GL{{\rm GL}}
\def\PGL{{\rm PGL}}

\def\Gal{{\rm Gal}}
\def\SO{{\rm SO}}
\def\OO{{\rm O}}
\def\Out{{\rm Out}}

\def\Sym{{\rm Sym}}

\begin{document}


\title{Cohomological representations for real reductive groups}
\author[Arvind Nair and Dipendra Prasad]{Arvind N. Nair and Dipendra Prasad}

\begin{abstract} 
For a connected reductive group $\sG$ over $\R$, we study cohomological $A$-parameters, which are Arthur parameters with the infinitesimal character of a finite-dimensional representation of $\sG(\C)$.  We prove a structure theorem for such $A$-parameters, and deduce from it that a morphism of $L$-groups which takes a regular unipotent element to a regular unipotent element respects cohomological $A$-parameters. This is used to give complete understanding of cohomological $A$-parameters for all classical groups.  We review the parametrization of Adams-Johnson packets of cohomological representations of $\sG(\R)$ by cohomological $A$-parameters and discuss various examples.  We prove that the sum of the ranks of cohomology groups  in a packet 
on any real group (and with any infinitesimal character)
is independent of the packet under consideration, and can be explicitly calculated.  
This result has a particularly nice form 
when summed over all pure inner forms. 
\end{abstract}

\maketitle
{\hfill \today}

\vspace{.5cm}

\tableofcontents

\section{Introduction}
Let $\mathsf{G}$ be a connected reductive algebraic group over $\R$, $G=\mathsf{G}(\R)$ with $\g_0$ its Lie algebra, and $\g = \g_0 \otimes _{\R}\C$. Let
$K$ be a maximal compact subgroup of $\sG(\R)$ with $K_0$ the connected component of $K$ containing the identity,  $\gk_0$ its Lie algebra, and $\gk = \gk_0 \otimes _{\R}\C$. For any 
continuous representation $\pi$ of $\sG(\R)$, let $\pi$ also denote the associated $(\g,K)$-module. For such a $(\g,K)$-module $\pi$, and   a finite
dimensional representation $V$ of $G$, one is often interested in the cohomology groups $H^i(\g,K_0, \pi \otimes V) = H^i(\g,\gk, \pi \otimes V)$.
The irreducible unitary representations $\pi$ of $(\g,K)$
for which there is a finite-dimensional algebraic representation $V$ of $\mathsf{G}(\C)$  with $H^i(\g,\gk, \pi \otimes V) \not = \{0\}$ for some $i$ are called cohomological representations of $G$ with coefficients in  $V$; thus to emphasize, in this work, the real group $G$ will always be the  real points
of a connected reductive algebraic group over $\R$, and the coefficient system $V$ will always be the restriction
to $\sG(\R)$ of a finite dimensional algebraic representation of $\sG(\C)$. (This use of the term `cohomological representation
of $\sG(\R)$' excludes certain representations of $\sG(\R)$ which might be considered cohomological, such as the next-to-minimal
discrete series representation, call it $D_3$,  of $\GL_2(\R)$, normalized so that its central character is the sign character of $\R^\times$. Since the sign character of $\R^\times$ is not the restriction of an algebraic character of $\C^\times$, $D_3$ is not a
cohomological representation of $\GL_2(\R)$
in our usage of the word but the restriction of $D_3$ to $\SL_2(\R)$ is.)

The aim of this paper is  to study irreducible unitary cohomological  representations from the point of view of Langlands parameters and Arthur parameters 
(abbreviated to $L$-parameters and $A$-parameters). This 
study has been undertaken by several authors in the past,  starting with the fundamental work of Vogan-Zuckerman \cite{VZ} and Adams-Johnson \cite{AJ}, and reformulated and exposed, 
among others, by Arthur \cite{Ar}, Blasius-Rogawski \cite{BR}, and very recently by Taibi \cite{Ta}. Several of these works at various points make assumptions on the group $G$, such as it 
having a discrete series representation or being semisimple. This work makes no such assumption, and adds,  if rather minimally,  to clarifying the $A$-parameters 
associated with cohomological representations of real 
reductive groups.  The following definition is basic to this paper: 

\begin{defn}
Let $\sG$ be a real reductive algebraic group with a Cartan involution $\theta$ on $\sG(\R)$.  A  {\it cohomological $A$-parameter} for $\sG$ or $\sG(\R)$ is an $A$-parameter   $\phi: W_\R \times \SL_2(\C) \rightarrow {}^L\sG$ with the infinitesimal character
(defined in \S \ref{prelim}) of a
finite dimensional irreducible algebraic representation $V$ of $\G(\C)$ with $V\cong V^\theta$.
  \end{defn}

{ We prove in Theorem \ref{cohomo} a structure theorem
about such  cohomological $A$-parameters, and  construct
in Theorem \ref{AJ1} and \ref{AJ2}
(which is mostly a summary of the works of Vogan-Zuckerman \cite{VZ} and Adams-Johnson \cite{AJ})
a map from 
cohomological $A$-parameters to finite sets of cohomological unitary representations of $\sG(\R)$,
called  Adams-Johnson packets (or $AJ$-packets for short). Only recently, in the works of 
N. Arancibia, C. Moeglin and  D. Renard \cite{AMR}, has it been  proven that the Adams-Johnson packets are Arthur packets for quasi-split classical groups.
In this work we will always think of $A$-parameters as giving rise to $AJ$-packets of cohomological representations of $\sG(\R)$ and we will not concern
ourselves with the identification of  $AJ$- with  $A$-packets. In particular, as far as we know, even the $L$-packet underlying an $A$-parameter has not been shown
to be contained in the $AJ$-packet in any generality 
(except of course in \cite{AMR}), though it is expected for any reductive group. }

Eventually, the essence of our work is that although for general parameters, there is a considerable difference between $A$-parameters  $W_\C \times \SL_2(\C)\rightarrow {}^L\sG $, and
$A$-parameters  $W_\R \times \SL_2(\C)\rightarrow {}^L\sG $, for the purposes of cohomological representations, there is none (up to twisting representations of $\sG(\R)$ by a quadratic character). A simple example to keep in mind for this difference between parameters for $W_\C$ and $W_\R$ is already for $\SL_2(\R)$ where to define a unitary principal series, one must
define a character of $\R^\times$ and not just one of $\R^+$ which is what restriction to $W_\C$ of the corresponding $W_\R$ parameter will give, whereas for a discrete series representation of $\SL_2(\R)$, restriction to $W_\C$ uniquely determines
the parameter for $W_\R$. Our main result on cohomological $A$-parameters is Theorem \ref{cohomo}, which is a bit long to state precisely here, so we mention  that, 
roughly speaking, this theorem is about the existence and uniqueness of extensions of a  
parameter  $\phi_\C: W_\C \times \SL_2(\C) \rightarrow {}^L\sG$ to  parameters $\phi: W_\R \times \SL_2(\C) \rightarrow {}^L\sG$, assuming only that the infinitesimal character of $\phi_\C$ is that of a finite-dimensional representation of $\G(\C)$.

Our immediate reason to study $L$- and $A$-parameters of cohomological representations was to understand a basic question about cohomological representations:   How do cohomological representations behave 
under functoriality for morphism of $L$-groups $^L\sG_1 \rightarrow {}^L\sG_2$?
This natural question has not been paid much attention to since  functoriality usually 
does not take cohomological representations of $\mathsf{G}_1(\R)$ to cohomological representations of $\mathsf{G}_2(\R)$. For example, if we take the natural embedding of $L$-groups:
$$\iota: \GL_n(\C) \times \GL_m(\C) \hookrightarrow \GL_{m+n}(\C),
$$
one sees that the infinitesimal character of  the trivial representation of $\GL_n(\R) \times \GL_m(\R)$ does not go to  the infinitesimal character 
of a cohomological representation of $\GL_{n+m}(\R)$ (for any coefficient system) since the lifted representation --- call it $\iota( 1 \times 1)$ --- of $\GL_{n+m}(\R)$ in fact does 
not have the infinitesimal character of a finite dimensional
representation of $\GL_{n+m}(\R)$, let alone $\iota(1 \times 1)$ being cohomological. If $\nu(x)= |x|$ is the character
of $\R^\times$, then the infinitesimal character of the trivial representation of $\GL_n(\R) \times \GL_m(\R)$ is the character of $\R^\times$ given by: 
$$(\nu^{(n-1)/2}, \cdots, \nu^{-(n-1)/2},  \nu^{(m-1)/2}, \cdots, \nu^{-(m-1)/2}),$$ which is not the infinitesimal character of a finite dimensional 
representation of $\GL_{n+m}(\R)$ which is of the form $$(\nu^{\mu_1}, \nu^{\mu_2}, \cdots, \nu^{\mu_{n+m}}) {\rm~for~} 
\mu_1 > \mu_2 > \cdots > \mu_{n+m},$$ with $\mu_i$ all integers, or all half-integers.

However, the following theorem gives a nice and useful case where 
morphism  of $L$-groups $\phi: ^L\sG_1 \rightarrow {}^L\sG_2$ takes cohomological $A$-parameters of $G_1$ to cohomological $A$-parameters of $G_2$. This is an immediate consequence of Theorem \ref{cohomo}
combined with Proposition \ref{inf-ch}.

\begin{thm} \label{main} Let $\mathsf{G}_1$ and $\mathsf{G}_2$ be connected real reductive groups with $  G_1= \mathsf{G}_1(\R)$ and $G_2=\mathsf{G}_2(\R)$. 
Let $^L\sG_1 = \wG_1 \cdot W_\R$ and $^L\sG_2 = \wG_2\cdot W_\R$ be the $L$-groups of  $\mathsf{G}_1$ and $\mathsf{G}_2$ respectively. Let $\phi: {}^L\sG_1 \rightarrow {}^L\sG_2$ be a homomorphism of 
$L$-groups which takes a regular unipotent element in $\wG_1$ to a regular unipotent element in $\wG_2$.
Then $\phi$ takes cohomological  $A$-parameters of $G_1$ to cohomological  $A$-parameters of $G_2$.
Furthermore, if  $\phi: {}^L\sG_1 \rightarrow {}^L\sG_2$ has abelian kernel then for an $A$-parameter $\sigma$ for
$G_1$ for which $\phi(\sigma)$ is a cohomological $A$-parameter for $G_2$, 
  $\sigma$ itself is a cohomological $A$-parameter for $G_1$.
\end{thm}

The condition on a homomorphism $\phi: {}^L\sG_1 \rightarrow {}^L\sG_2$ 
that it takes a regular unipotent element in $\wG_1$ to a regular unipotent element in $\wG_2$ is
met in particular for the following embeddings of classical groups: 
\begin{enumerate}
            \item $\Sp_{2n}(\C) \subset \GL_{2n}(\C)$,
            \item $\SO_{2n+1}(\C) \subset \GL_{2n+1}(\C)$
            \item $\SO_{2n-1}(\C) \subset \SO_{2n}(\C)$.
           \end{enumerate}

The condition on a homomorphism $\phi: {}^L\sG_1 \rightarrow {}^L\sG_2$ 
that it takes a regular unipotent element in $\wG_1$ to a regular unipotent element in $\wG_2$ is also met for 
$\sG_1 = \sG(\R)$, and $\sG_2= {\rm R}_{\C/\R} \sG(\C)$ where ${\rm R}_{\C/\R}$ denotes Weil restriction of scalars 
from $\C$ to $\R$, and for which the embedding of $\wG_1$ in $\wG_2$ is given by the diagonal embedding:

\begin{enumerate}
\setcounter{enumi}{3}

            \item $\wG \hookrightarrow \wG \times \wG.$

           \end{enumerate}

As a consequence of Theorem \ref{main},
we have a rather complete understanding of cohomological $A$-parameters in the following theorem
for $\Sp_{2n}(\R)$ and  $\SO(p,q)(\R)$ with $p+q$ odd. The condition on a homomorphism $\phi: {}^L\sG_1 \rightarrow {}^L\sG_2$ 
that it takes a regular unipotent element in $\wG_1$ to a regular unipotent element in $\wG_2$ is
not met for the embedding $\SO_{2n}(\C) \subset \GL_{2n}(\C)$
since the principal $\SL_2(\C)$ inside $\SO_{2n}(\C)$ corresponds 
to the representation $\Sym^{2n-2}(\C + \C) + \C$, so $\SO(p,q)(\R)$
for $p+q$ even needs to be treated
separately both for the proof as well as for the statement, which we do uniformly in Example \ref{cg} in section \ref{parameters}.

Before we state the next theorem, we recall, cf. \cite{GGP}, 
that $A$-parameters for classical groups can be treated as parameters for $\GL_n(\R)$ (or, $\GL_n(\C)$ for unitary groups) which preserve either a quadratic form, or a symplectic form
(or, are conjugate-selfdual of parity $(-1)^{p+q-1}$ for $\U(p,q)(\R)$).

\begin{thm} \label{main3} Let $\sG$ be one of the  classical groups
  $\Sp_{2n}(\R)$, or $\SO(p,q)(\R)$ with $p+q$ odd. 
Their $L$-groups come equipped with a natural embedding in say $\SL_m(\C)$.
An $A$-parameter for $\sG(\R)$ is cohomological if and only if considered as a parameter with values in $\SL_m(\C)$, 
it is a cohomological $A$-parameter for $\PGL_m(\R)$. Further,
a cohomological parameter for $\PGL_m(\R)$ is automatically a
parameter (hence cohomological) for the symplectic group if $m$ is odd, and odd
orthogonal group if $m$ is even.

If $\sG = \U(p,q)$, then an $A$-parameter for $\sG(\R)= \U(p,q)(\R)$ is cohomological if and only if its restriction to $W_\C$ is a cohomological $A$-parameter for $\sG(\R)=\GL_m(\C)$, 
$m=p+q$.
\end{thm}

\begin{remark}
 Besides the embeddings  of classical groups which go into determining cohomological parameters for classical groups in Theorem 
 \ref{main3}, here are the rest of embeddings of simple groups (with the smaller
 group not $\SL_2(\C)$) for which regular unipotent elements go to regular unipotent elements, cf. exercise 20 in Chapter IX of \cite{Bo}. 
 
 \begin{enumerate}
 \setcounter{enumi}{4}
            \item $F_4(\C) \subset E_6(\C)$,
            \item $G_2(\C) \subset \Spin_{7}(\C) \subset \SO_8(\C)$,
            \item $G_2(\C) \subset \SO_{7}(\C) \subset \SL_7(\C)$.
           \end{enumerate}
 
\end{remark}

Before proceeding further, we first  make the following definition.

\begin{defn}(Tempered companion of an $A$-parameter) \label{companion}
For an $A$-parameter $\Lambda: W_\R \times \SL_2(\C)  \rightarrow {}^L \sG$, we define $T(\Lambda)$ --
{\it the tempered companion}  of the A-parameter $\Lambda$ --
by the following commutative diagram
$$
\xymatrix{
W_\R \ar[r]^-{T(\Lambda)} 
\ar[rd]^-{\sigma_1} & {}^L\sG\\
& W_\R \times \SL_2(\C), \ar[u]_-{\Lambda}
}
$$
 where the parameter $\sigma_1: W_\R \rightarrow W_\R \times \SL_2(\C) $ corresponds to the lowest discrete
 series representation $D_2$ of $\PGL_2(\R)$. (By Lemma \ref{inf} below on the infinitesimal character, the
 infinitesimal characters associated with  $\Lambda$ and $T(\Lambda)$ are the same.)
\end{defn}
Our next  theorem describes $A$-parameters of cohomological representations in terms of
much easier information on $L$-parameters of cohomological tempered representations, and is a direct consequence of Theorem \ref{cohomo}. We state it here as it may have an  independent interest. This theorem may be useful
to think about $A$-parameters, but it also allows
one to read off the $L$-parameter of the tempered  representations of $\sG(\R)$
with the infinitesimal character of the trivial representation.

\begin{thm} \label{coh-para}
  Let $\mathsf{G}$ be a  connected reductive algebraic group defined over $\R$,
  and $F$, a finite dimensional irreducible 
  representation of $\sG(\C)$.
      Then an $A$-parameter
  $\Lambda: \SL_2(\C) \times W_\R \rightarrow {}^L \sG$ is a 
  cohomological $A$-parameter with coefficients in $F$, 
  if and only if $T(\Lambda): W_\R \rightarrow {}^L \sG$,
the tempered companion of the $A$-parameter $\Lambda$ just defined,  is a 
cohomological tempered $L$-parameter with coefficients in $F$, which,
up to twisting by $H^1(W_\R, Z(\wG))$, is independent of $\Lambda$.
In particular, up to twisting by $H^1(W_\R, Z(\wG))$,
a tempered $L$-parameter of $\sG(\R)$ with the infinitesimal character of the trivial representation 
is $T(\Lambda)$ for the $A$-parameter
  $\Lambda: \SL_2(\C) \times W_\R \rightarrow {}^L \sG$ corresponding to the trivial representation of $\sG(\R)$.
\end{thm}

As a simple application of the above theorem, we note the following corollary,
using and generalizing the well-known theorem, that a tempered cohomological representation
of a reductive group $\sG(\R)$ is a discrete series representation if  $\sG(\R)$  has a discrete series
representation.

\begin{cor} \label{discrete}
   Let $\mathsf{G}$ be a  connected reductive algebraic group defined over $\R$. 
      Then an $A$-parameter
  $\Lambda: \SL_2(\C) \times W_\R \rightarrow {}^L \sG$ which is a 
  cohomological $A$-parameter with coefficients in $F$, a finite dimensional irreducible 
  representation of $\sG(\C)$, is discrete (i.e., the image of $\Lambda$ does not land inside a
   parabolic in ${}^L \sG$)
  if $\sG(\R)$ has a discrete series representation. (However, the converse is not true, for example since the
  trivial representation of a group $\sG(\R)$ with $A$-parameter the principal $\SL_2(\C)$ in the $L$-group
  is always a discrete $A$-parameter irrespective of whether $\sG(\R)$ has discrete series or not;
  or the Speh representations on $\GL_{2n}(\R)$ have discrete $A$-parameters.)
  \end{cor}

\begin{remark}
For an $A$-parameter $\Lambda: W_\R \times \SL_2(\C) \rightarrow {}^L \sG$, its 
{ tempered companion} $T(\Lambda)$, is defined much in the same spirit as in the case of $p$-adic groups
where for a $p$-adic field $F$ with Weil group $W_F$,  to an $A$-parameter  $\Lambda: W_F \times \SL_2(\C) \times
\SL_2(\C) \rightarrow {}^L \sG$, one constructs a tempered $L$-parameter  $\Delta(\Lambda): W_F \times \SL_2(\C) \rightarrow {}^L \sG$, using the diagonal map $\SL_2(\C) \stackrel{\Delta}{\rightarrow} \SL_2(\C) \times
\SL_2(\C)$.
According to a well-known theorem of Moeglin, the possible tempered parts of an $A$-packet for an $A$-parameter
$\Lambda: W_F \times \SL_2(\C) \times \SL_2(\C) \rightarrow {}^L \sG$, are contained in the $L$-packet corresponding
to the tempered $L$-parameter $\Delta(\Lambda)$,
for $G$ a classical groups over $p$-adic fields.

The analogy between $T(\Lambda)$  and  $\Delta(\Lambda)$ motivates us to ask if an analogue of
the above theorem of Moeglin regarding the possible tempered part of an $A$-packet for classical groups over $\R$ holds.
We would like to say that  for an $A$-parameter $\Lambda: \SL_2(\C) \times W_\R \rightarrow {}^L \sG$, tempered representations
of $\sG(\R)$ -- if any -- in this $A$-packet have $L$-parameter given by $T(\Lambda)$.
Theorem \ref{coh-para} proves this in the
affirmative for cohomological $A$-packets (assuming that Adams-Johnson packets are $A$-packets).
If $\sG(\R)$ is either a complex group, or is a unitary group $\U(p,q)$,
all the Levi subgroups $\sL(\R)$ are connected, and therefore infinitesimal character determines tempered $L$-packets for these groups, and hence for these groups too this question has an affirmative answer.
  \end{remark}

In the following theorem which we take up in the last section of this paper,
we prove that the sum of the ranks of cohomology groups  in an $A$-packet 
on any real group (and any infinitesimal character)
is independent of the $A$-packet under consideration, and can be explicitly calculated. This result when summed over all pure inner forms has a particularly nice form for which we refer to Theorem \ref{serre}.

\begin{thm}\label{packetsum}
  Let ${G}=\sG(\R)$ be the  group of real points of a connected reductive real algebraic group $\sG$, $\theta$ a Cartan involution on $G$ with corresponding maximal compact subgroup $K=G^\theta$, and $G^u$ the compact real form of $G$.  Let $\sT^c$ be a $\theta$-stable fundamental torus in $\sG$ and write $\sT^c=\sT\,\mathsf{A}$ with $\sT$ maximally anisotropic. 
Let $E\cong E^\theta$ be a finite-dimensional algebraic representation of $\sG(\C)$. Then for a nonempty Adams-Johnson packet $\Pi$ of representations with $(\g,K)$-cohomology  with respect to $E$, the sum
$$
\sum_{\pi \in \Pi} \sum_{i\geq 0} 
\dim_\C H^i(\g,K, \pi \otimes E) 
$$
is independent of the packet, and equals  
$$
\sum_i \dim_\C H^i(\g,K,\C) = \sum_i \dim_\C H^i(G^u/K) = 
2^d\left |\frac{W(\sG,\sT^c)^\theta}{W(G,T^c)}\right|
$$
where $d = \dim \mathsf{A}= {\rm rank} (\sG) - {\rm rank}(K)$ is the  discrete series defect, $W(\sG,\sT^c)^\theta=\{w\in W(\sG,\sT^c): w\theta=\theta w\}$, and $W(G,T^c)$ is the subgroup of elements with a representative in $G$ (equivalently, it is the image of $W(K,T)$ in $W(\sG,\sT^c)$). 
\end{thm}

We now briefly describe the contents   of the paper.
\vspace{1mm}

Section \ref{prelim} fixes some notation used in the paper, and discusses some preliminaries needed for this work. In particular,  we recall how the infinitesimal character of a $(\g,K)$-module can be read off from
its Langlands parameter.

In section \ref{FD}, we discuss how a mapping of $L$-groups
$\phi: {}^L\sG_1 \rightarrow {}^L\sG_2$ 
that takes a regular unipotent element in $\wG_1$ to a regular unipotent element in $\wG_2$ takes a finite dimensional representation
of $\sG_1(\C)$ naturally to a finite dimensional representation of $\sG_2(\C)$ which is how the coefficient systems are to be used
in cohomological transfer of representations from  $\sG_1(\R)$ to  $\sG_2(\R)$ throughout the paper.

In section \ref{desiderata}, we discuss how cohomological representations and their parameters are related for groups which differ only through their centers.

In section \ref{unique} and \ref{parameters},
we prove the  existence (for each self-associate parabolic subgroup in ${}^L\sG$) and uniqueness (up to twists by characters of $\sG(\R)$) of $A$-parameters
  $\sigma: W_\R \times \SL_2(\C) \rightarrow {}^L \sG$
with  infinitesimal character that of a fixed finite dimensional representation of $\sG(\C)$.
Theorem \ref{cohomo} is the central result of this paper proved in section \ref{parameters}. Section \ref{aux} has some preparatory material for \S \ref{unique} and \S \ref{parameters}.

In section \ref{AJpack}, we recall the fundamental theorem of Adams-Johnson \cite{AJ},
  parametrizing cohomological representations  of $\sG(\R)$  via cohomological $A$-parameters $\sigma: W_\R \times \SL_2(\C) \rightarrow {}^L\sG$. 
  As the theorem of Adams-Johnson plays an important role for us, we have given a sketch of their argument based on the work of Vogan-Zuckerman \cite{VZ}. Section \ref{AJexam} discusses several examples to illustrate the structure of
  $AJ$ packets.

    Given a complete understanding of cohomological parameters in section \ref{parameters},
       section \ref{cohom-para} proves Theorem \ref{coh-para}.

Section \ref{Cgroup}  on complex groups recalls the
well-known classification of cohomological representations due to Enright although we believe our
formulation of his results is  more precise.   An important purpose of the  section is to also
discuss why the condition on {\it self-associate} parabolic does not show up from the point of view of
treating a complex group as a real group.

In section \ref{gln}, we discuss cohomological parameters for $\GL_n(\R)$, $\GL_n(\C)$, and $\U_{p,q}(\R)$.
Section  \ref{more-exam}  again has a few explicit low dimensional example of groups.

In section \ref{coho},  we prove that the sum of dimensions of the $(\g,K)$-cohomology groups over representations in an Adams-Johnson packet 
on any real group (and any infinitesimal character)
is independent of the Adams-Johnson packet under consideration, which can be explicitly calculated.
This result could be useful to decide when certain members of an Adams-Johnson packet constitute the full Adams-Johnson packet! This result when added over all pure inner forms has a particularly nice form.  
The paper ends with refining the numerical statement of Theorem~\ref{packetsum}
to a statement about (mirror) Hodge polynomials (Theorem~\ref{mirror})  when $G/K$ has an invariant complex structure.

   \section{Notation and Preliminaries} \label{prelim}
   
   In this section we fix some notation to be used in the paper, and recall some preliminary material necessary for what
   we do in this paper.

Recall that $W_{\R}= \C^\times \cdot \langle j \rangle$ 
with $j^2 =-1, jzj^{-1} = \bar{z}$ for $z \in \C^\times$.  One has $W_\R^{\rm ab}$, the maximal abelian quotient
of $W_\R$, isomorphic to $\R^\times$, in which the map from $W_\R$ to $\R^\times$, 
when restricted to $\C^\times \subset W_\R$, is just the norm mapping from $\C^\times$ to $\R^\times$ 
(given by $z\in \C^\times \rightarrow z\bar{z} \in \R^\times$). Thus,
characters  $W_\R \rightarrow \C^\times$ can be identified with 
characters of $\R^\times$, in particular, the character $x\rightarrow |x|$ on $\R^\times$ is a character on $W_\R$ which is denoted by $\nu$, and we let  $\nu$ also denote its restriction to  $\C^\times \subset W_\R$ ($\nu(z)=z\bar{z}$); also, we will use
$\nu$ to denote the character of $\GL_n(\R)$, $\GL_n(\C)$ defined by  $\nu(g) = \nu(\det g)$.

Let $\omega_\R$ denote the character of $\R^\times$ of order $2$. Since $\R^\times$ is a quotient of $W_\R$, 
it defines a character of $W_\R$ of order 2  which will also be denoted by
$\omega_\R$.

Note that the characters of $\R^\times$ are of the form $\omega_\R^{\{0,1\}}(t) |t|^s$ where $s \in \C$.   On the other hand, characters of $\C^\times$
can be written as,
$$z \rightarrow z^\mu \cdot \bar{z}^{\nu} = 
z^{\mu -\nu}\cdot  (z \bar{z})^\nu, \quad \mu,\nu \in \C, \mu -\nu \in \Z. $$

Any irreducible representation of $W_\R$ is of dimension $\leq 2$. It is an important fact that for an irreducible two dimensional representation $\sigma$ of $W_\R$, 
$$\sigma \otimes \omega_\R \cong \sigma.$$

Let  $\sigma_d$ be the 2 dimensional irreducible representation of $W_\R$ given by
${\rm Ind}_{\C^\times}^{ W_\R} (z/\bar{z})^{d/2}$ 
where  $(z/\bar{z})^{1/2}$ is the character $ z = r e^{i \theta}\rightarrow e^{i \theta}$. The representation $\sigma_d$ corresponds to  the discrete series 
representation $\pi_{d+1}$, $d\geq 1$, of $\GL_2(\R)$ with lowest weight $d+1$ which has trivial central character restricted to $\R^+ \subset \R^\times$. The representations $\sigma_d$ are self-dual representations of $W_\R$ which are
orthogonal if $d$ is even, and symplectic if $d$ is odd. Also,
 \[
    \det(\sigma_d)=\left\{
                \begin{array}{ll}
                   1 &  {\rm ~if ~} d {\rm ~ is ~ odd,} \\
                  \omega_{\R} & {\rm ~if ~} d {\rm ~ is ~ even.} 
                                  \end{array}
              \right.
  \]

Every irreducible representation $\pi$ of $\GL_n(\R)$ has an associated Langlands parameter $\sigma_\pi$ which is 
an $n$-dimensional semi-simple representation of $W_\R$.	
According to the local Langlands correspondence for $\GL_n(\R)$ (due to Langlands for all real reductive groups),
the association $\pi \rightarrow \sigma_\pi$ is a bijective correspondence between irreducible representations of $\GL_n(\R)$ (continuous representations up to, say, infinitesimal equivalence, i.e., 
one in which two continuous representations of $\GL_n(\R)$ are identified if their $(\g,K)$-modules are isomorphic)  and  
$n$-dimensional semi-simple representations of $W_\R$. For a character $\chi$ of $\R^\times$,  we can twist
a representation $\pi$  of $\GL_n(\R)$ by $\chi$, denoted by $\pi \otimes \chi$ and defined by $ (\pi \otimes \chi)(g)
= \pi(g) \chi(\det g)$. The local Langlands correspondence for $\GL_n(\R)$ is equivariant under this twisting, i.e.,
$ \sigma_{\pi \otimes \chi} = \sigma_\pi \otimes \chi$. One often writes $\pi \otimes \chi$ as $\pi\cdot \chi$.

The notion of an $A$-parameter is more relevant to the study of cohomological representations.  Recall that $A$-parameters for a general reductive group $\mathsf{G}$ over $\R$, are admissible homomorphisms of 
$W_\R \times \SL_2(\C)$ 
into the $L$-group ${}^L\sG=\wG\cdot W_\R$ of $\mathsf{G}$ whose image in $\wG$ are bounded on $W_\R$.
There is a natural map from $W_\R$ to $W_\R \times \SL_2(\C)$ which when projected to $\SL_2(\C)$ is the homomorphism into the group of diagonal
matrices given by the character $(\nu^{1/2},\nu^{-1/2})$. This gives a natural map from the set of $A$-parameters to the set of
$L$-parameters, which is known to be injective, cf. Proposition 1.3.1 of \cite{Ar1}.

For cohomological representations of $\GL_n(\R)$, a particular unitary representation called the Speh representation
--- introduced and studied by Speh in \cite{Sp} --- plays a very important role;
it is defined for a pair of integers $(d,m)$.
If $[m]$ is the unique irreducible representation of $\SL_2(\C)$ of dimension $m$,
the representation of $\GL_{2m}(\R)$ associated with the $A$-parameter
$$\sigma_d \otimes [m],$$
is an
irreducible unitary representation $\pi_d[m]$ of $\GL_{2m}(\R)$.

One way to define $\pi_d[m]$ is to
say that it is the Langlands quotient representation for $\GL_{2m}(\R)$ obtained from the essentially tempered
representation $$ \pi_d\cdot \nu^{(m-1)/2} \times \cdots \times  \pi_d\cdot \nu^{-(m-1)/2},$$
of the Levi subgroup $M = \GL_2(\R) \times \cdots \times \GL_2(\R)$ (product of $m$ factors), where $\pi_d$ is
the discrete series representation of $\GL_2(\R)/\R^+$ with lowest weight $(d+1)$, and of Langlands parameter $\sigma_d$.
The representations $\pi_d[m]$ with $d\geq  m$ are cohomological, a theorem due to Speh in \cite{Sp} for $d=m$.

We next review infinitesimal character which has meaning both for $(\g,K)$-modules as well as for $L$- and  $A$-parameters. Its importance for cohomological representations, in particular for the present work, cannot be overemphasized.

Recall that by the well-known Harish-Chandra homomorphism theorem, the center $\Zk(\g)$ of the enveloping algebra $\Uk(\g)$ of $\g$
is isomorphic to the Weyl group invariants in the enveloping algebra of a maximal torus $\t$ inside $\g$. This center  $\Zk(\g)$ operates on any irreducible
$(\g,K)$-module $\pi$ by a character $\chi_\pi$, called the infinitesimal character of the representation $\pi$.
By the Harish-Chandra homomorphism theorem, we thus get a character $\chi_\pi: \U(\t)^W \rightarrow \C$, an algebra homomorphism,
which is equivalent to a linear form
$\ell_\pi: \t \rightarrow \C$
which is well-defined up to the action of $W$ on $\t$.
If $\pi$ is a finite dimensional highest weight module of $\g$ of highest weight $\lambda: \t \rightarrow \C$
(for a fixed Borel subalgebra $\b$ containing $\t$), then
the infinitesimal character of $\pi$ is $\lambda +\rho$ where $\rho$ is half the sum of positive roots for $(\b,\t)$. 

The infinitesimal character of $\pi$ which is a linear form  $\ell_\pi: \t \rightarrow \C$, can also be considered as
a linear
map $\hat{\ell}_\pi: \C \rightarrow \hat{\t}$,
where $\hat{\t}$ is the Lie algebra of the maximal torus in the dual group associated
to the real reductive group $\sG(\R)$, both $\ell_\pi$ and $\hat{\ell}_\pi$  
are to be considered  up to the action of $W$ (on $\t$ and $\hat{\t}$ respectively).

There is also a  notion of infinitesimal character  associated with
$L$- and  $A$-parameters. The basic reason being that all the representations in an $L$-packet, or an Adams-Johnson packet,
have the same infinitesimal character.

The infinitesimal character of an $A$-parameter is defined to be  the infinitesimal
character of the associated  $L$-parameter, thus it is enough to review  the  notion of the infinitesimal character
of an $L$-parameter $\sigma: W_\R \rightarrow {}^L \sG$.  The infinitesimal character of $\sigma$ depends
only on its restriction to $\C^\times$. Assume without loss of generality that $\sigma(\C^\times) \subset \wT(\C)$, and $\sigma|_{\C^\times}$ is given by:
$$\sigma(z) = z^{\lambda} \bar{z}^\mu = (z\bar{z}) ^{(\lambda+\mu)/2} (z/\bar{z}) ^{(\lambda-\mu)/2} , {\rm ~~ where~~} \lambda , \mu \in X_\star(\wT)\otimes \C  {\rm ~~with~~} \lambda- \mu \in X_\star(\wT), $$
where in more detail, one defines $z^\lambda \in \wT(\C)$ for  $z\in \C^\times,$ and
$ \lambda \in X_\star(\wT)\otimes \C$ say for $\lambda = \alpha \otimes a$
by $$z^{ \alpha \otimes a} = \alpha(z^a),$$
where $z^a$ is  meaningful for either $z \in \R^{>0}$, $a$ an arbitrary complex number, or $z$ arbitrary, and $a$ an integer.

With this notation,  one defines the  infinitesimal character of the parameter $\sigma$ to be  $ \chi_\sigma =  \lambda \in X_\star(\wT)\otimes \C$. The importance
of this definition stems from the following folklore lemma, for
which not having found any proof in the literature, we provide one.

\begin{lemma} \label{inf}
  Suppose 
  $\sigma_\pi: W_\R \rightarrow {}^L \sG$ is the $L$-parameter of an irreducible admissible $(\g,K)$-module $\pi$.
  If $\sigma_\pi$ restricted to $\C^\times$ is given by
$$\sigma_\pi(z) = z^{\lambda} \bar{z}^\mu, {\rm ~~ where~~} \lambda , \mu \in X_\star(\wT)\otimes \C  {\rm ~~with~~} \lambda- \mu \in X_\star(\wT), $$    
  then the  infinitesimal character of $\pi$ is given by $  \lambda \in X_\star(\wT)\otimes \C = \Hom[\C, \hat{\t}]$, where $\hat{\t}$
  is the Lie algebra of $\wT$.
  \end{lemma}
\begin{proof}
  Let us begin by observing that the lemma is correct for essentially
  discrete series representations of any reductive group $\sM$ since for an essentially  discrete
  series representation with infinitesimal character $\lambda+ \rho$, the associated parameter restricted to $\C^\times$ is given by:
  $$z\longrightarrow (z/\bar{z})^{\lambda+ \rho}.$$
  (We will not ask why this is the parameter for discrete series as it is written in the Scriptures \cite{La}!) The lemma follows
  from this,  
  since one goes from discrete series to tempered by parabolic induction, and then from tempered to general representations
  by the Langlands quotient theorem. For parameters, the way we have defined infinitesimal character,
  it is as if nothing has happened in both the steps of going from a discrete series representation to a finite direct
  sum of tempered representations, and from a tempered representation
  to a general representation: it is the same parameter being considered inside $\wG$ instead of $\wM$ under the natural inclusion
  of $\wM$ inside $\wG$. The same is true for infinitesimal
  character under parabolic induction of $(\g,K)$-modules: ${\rm Ind}_P^{\sG(\R)} (\pi)$ has the same infinitesimal character
  as a representation of $\sG(\R)$
  as the representation $\pi$ of the Levi subgroup $\sM(\R)$ associated with the parabolic $P$, noting that both
  $\sG(\C)$ and $\sM(\C)$
  share a common maximal torus $\sT(\C)$, and that for both the groups $\sG(\R)$ and $\sM(\R)$, infinitesimal
  character is a homomorphism  $\ell: \t \rightarrow \C$
  on the Lie algebra $\t$ of $\sT(\C)$.
    \end{proof}

As a practice on using infinitesimal characters, we prove the following lemma which will come useful later.

\begin{lemma} \label{equal}
Let
$\sigma: \C^\times \times \SL_2(\C) \rightarrow \wG$
be an $A$-parameter with real infinitesimal character.
  If $\sigma$ restricted to $\C^\times$ is given by
  $$\sigma(z) = z^{\lambda} \bar{z}^\mu, {\rm ~~ where~~} \lambda , \mu \in X_\star(\wT)\otimes \C  {\rm ~~with~~} \lambda- \mu \in X_\star(\wT), $$    then,
  $$\lambda+ \mu = 0,$$
  equivalently, $\sigma: \C^\times \times \SL_2(\C) \rightarrow \wG$ factors through $\sigma: (\C^\times/\R^+) \times \SL_2(\C) \rightarrow \wG$.
  \end{lemma}
\begin{proof}
The parameter $ \sigma: \C^\times \times \SL_2(\C) \rightarrow \wG$ can be assumed to take $\C^\times \times \C^\times$ 
into the maximal torus $\wT$ in $\wG$ used to define the dual group, where the second $\C^\times$ is the diagonal torus in $\SL_2(\C)$
written as $ z = \left ( \begin{array}{cc} z & 0 \\  0 & z^{-1} \end{array}\right )$.

We denote $\sigma$ restricted to $\C^\times \times \C^\times$ as $\sigma(z_1,z_2) = \sigma_1(z_1) \sigma_2(z_2)$ where $\sigma_2: \C^\times \rightarrow \wT$
is an algebraic homomorphism 
$\sigma_2(z_2) = z_2^\alpha$ for some $\alpha \in X_{\star}(\wT)$.

The homomorphism $\sigma_1: \C^\times \rightarrow \wT$
has the form:
 $$z \longrightarrow z^\lambda \bar{z}^\mu = (z\bar{z})^\lambda {\bar z}^{(\mu -\lambda)}, \hspace{1cm} 
 \lambda,\mu \in X_{\star}(\wT) \otimes \C \hspace{.5cm} {\rm with} \hspace{.5cm}  
 \lambda -\mu \in X_{\star}(\wT).$$
 
 By the definition of an $A$-parameter, $\sigma_1: \C^\times \rightarrow \wT$ is unitary, i.e., $|\sigma_1(z)|=1$,
 which when applied to the definition of $\sigma_1(z) = z^\lambda \bar{z}^\mu$, it follows that:
 \[ \lambda+ \mu \in \sqrt{-1} X_{\star}(\wT) \otimes \R. \tag {1} \]

 We calculate the infinitesimal character of $\sigma$ by using the underlying $L$-parameter.
  For this,  consider the map 
 
$$ \begin{array}{ccccc}\C^\times & \rightarrow &  \C^\times  \times \SL_2(\C) & \rightarrow & \wG \\
   z & \rightarrow & z \times \left ( \begin{array}{cc} (z\bar{z})^{1/2} & 0 \\  0 & (z\bar{z})^{-1/2} \end{array}\right  ) & \rightarrow & \wG,
  \end{array} $$
  which is $$\sigma(z,(z\bar{z})^{1/2}) = 
 \sigma_1(z) \sigma_2((z\bar{z})^{1/2}) = z^\lambda {\bar z}^{\mu} (z\bar{z})^{\alpha/2}.$$
 Thus by Lemma \ref{inf}, the infinitesimal character of $\sigma$ is
 \[z \rightarrow z^\lambda z^{\alpha/2}  .  \tag{2} \]
 Since the infinitesimal character of  $\sigma$ is given to be real,
 we must have  $$(\lambda + \alpha/2)  \in X_{\star}(\wT) \otimes \R,$$
 in particular, since $\alpha \in X_{\star}(\wT)$,
 \[\lambda  \in X_{\star}(\wT) \otimes \R. \tag{3} \]

 Since $ \lambda -\mu \in X_{\star}(\wT),$ this together with $(1)$ and $(3)$ proves the lemma.  \end{proof}

This paper makes considerable use of the principal $\SL_2(\C)$ in a complex reductive group $\sH(\C) $ (which usually arise
as the dual group of a real reductive group). The principal $\SL_2(\C)$ inside $\sH(\C)$ has the property that a
non-trivial unipotent element in $\SL_2(\C)$ goes to a regular unipotent element in $\sH(\C)$. Existence and uniqueness
(up to conjugacy by $\sH(\C)$) of the principal $\SL_2(\C)$ is a basic property of these principal $\SL_2(\C)$ inside
$\sH(\C)$ (which may not inject $\SL_2(\C)$ inside $\sH(\C)$ but may have the central $\pm 1$ as the kernel).

The following
proposition summarizes the properties of the principal $\SL_2(\C)$ that we will need.

\begin{prop} \label{J-M}
  Let ${}^L \sG = \wG \cdot \Gal(\C/\R) $ be (the Galois form) of the  $L$-group of a reductive group $\sG$ over $\R$,
  which comes equipped with a pair
  $(\wB,\wT)$ consisting of  a Borel subgroup $\wB$ containing
  a maximal torus $\wT$ together with a pinning on $\wB$ left invariant by $\Gal(\C/\R)$.
  Then there exists a unique homomorphism $ \Lambda: \SL_2(\C) \times W_\R \rightarrow {}^L \sG$
  which restricted to $\SL_2(\C)$, denoted $\Lambda_{\wG}$,  is a principal $\SL_2(\C)$ in $\wG$ and takes the standard
  pinning on $\SL_2(\C)$ to the fixed pinning on $\wG$, and  is the
  natural map  from $W_\R$ to $\Gal(\C/\R)$.
    Identifying the diagonal torus of $\SL_2(\C)$ with $\C^\times$, in which the diagonal matrix in $\SL_2(\C)$
  with entries $(z,z^{-1})$ goes to $z \in \C^\times$, the resulting homomorphism from $\C^\times$ to $\wT$ is the co-character
  $ 2 \check{\rho}_{\wG},$ which is the sum of positive coroots $\C^\times \rightarrow \wT$.

  Let $s  =  \left ( \begin{array}{cc} 0 & 1 \\  -1 & 0 \end{array}\right ) \in \SL_2(\C)$,
  and 
  $\varpi_{\wG}= \omega_{\wG} \cdot j = \Lambda( s \times j) \in{}^L \sG   $ where $\omega_{\wG} \in \wG$ represents
  the longest element in the
  Weyl group of $\wG$ defined using $(\wB,\wT)$. Then  conjugation by $\varpi_{\wG}$
  preserves $\wT$,  takes $\wB$ to the opposite Borel subgroup
  $\wB^-$,  and has the property that \[\varpi_{\wG}^2= \Lambda( s \times j)^2 = \Lambda(-1,-1) = \epsilon, \]
    where  $\epsilon=
  2 \check{\rho}_{\wG}(-1)$ is  an
  element in the center of $\wG$,  with $\epsilon^2=1$.    \end{prop}
\begin{proof}
  All of this proposition is of course standard and elementary.
  We only comment that the assertion on
  restriction of the principal $\SL_2(\C)$
  to the diagonal torus, a homomorphism from $\C^\times \rightarrow \wT$,  being the co-character
  $ 2 \check{\rho}_{\wG}$ (sum of positive coroots $\C^\times \rightarrow \wG$), 
is a reflection of 
  the well-known identity:
\[2 \langle \rho_{\sG}, \alpha^\vee \rangle = 1,\]
  for all positive simple roots $\alpha$ in $\sG$.

  For the fact that $\epsilon= 2 \check{\rho}_{\wG}(-1)$ is an
  element in the center of $\wG$ with $\epsilon^2=1$, we refer to \cite{GR}, equation (65), from where we have also borrowed
  the notation $\epsilon$ for  the element $2 \check{\rho}_{\wG}(-1)$ in the center of $\wG$. (The work \cite{GR} uses
this element for $\sG$!) \end{proof}

\begin{remark}
  In this paper we will have many occasions to use $ \omega_{\wG}$, the longest element in the Weyl group,
  as an element in $\wG$. The above proposition defines it uniquely once we are given a pinning
  on $\wG$ which dual groups are supposed to come equipped with.
  \end{remark}

  \section{Functoriality of finite dimensional representations} \label{FD}
 
The condition on a homomorphism $\phi: {}^L\sG_1 \rightarrow {}^L\sG_2$ 
that it takes a regular unipotent element in ${\wG}_1$ to a regular unipotent element in $\wG_2$
was suggested by the well-known observation that a principal $\SL_2(\C)$ inside
the $L$-group of a split group, treated as an $A$-parameter,  corresponds to the trivial representation of the group, therefore the condition that $\phi: {}^L\sG_1 \rightarrow {}^L\sG_2$ 
takes a regular unipotent element in $\wG_1$ to a regular unipotent element in $\wG_2$ implies that the trivial representation of $\sG_1(\R)$ goes to the trivial
representation of $\sG_2(\R)$. 

In this section we check that once $\phi: {}^L\sG_1 \rightarrow {}^L\sG_2$ 
takes a regular unipotent element in $\wG_1$ to a regular unipotent element in $\wG_2$, then the $L$-parameters of finite dimensional irreducible representations of $\sG_1(\C)$ go to the $L$-parameters of finite
dimensional irreducible representations of $\sG_2(\C)$. This is  essential for us if $\phi$ were to take a  cohomological
representation of $\sG_1(\R)$ to a cohomological representation of $\sG_2(\R)$ since cohomological
representations have an underlying coefficient system, so the coefficient system needs to be related too.

Let's begin with a homomorphism of the triple $\phi: (\wG_1, \wB_1,\wT_1) \rightarrow (\wG_2,\wB_2,\wT_2)$ 
that  takes a regular unipotent element in ${\wG}_1$ to a regular unipotent element in $\wG_2$.
In particular, $\phi$ gives rise to a mapping of co-character groups $\phi_\star: X_\star(\wT_1) \rightarrow X_\star(\wT_2)$.
Let $\check{\rho}_{\wG_1}$ be half the sum of positive coroots for $(\wG_1,\wB_1,\wT_1)$ which is an element
in $ X_\star(\wT_1) \otimes \Q$, and similarly
let $\check{\rho}_{\wG_2}$ be half the sum of positive coroots for $(\wG_2,\wB_2,\wT_2)$ which is an element
in $ X_\star(\wT_2) \otimes \Q$.

\begin{lemma} \label{rho} Let $\phi: (\wG_1, \wB_1,\wT_1) \rightarrow (\wG_2,\wB_2,\wT_2)$ 
  be a homomorphism that takes a regular unipotent element in ${\wG}_1$ to a regular unipotent element in $\wG_2$.
    Then for the mapping of co-character groups $\phi_\star: X_\star(\wT_1) \otimes \Q \rightarrow X_\star(\wT_2) \otimes \Q$,
  $$\phi_\star(\check{\rho}_{\wG_1}) =  \check{\rho}_{\wG_2}.$$
  \end{lemma}
\begin{proof}
Since the only semisimple element in the centralizer of a regular unipotent element in any reductive group
  are the central elements, the hypothesis in the lemma implies that
  the center of ${\wG}_1$ goes to the center of ${\wG}_2$, therefore $\phi$ induces a map $\bar{\phi}$ on the adjoint groups
  $\bar{\phi}: (\wG_1/\wZ_1, \wB_1/\wZ_1,\wT_1/\wZ_1) \rightarrow (\wG_2/\wZ_2,\wB_2/\wZ_2,\wT_2/\wZ_2)$. 

  Observe that  for the  homomorphism of the co-character groups $\iota_\star: X_\star(\wT_1) \rightarrow X_\star(\wT_1/\wZ_1)$
  induced by the natural surjection
  $\iota: \wT_1 \rightarrow \wT_1/\wZ_1$, we have
\[\iota_\star(\check{\rho}_{\wG_1}) =  \check{\rho},\]
where   $\check{\rho}$ denotes half the sum of positive coroots of $\wG_1/\wZ_1$.

It suffices then to prove the lemma assuming that $\wG_1$ is an adjoint group, hence a product of simple groups. Since $\check{\rho}_{\wG}$, for $\wG$  a product of groups,  is the direct sum of the individual components, it suffices to assume that 
$\phi: \wG_1 \rightarrow \wG_2$ is injective, and therefore  $\phi_\star: X_\star(\wT_1) \rightarrow X_\star(\wT_2)$ is injective too.
 Because of this injectivity, it suffices to prove the lemma in the special case of
 $\wG_1 = \SL_2(\C)$ mapping to $\wG_2$ through a principal
 $\SL_2(\C)$. In this case, the lemma is a consequence of Proposition \ref{J-M} on noting that
 for $\SL_2(\C)$, $2 \check{\rho}$ is the standard co-character $z\rightarrow (z,z^{-1})$, which by
 Proposition \ref{J-M} goes to $2  \check{\rho}_{\wG_2}$ under the principal $\SL_2(\C)$ inside $\wG_2$.
 \end{proof}

A finite dimensional irreducible
representation of $\sG_1(\C)$
with highest weight $\lambda_1$ (treated here as an element of $ X_\star(\wT_1)$)
has an L-parameter 
given by the  map: $W_\C=\C^\times \rightarrow \wT_1$
which is $z\rightarrow z^{\lambda_1+  \check{\rho}_{\wG_1}}\bar{z}^{ \check{\rho}_{\wG_1}}$. (This follows by realizing a finite
dimensional representation of $\sG_1(\C)$
with highest weight $\lambda_1$ as a Langlands quotient of the un-normalized
principal series representation induced from the
character $t\rightarrow t^{\lambda_1}$ on the maximal torus $\sT_1(\C)\subset B_1(\C)$, and noting that
the modulus character for $\sG_1(\C)/B_1(\C)$ is ${\rho}_{\G_1}(t) \cdot \overline{ {\rho}_{\G_1}(t)}$.) 
By lemma \ref{rho} we get the following.

  \begin{prop} \label{hw}
Let $\phi: (\wG_1, \wB_1,\wT_1) \rightarrow (\wG_2,\wB_2,\wT_2)$ 
  be a homomorphism that takes a regular unipotent element in ${\wG}_1$ to a regular unipotent element in $\wG_2$.
    Then  $\phi: \wG_1 \rightarrow \wG_2$ 
    takes the Langlands parameter (i.e., the associated map $W_\C=\C^\times \rightarrow {\wG_1}$)
    of a finite dimensional  irreducible
representation of $\sG_1(\C)$
with highest weight $\lambda_1 \in  X^\star(\sT_1) = X_\star(\wT_1)$ to 
the Langlands parameter of the finite dimensional  irreducible
representation of $\sG_2(\C)$
with highest weight $\lambda_2=\phi_\star(\lambda_1) $ where $\phi_\star: X_\star(\wT_1) \rightarrow X_\star(\wT_2)$ is
  induced from the mapping of tori $\wT_1 \stackrel{\phi}\rightarrow \wT_2$.
\end{prop}

\begin{example} For the mapping of the dual groups $\Sp_{2n}(\C) \subset \GL_{2n}(\C)$, the corresponding mapping between finite dimensional highest weight modules, takes irreducible highest weight module
  $V_{\underline{\lambda_1}}$ of $\SO(2n+1)(\C)$
  where $$\underline{\lambda_1} = \{\lambda_1 \geq  \lambda_2 \geq  \cdots \geq  \lambda_n \geq 0\},$$
  to the irreducible highest weight module $V_{\underline{\lambda_2}}$ of $\GL_{2n}(\C)$:
  $$\underline{\lambda_2} = \{\lambda_1 \geq  \lambda_2 \geq  \cdots \geq  \lambda_n  \geq -\lambda_n \geq \cdots \geq -\lambda_1 \}.$$
  \end{example}

\begin{remark}Throughout the paper, when  talking of $\phi: {}^L\sG_1 \rightarrow {}^L\sG_2$ taking cohomological representations of $\sG_1(\R)$ to cohomological representations of $\sG_2(\R)$, we always mean that their coefficient systems are irreducible 
    representations of  $\sG_1(\C)$ and $\sG_2(\C)$
        which are related as in Proposition \ref{hw}.
\end{remark}

  \begin{prop} \label{inf-ch}
    Let
$\phi: {}^L\sG_1 \rightarrow {}^L\sG_2$ 
    be a homomorphism of $L$-groups whose restriction to the dual groups is also denoted by the same symbol
    $\phi: (\wG_1, \wB_1,\wT_1) \rightarrow (\wG_2,\wB_2,\wT_2)$, and which  
  takes a regular unipotent element in ${\wG}_1$ to a regular unipotent element in $\wG_2$.
    Then  $\phi: \wG_1 \rightarrow \wG_2$ 
    takes a Langlands parameter $\lambda: W_\R \rightarrow {}^L{\sG_1}$
    with the infinitesimal character of  a finite dimensional  irreducible
representation of $\sG_1(\C)$
with highest weight $\lambda_1 \in  X^\star(\sT_1) = X_\star(\wT_1)$ to 
the Langlands parameter $\phi\circ \lambda: W_\R \rightarrow {}^L{\sG_2}$
 with the infinitesimal character of  a finite dimensional  irreducible
representation of $\sG_2(\C)$
with highest weight $\lambda_2=\phi_\star(\lambda_1) $ where $\phi_\star: X_\star(\wT_1) \rightarrow X_\star(\wT_2)$ is
induced from the mapping of tori $\wT_1 \stackrel{\phi}\rightarrow \wT_2$. The same is true also of
$A$-parameters $\Lambda: W_\R\times \SL_2(\C) \rightarrow {}^L{\sG_1}$.
  \end{prop}
  \begin{proof}
    By Lemma \ref{inf},
    the restriction of  the  Langlands parameter $\lambda: W_\R \rightarrow {}^L{\sG_1}$
to $\C^\times \subset W_\R$ is of the form  (for some $\mu_1 \in X_\star(\wT_1)$):
\[ z\rightarrow z^{\lambda_1+  \check{\rho}_{\wG_1}}\bar{z}^{ \mu_1},\]
and therefore composing with $\phi: (\wG_1, \wB_1,\wT_1) \rightarrow (\wG_2,\wB_2,\wT_2)$, and
using Lemma \ref{rho} and Proposition \ref{hw}, we find that the restriction of the
Langlands parameter $\phi\circ \lambda: W_\R \rightarrow {}^L{\sG_2}$ to
$\C^\times \subset W_\R$ is of the form:
\[ z\rightarrow z^{\lambda_2+  \check{\rho}_{\wG_2}}\bar{z}^{ \phi_\star(\mu_1)},\]
therefore we are done by Lemma \ref{inf}.   \end{proof}
  
\section{Desiderata on cohomological representations} \label{desiderata}
The aim of this section is to relate cohomological representations of real reductive groups which differ only through their centers; a good example to 
keep in mind are the groups $\GL_n(\R), \PGL_n(\R), \SL_n(\R)$ which differ only through their centers,
and to assert --- what was already known to everyone --- that for 
matters of cohomology, there is no difference
in using one's favorite group in the isogeny class, which for our paper is semi-simple simply connected group,
for example, because then $\sG(\R)$ is connected.

\begin{lemma} Let $G$, a real Lie group, be an almost direct product of Lie subgroups $G_s$ and $Z$, with $Z$  a central subgroup of $G$,
thus with $G_s \cap Z$ a finite central subgroup in $G$. Similarly, let
$K_0$ be an almost direct product of $K_s^0$ and $K^0_Z$ where $K_0$ is the connected component of identity of $K$, a maximal compact subgroup of $G$, and  
$K_s^0, K_Z^0$ are connected components of identity of  the maximal compacts $K \cap G_s$ and $K\cap Z$
inside $G_s$ and $Z$ respectively. Write the corresponding Lie algebras as direct sum $\mathfrak{g} = \mathfrak{g}_s+ 
\mathfrak{z}$ and  $\mathfrak{k} = \mathfrak{k}_s+ \mathfrak{k}_{\mathfrak z}$. 
Then for a representation $\pi$ of $(\mathfrak{g},K)$ on which $\mathfrak{z}$ operates by the same scalar as it does on a finite
dimensional representation $V$ of $G$, 
$$H^*(\mathfrak{g}, \gk, \pi \otimes V^\vee) \cong H^*(\mathfrak{g}_s, \gk_s, \pi \otimes V^\vee) \otimes H^*(\mathfrak{z}, \mathfrak{k}_{\mathfrak z}, \C),$$
and therefore since  $H^0(\mathfrak{z}, \mathfrak{k}_{\mathfrak z}, \C) \not = 0$,
$$ H^*(\mathfrak{g}, K_0, \pi \otimes V^\vee) \not = 0  \iff H^*(\mathfrak{g}_s, \gk_s, \pi \otimes V^\vee) \not = 0.$$
 
\end{lemma}
\begin{proof} 
The isomorphism  $H^*(\mathfrak{g}, \gk, \pi \otimes V^\vee) \cong H^*(\mathfrak{g}_s, \gk_s, \pi \otimes V^\vee) \otimes H^*(\mathfrak{z}, \gk_{\mathfrak z}, \C)$ is a consequence 
of Kunneth theorem, cf. [I, 1.3] of \cite{BW}. 
 Non-vanishing of $H^0(\mathfrak{z}, \gk_{\mathfrak z}, \C),$ follows from the definition of relative Lie algebra cohomology.
\end{proof}

\begin{lemma} 
 Let $f: G_1 \rightarrow G_2$ be a homomorphism of Lie groups with finite kernel and finite co-kernel. Assume 
 that $K_1$ and $K_2$ are maximal compact subgroups of $G_1$ and $G_2$ such that $f(K_1) \subset K_2$. Let $K_{1,0}$ and $K_{2,0}$ be the connected
 components of their identity elements with their complexified Lie algebras $\gk_1$ and $\gk_2$ respectively. Let $\pi_2$ be a  $(\g_2, K_2)$ module, and $V_2$ 
 a finite dimensional representation of $G_2$ with the same central and infinitesimal characters as $\pi_2$.
 Let $\pi_1= f^*(\pi_2)$ be the restriction of $\pi_2$ via $f$, giving rise to a $(\g_1, K_1)$ module 
  and $V_1 = f^*(V_2)$ the corresponding representation of $G_1$. Then, 
 $$ H^*(\mathfrak{g}_2, \gk_2, \pi_2 \otimes V^\vee_2 ) \cong  H^*(\mathfrak{g}_1, \gk_1, \pi_1 \otimes V_1^\vee).$$
 \end{lemma}
\begin{proof} Under the conditions on $f: G_1 \rightarrow G_2$, we have an isomorphism of Lie algebras $\g_1 \cong \g_2$, and an isomorphism
$f:  \gk_1 \rightarrow \gk_2$, and $\pi_2 \otimes V^\vee_2$ is isomorphic to $\pi_1 \otimes V^\vee_1$ so there is nothing to be proved! 
 \end{proof}

 \begin{lemma} Let $\sG$ be a real reductive group with $Z$ its center. Then the action of the adjoint group $(\sG/\mathsf{Z})(\R)$ on $\sG(\R)$ takes irreducible 
 cohomological representations of $\sG(\R)$ to irreducible cohomological representations of $\sG(\R)$.
   \end{lemma}
\begin{proof} Let $K$ be a maximal compact subgroup of $\sG(\R)$, and $K^{\rm ad}$ a maximal compact subgroup of $(\sG/\mathsf{Z})(\R)$ containing the image of 
$K$ under the natural map $\sG(\R) \rightarrow (\sG/\mathsf{Z})(\R)$. 
Note  that the action of the adjoint group $(\sG/\mathsf{Z})(\R)$ on isomorphism classes of representations of 
$\sG(\R)$, which is through the group of connected components of the adjoint group $(\sG/\mathsf{Z})(\R)$,
can be realized through the action of $K^{\rm ad}$ on   $\sG(\R)$ since a maximal compact subgroup intersects all connected components, and hence $K^{\rm ad}$ surjects on
the group of connected components of the adjoint group $(\sG/\mathsf{Z})(\R)$. The proof now follows since $K^{\rm ad}$ normalizes $K_0$, and therefore takes 
cohomological representations of $(\g,\gk_0)$ to cohomological representations of $(\g,\gk_0)$.   
 \end{proof}

The previous three lemmas are combined to give the following proposition.
\begin{prop} \label{1}
 Let $f: \sG_1 \rightarrow \sG_2$ be a homomorphism of reductive algebraic groups over $\R$ which we assume is an isogeny restricted to their derived subgroups. Then
 an irreducible representation $\pi_2$ of $\sG_2(\R)$ is cohomological for coefficients $V_2$, an irreducible
 representation of $\sG_2(\C)$,  if and only if its restriction $\pi_1 = f^*(\pi_2)$ to  $\sG_1(\R)$ is cohomological
 for the coefficients $V_1=f^*(V_2)$ of $\sG_1(\C)$. Further,  
 if $\pi_1= f^*(\pi_2) $ is a  cohomological representation of $\sG_1(\R)$, any irreducible summand of $\pi_1$ is cohomological too.  
 
\end{prop}

\begin{proof}We only need to prove the assertion at the end that if $\pi_1$ as in the proposition is cohomological, any irreducible summand of $\pi_1$ is cohomological too.
Since $f(\sG_1(\R))$ is a normal subgroup of $\sG_2(\R)$, all the irreducible constituents of $\pi_1$ are conjugate under the action of $\sG_2(\R)$ --- this is the so-called Clifford theory. 
Now the proposition follows from the previous lemma.
\end{proof}

The following proposition proves that the {\it restriction problem} has the natural answer for $L$-parameters. This proposition is originally in Langlands's article \cite{La},
item
$(iv)$ on page 23, section 3. 

\begin{prop} \label{2}
 Let $f: \sG_1 \rightarrow \sG_2$ be a homomorphism of reductive algebraic groups over $\R$ which we assume is an isogeny restricted to their derived subgroups. Then
 for an irreducible representation $\pi_2$ of $\sG_2(\R)$ its restriction $\pi_1 = f^*(\pi_2)$ to  $\sG_1(\R)$ is a finite direct sum of irreducible 
 representations of $\sG_1(\R)$ which all belong to one $L$-packet whose $L$-parameter is obtained from that of $\pi_2$ through the homomorphism
 of $L$-groups: $f: {}^L {\sG}_2 \rightarrow {}^L {\sG}_1$. 
\end{prop}

\begin{proof} It is easy to see that the restriction $\pi_1 = f^*(\pi_2)$ of $\pi_2$ to  $\sG_1(\R)$ is a finite direct sum of irreducible representations of $\G_1(\R)$. 
Next we note  that the assertions in the proposition hold when $\pi_2$ is a discrete series representation 
in which case 
its restriction $\pi_1 = f^*(\pi_2)$ to  $\sG_1(\R)$ is a finite direct sum of irreducible discrete series representation, and all the
constituents of $\pi_1$ are in the same $L$-packet, and the one dictated  
by the morphism of  $L$-groups: $f: {}^L {\sG}_2 \rightarrow {}^L {\sG}_1$. This is clear since $\pi_1$ is clearly a discrete  series representation of $\sG_1(\R)$, and has an infinitesimal
character, therefore by the well-known result  on discrete series representations for $\sG_1(\R)$: any two discrete series representations
of $\sG_1(\R)$ with the same central and infinitesimal character are in the same $L$-packet, 
$\pi_1$, which
clearly has a central character, is a finite sum of discrete series representations of $\sG_1(\R)$ belonging to one $L$-packet. Once the result 
 is there for discrete series, it is clear too for tempered representations, and then for general representations it follows by the Langlands quotient formalism, see 
 section 4 of
\cite{AP}     for details. 
\end{proof}

\begin{remark}
  It would be interesting to formulate a suitable analogue of Proposition \ref{2} replacing $L$-parameter by $A$-parameter, and
  $L$-packets by Adams-Johnson packets, but we have not done this.
\end{remark}

\section{Some auxiliary results} \label{aux}

The proof of the following proposition has benefited from the comments of the referee.

\begin{prop}\label{Aux1} If  $\Lambda: \C^\times \times \SL_2(\C)  \rightarrow  \wG$ is an $A$-parameter with the infinitesimal
  character that of a finite dimensional representation $F_{\lambda_0}$ of $\sG(\C)$ of highest dominant
  weight $\lambda_0 \in X^\star(\sT)=X_\star(\wT)$, then $\Lambda(\SL_2(\C))$ is a principal
  $\SL_2(\C)$ in the centralizer of $\Lambda(\C^\times)$ in $\wG$, which we can assume by conjugation by $\wG$ to be a
  standard Levi subgroup $\wL$ in $\wG$, such that $\Lambda$ restricted to $\C^\times \times \C^\times$ lands inside $\wT$, and is
  given by
    \[ \lambda_1: z \rightarrow (z/\bar{z})^{\lambda +\check{\rho}_{\wG}-\check{\rho}_{\wL}}, \tag{a}\]
    \[ \lambda_2: z \rightarrow z^{2 \check{\rho}_{\wL}}, \tag{b}\]
    on the two $\C^\times$ inside $\C^\times \times \C^\times$.
\end{prop}
\begin{proof}
  Assume that  $\Lambda$ restricted to $ \C^\times \times \C^\times$ lands inside the maximal torus $\wT$ in $\wG$ used to define the dual group, where the second $\C^\times$ is the diagonal torus in $\SL_2(\C)$
written as $ z = \left ( \begin{array}{cc} z & 0 \\  0 & z^{-1} \end{array}\right )$.
We denote $\Lambda$ restricted to $\C^\times \times \C^\times$ as $\Lambda(z_1,z_2) = \lambda_1(z_1) \lambda_2(z_2)$ where $\lambda_2: \C^\times \rightarrow \wT$
is an algebraic homomorphism 
$\lambda_2(z_2) = z_2^{\alpha_2}$ for some $\alpha_2 \in X_{\star}(\wT)$,
and the homomorphism $\lambda_1: \C^\times \rightarrow \wT$
has the form:
\[z \longrightarrow (z/\bar{z})^{\alpha_1}
\hspace{.5cm}    {\rm with} \hspace{.5cm}  
 2\alpha_1 \in X_{\star}(\wT). \tag{1}\] 

 Thus the restriction of the $L$-parameter associated with the $A$-parameter $\Lambda$ to $W_\C=\C^\times$,
   is given by 
 \[z \rightarrow (z/\bar{z})^{ \alpha_1} (z\bar{z})^{\alpha_2/2}  .  \tag{2} \]

 Assume (taking a conjugate of $\Lambda$ if necessary) that the  cocharacter  $2\alpha_1 + \alpha_2 \in X_\star(\wT)$
 is dominant integral.

 By Lemma \ref{inf} on the infinitesimal character, since
the infinitesimal character of $\Lambda$ is that of  the 
representation $F$ of $\sG(\C)$ of highest weight $\lambda_0 \in X^\star(T) = X_\star(\wT)$, we must have:

\[w(2\alpha_1+ \alpha_2)  = 2 \lambda_0 + 2 \check{\rho}_{\wG} , \]
for some element $w$ in the Weyl group of $\wT$, but as both $2\alpha_1+ \alpha_2$  and $ 2 \lambda_0 + 2 \check{\rho}_{\wG}$ are
dominant, we have the equality:
\[2\alpha_1+ \alpha_2  = 2 \lambda_0 + 2 \check{\rho}_{\wG}  .  \tag{3} \]

Let $\wL \subset \wG$  be the centralizer of   $\lambda_1(\C^\times)$ which since $\alpha_1$ is not
known to be dominant, is not necessarily a standard Levi subgroup of $\wG$.
   We have by (1) above, \[\langle \alpha_1, \beta \rangle = 0 \tag{4}\]
   for all  roots $\beta $ of  $\wL$
   where $\langle -, - \rangle$ denotes the canonical
   pairing $X_\star (\wT) \times X^\star (\wT) \rightarrow \Z$. From equation (3) above,
   for all the roots $\beta$ of $\wT$ in $\wG$, we have
   \[\langle 2\alpha_1+ \alpha_2, \beta \rangle   = \langle 2 \lambda_0 + 2 \check{\rho}_{\wG}, \beta \rangle. \tag{5}\]
   
   As $\langle   \check{\rho}_{\wG}, \alpha \rangle =1$ for all simple roots $\alpha$ of $\wT$ in $\wB$,
   $\langle 2 \check{\rho}_{\wG}, \beta \rangle \in 2\Z$ for all roots $\beta$ of $\wT$ in $\wG$,
   so  for a root $\beta$ of $\wT$ in $\wL$, by equations (4) and (5) above,
   $ \langle \alpha_2, \beta \rangle   = \langle 2 \lambda_0 + 2 \check{\rho}_{\wG}, \beta \rangle \in 2 \Z$,
   but is not 0 for any root $\beta$ of $\wT$ in  $\wL$. 
   Therefore by Lemma \ref{princ-sl2}, $\Lambda(\SL_2(\C))$ is a principal $\SL_2(\C)$ inside $\wL$.

   To complete the proof of the proposition, we now prove that 
   $\wL$ is
   a standard Levi subgroup of $\wG$. For this, consider the set of roots $\beta$ of $\wT$ in $\wL$ such that
   $ \langle \alpha_2, \beta \rangle > 0$, defining a positive system of roots inside $\wL$ with simple roots
    \[\{\beta| \langle \alpha_2, \beta \rangle =2 \} = \{\beta | \langle \lambda_0 +  \check{\rho}_{\wG}, \beta \rangle =1 \}.\]

    Clearly, the roots $\beta$ of $\wG$
    with  $\langle \lambda_0 +  \check{\rho}_{\wG}, \beta \rangle =1$ must be simple roots of $\wG$ (and orthogonal to $\lambda_0$), proving that  $\wL$ is
   a standard Levi subgroup of $\wG$. 

   Since $\Lambda$ restricted to $\SL_2(\C)$ is a principal $\SL_2(\C)$ inside $\wL$, $\lambda_2$ as given in equation (b) is a
   consequence of Proposition \ref{J-M}, from which equation (a) clearly follows. 
   (It also is easy to check that in fact $\alpha_1$ is dominant, but we do not need it now.)
\end{proof}

\begin{lemma} \label{princ-sl2}
  Let $(\wL,\wT)$ be a complex reductive group  $\wL$ containing a maximal torus $\wT$. Suppose $\iota: \SL_2(\C) \rightarrow \wL$ is a homomorphism taking the diagonal torus of $\SL_2(\C)$
  to $\wT$. 
  Let $\alpha$ be the co-character in $X_\star(\wT)$ so defined
  using the restriction of $\iota$ to the diagonal torus in $\SL_2(\C)$.
Then if for all  roots $\beta$ of $\wT$ in $\wL$, one has
   $\langle \alpha, \beta \rangle\in 2\Z $ but never 0,  
  then $\iota: \SL_2(\C) \rightarrow \wL$
  is a principal $\SL_2(\C)$ in $\wL$.
  \end{lemma}
\begin{proof}
  This is  Proposition 7, \S11, chapter VIII of \cite{Bo}, stated there for Lie algebras, easily seen to be equivalent
  to the assertion here for algebraic groups. (Readers will recognize this lemma as a special case of the Dynkin-Kostant classification
  of nilpotent conjugacy classes in terms of the weighted Dynkin diagram, in this case with weight 2 at each node,
  which is the regular nilpotent conjugacy class.)
    \end{proof}

\begin{lemma} \label{jacobson}
  Let $(\wL, \wB,\wT)$ be a triple consisting of a reductive  group $\wL$ over $\C$,
  a Borel subgroup $\wB$, and a maximal
  torus $\wT \subset \wB$.
  Let $\phi$ be an involutive automorphism of $(\wL, \wT)$ which takes $\wB$ to $\wB^-$.
  Then there is a
  homomorphism $\iota: \SL_2(\C) \rightarrow \wL$
taking the diagonal torus of $\SL_2(\C)$
  to $\wT$, and the group of upper triangular matrices in $\SL_2(\C)$
  to $\wB$  whose image is a principal $\SL_2(\C)$ inside $\wL$, such that:
  \[ \phi(\iota ( g))= \iota(sgs^{-1}), {\rm~where~}  s =  \left ( \begin{array}{cc} 0 & 1 \\  -1 & 0 \end{array}\right ).\]

  Further, any two such homomorphisms $\iota: \SL_2(\C) \rightarrow \wL$
  are conjugate by an element  $t \in \wT$
  such that $t^{-1}\phi(t) \in Z(\wL)$.
\end{lemma}
\begin{proof}
  The proof we offer is modeled on an argument of Kostant for constructing a principal $\SL_2(\C)$ inside the reductive group $\wL$.

  We will make an
  argument on the Lie algebra, constructing a Lie triple $\{H, X, Y\}$ in the Lie algebra of $\wL$ with
  $H$ belonging to the Lie algebra of $\wT$, $X$ (resp. $Y$) a regular nilpotent element in the Lie algebra of the unipotent radical of $\wB$
  (resp. the opposite Borel subgroup $\wB^-$),  and with: 
  $$[H,X] =2X, [H,Y] = -2Y, [X,Y] = H, {\rm ~~ and~~} \phi(X)=Y.$$
  Such a Lie triple  $\{H, X, Y\}$ is  invariant under $\phi$,  on which $\phi$ operates by conjugation by
  $ s   $.

  Since $\phi$   takes $\wB$ to $\wB^-$, there is an involution on the set $S$, again to be denoted by $\phi$,  of the set simple roots of $\wT$ in $\wL$,  such that
  $\phi$ takes $\alpha$ root space to $-\phi(\alpha)$ root space.

  For each $\alpha \in S$, let $\{H_\alpha, X_\alpha, Y_\alpha \}$  with
  $$[H_\alpha,X_\alpha] =2X_\alpha, [H_\alpha,Y_\alpha] = -2Y_\alpha, [X_\alpha,Y_\alpha] = H_\alpha,$$
  be a Lie triple with $X_\alpha, Y_\alpha$ belonging  to $\alpha$ and $-\alpha$ root spaces. If $\phi(\alpha) = -\alpha$, by scaling $X_\alpha$
  appropriately, one can assume that $Y_\alpha = \phi(X_\alpha)$. On the other hand, if $\phi(\alpha) = \beta \not = \alpha$,
  we can assume that
  $$\phi(X_\alpha) = Y_\beta, \phi( Y_\alpha) = X_\beta, \phi( H_\alpha) = -H_\beta.$$

  With these choices of  $\{H_\alpha, X_\alpha, Y_\alpha\}$ for all simple roots $\alpha \in S$,
  define
\begin{eqnarray*}
  X_S & = & \sum _{\alpha \in S} t_\alpha X_\alpha, \\
  Y_S & = & \phi(X_S)= \sum _{\alpha \in S} t_{\phi(\alpha)} Y_\alpha \\
  H_S & = & [X_S,Y_S]
\end{eqnarray*}
where $t_\alpha$ are certain complex numbers to be fixed later.

Let $H = \sum_{\alpha \in S}  a_\alpha H_\alpha$ be the unique vector in the span of $H_\alpha, \alpha \in S$ such that
$[H,X_\alpha] = 2 X_\alpha$ for all $\alpha \in S$.

If we can  construct $\{ t_\alpha \in \C^\times | \alpha \in S \}$
such that $H = H_S=[X_S,Y_S]$, then $\{H_S, X_S, Y_S\}$ will be a Lie triple
invariant under $\phi$.
Observe that (because $[X_\alpha,Y_\beta] = 0$ if $\alpha \not = \beta$)
$$ H_S=[X_S,Y_S] = \sum_{\alpha \in S}  t_\alpha t_{\phi(\alpha)} H_\alpha,$$
therefore $$H_S=H  = \sum_{\alpha \in S}  a_\alpha H_\alpha,$$
if and only if we can choose $\{ t_\alpha| \alpha \in S \}$ such that
$$a_\alpha =  t_\alpha t_{\phi(\alpha)},$$
for all $\alpha \in S$. Clearly, this is so if and only if for any $\alpha$ with $\phi(\alpha) \not = \alpha$,
in the expression for $H = \sum_{\alpha \in S}  a_\alpha H_\alpha$, $a_\alpha = a_{\phi(\alpha)}$. For proving this, observe that
 $H = \sum_{\alpha \in S}  a_\alpha H_\alpha$ is the unique vector in the span of $H_\alpha, \alpha \in S$ such that
$[H,X_\alpha] = 2 X_\alpha$ for all $\alpha \in S$, equivalently, $[H,Y_\alpha] = -2 Y_\alpha$ for all $\alpha \in S$.
Since $\phi(Y_\alpha) = X_{\phi(\alpha)}$ for all $\alpha \in S$, by applying $\phi$ to the system of equations:
$[H,Y_\alpha] = -2 Y_\alpha$, we find that $[\phi(H), X_\alpha] = -2 X_\alpha$, and therefore by the uniqueness of
$H$, $\phi(H) = -H$. This combined with what we noted earlier that $\phi( H_\alpha) = -H_{\phi(\alpha)}$, 
proves that $a_\alpha = a_{\phi(\alpha)}$ if  $\phi(\alpha) \not = \alpha$
allowing us to construct $\{ t_\alpha \in \C^\times| \alpha \in S \}$
such that $a_\alpha =  t_\alpha t_{\phi(\alpha)},$
constructing the desired homomorphism $\iota: \SL_2(\C) \rightarrow \wL$.

The uniqueness follows directly by
 using  $\phi(\iota ( g))= \iota(sgs^{-1}),$ and 
noting that any two principal $\SL_2(\C)$ in $\wL$ are conjugate in $\wL$, and in this case
any two homomorphism $\iota_1,\iota_2: \SL_2(\C) \rightarrow \wL$ are conjugate by $\wT$.   \end{proof}

The following corollary will play an important role in the proof of Theorem \ref{cohomo}.
 
\begin{cor} \label{rem}
  Let ${}^L \sG = \wG \cdot \Gal(\C/\R) $ be (the Galois form) of the  $L$-group of a reductive group $\sG$ over $\R$,
  which comes equipped with a pair
  $(\wB,\wT)$ consisting of  a Borel subgroup $\wB$ containing
  a maximal torus $\wT$ together with a pinning on $\wB$ left invariant by $\Gal(\C/\R)$. Let $\wL$ be a standard Levi subgroup of $\wG$ left invariant by $\omega_{\wG} \cdot j \in \wG \cdot j $
and $\phi$ the automorphism  of $\wL$ which is
  conjugation by $\omega_{\wG} \cdot j$.
  In the notation of Lemma \ref{jacobson}, let $\omega_{\wL}= \iota(s)$,  then
  the automorphism $ \phi$ of $\wL$  fixes the element
  $\omega_{\wL} \in \wL$.
  Thus as elements of ${}^LG$, we have
  \begin{enumerate}
  \item commuting elements $j$ and $\omega_{\wG}$
    with $j^2=1$, $\omega_{\wG}^2= \epsilon = 2\hat{\rho}_{\wG}(-1) \in Z(\wG)$;
  \item  $\omega_{\wL}= \iota(s)$ with $(\omega_{\wG} \cdot j) \omega_{\wL}(\omega_{\wG} \cdot j)^{-1}= \omega_{\wL}$;
  \item $\omega_{\wL}^2=  2\hat{\rho}_{\wL}(-1) \in Z(\wL)$;
  \item The automorphism  $ \omega_{\wL}\omega_{\wG} \cdot j$
    is identity on the principal $\SL_2(\C)$ given by
    $\iota: \SL_2(\C) \rightarrow \wL$,
    therefore the automorphism  $ \omega_{\wL}\omega_{\wG} \cdot j$ of $\wL$
    can be considered as a pinned involution of $\wL$ with the pinning coming from $\iota: \SL_2(\C) \rightarrow \wL$.
  \end{enumerate}
\end{cor}
 \begin{prop} \label{central} An $A$-parameter $\Lambda: \C^\times \times \SL_2(\C) \rightarrow \wG$ with an integral 
 infinitesimal  character -- in particular, a cohomological parameter --
 sends the element $(-1,-1) \in \C^\times \times \SL_2(\C)$
to  the element $\epsilon = 2\hat{\rho}_{\wG}(-1)$, a central element in $Z(\wG)$
 introduced in Proposition \ref{J-M}, in particular to the trivial element in $\wG$ if $\sG$ is  simply connected. 
   \end{prop}
\begin{proof}
The parameter $ \Lambda: \C^\times \times \SL_2(\C) \rightarrow \wG$ can be assumed to take $\C^\times \times \C^\times$ 
into the maximal torus $\wT$ in $\wG$ used to define the dual group, where the second $\C^\times$ is the diagonal torus in $\SL_2(\C)$
written as $ z = \left ( \begin{array}{cc} z & 0 \\  0 & z^{-1} \end{array}\right )$.

We denote $\Lambda$ restricted to $\C^\times \times \C^\times$ as $\Lambda(z_1,z_2) = \lambda_1(z_1) \lambda_2(z_2)$ where $\lambda_2: \C^\times \rightarrow \wT$
is an algebraic homomorphism 
$\lambda_2(z_2) = z_2^{\alpha_2}$ for some $\alpha_2 \in X_{\star}(\wT)$.

The homomorphism $\lambda_1: \C^\times \rightarrow \wT$ has the form (being trivial on $\R^+\subset \C^\times$ by Lemma \ref{equal}):
 $$z \longrightarrow (z/\bar{z})^{\alpha_1},  \hspace{1cm}
 2 \alpha_1 \in X_{\star}(\wT) .$$
 
 We calculate the infinitesimal character of the $A$-parameter $\Lambda: \C^\times \times \SL_2(\C) \rightarrow \wG$ which by
 definition is the infinitesimal character of the underlying $L$-parameter.
  For this,  consider the map,
 
$$ \begin{array}{ccccc}\C^\times & \rightarrow &  \C^\times  \times \SL_2(\C) & \rightarrow & \wG \\
   z & \rightarrow & z \times \left ( \begin{array}{cc} (z\bar{z})^{1/2} & 0 \\  0 & (z\bar{z})^{-1/2} \end{array}\right  ) & \rightarrow & \wT,
  \end{array} $$
  which is,  $$\Lambda(z,(z\bar{z})^{1/2}) = 
 \lambda_1(z) \lambda_2((z\bar{z})^{1/2}) = (z/\bar{z})^{\alpha_1}(z\bar{z})^{\alpha_2/2} .$$
 Thus by Lemma \ref{inf}, the infinitesimal character of  $\Lambda$ is
 $${\alpha_1+\alpha_2/2} .$$
 Since the infinitesimal character of 
 $\Lambda$ is given to be integral, we must have
 \[(\alpha_1 + \alpha_2/2)  \in X_{\star}(\wT) + \hat{\rho}_{\wG}. \tag{1} \]
 
(If half the sum of positive roots  for $\sG$
 is a weight for the maximal torus  of $\sG$, such as for simply connected groups $\sG$,
 then $ \hat{\rho}_{\wG}$ belongs to
 $X_{\star}(\wT  )$, and the coset space   $X_{\star}(\wT) + \hat{\rho}_{\wG}$ reduces to  $X_{\star}(\wT)$.)
 
 Returning to the calculation of $\Lambda(-1,-1)$, note that
 $$\Lambda(z_1,z_2) = \lambda_1(z_1) \lambda_2(z_2) = (z_1/\bar{z}_1)^{\alpha_1}  z_2^{\alpha_2},$$
  hence,
  $$\Lambda(-1,-1) =(-1)^{ -2\alpha_1} (-1)^{\alpha_2} = (-1)^{-2\alpha_1+\alpha_2},$$
  with $ 2\alpha_1, \alpha_2 \in X_{\star}(\wT).$
  
  By the integrality of the infinitesimal character, as noted before in equation (1),
    we have:
  
  $$-2\alpha_1+\alpha_2 =  -2(\alpha_1 + \alpha_2/2) + 2 \alpha_2  \in 2X_\star(\wT) +2
  \hat{\rho}_{\wG} \subset X_\star(\wT)  . $$
  Therefore, $$ (-1)^{-2\alpha_1+\alpha_2} = (-1)^{ 2 \hat{\rho}_{\wG}  } = \epsilon,$$
  proving the proposition.
 \end{proof}

 \begin{remark} If the $A$-parameter of a cohomological representation of $\GL_n(\R)$ is of the form:
  $$\Lambda = \sum_d \sigma_d \otimes [m_d + 1] \oplus \omega_\R^{\{0,1\}}[a],$$
 then it can be seen that the action of $(1,-1)\in \C^\times \times \SL_2(\C)$ 
 on the subspace of $\Lambda$ corresponding to $\sigma_d \otimes [m_d + 1]$
 is by $1$ or $-1$ depending on the 
 parity of $m_d$ in the above sum which can have both parities, whereas  $(-1,-1)\in \C^\times \times \SL_2(\C)$ 
 acts by the same sign on all summands, i.e., is a scalar in conformity with the above proposition.  In this example,
 all the irreducible summands in the decomposition of $\Lambda$ as 
 $\Lambda = \sum_d \sigma_d \otimes [m_d + 1] \oplus \omega_\R^{\{0,1\}}[a],$ are selfdual, and the fact that 
 $(-1,-1)\in \C^\times \times \SL_2(\C)$ 
 acts by the same sign on all summands is equivalent to say that all the summands are selfdual of the same parity (orthogonal 
 if $n$ is odd and
 symplectic if $n$ is even).
 This remark also proves an assertion made in Theorem \ref{main3} that 
 ``a cohomological parameter for $\PGL_n(\R)$ is automatically a
 parameter (hence cohomological) for the symplectic group if $n$ is odd, and is a parameter
 for the odd orthogonal group if $n$ is even''.

 \end{remark}

 Let us first note the following lemma whose straightforward proof will be left to the reader. This lemma
 will come useful when we try to extend $A$-parameters from $W_\C \times \SL_2(\C)$ to $W_\R \times \SL_2(\C)$.

 \begin{lemma} \label{trivial} Let ${\mathcal N}$ be a normal subgroup of a group ${\mathcal G}$ of index 2, with $\iota$ an element in
   ${\mathcal G}$ but not in ${\mathcal N}$ with
   $z= \iota^2 \in {\mathcal N}$. Suppose  we are given a homomorphism $\lambda: {\mathcal N} \rightarrow {\mathcal G}_1$, where ${\mathcal G}_1$ is any group,
then for $\lambda$ to extend to a homomorphism $\tilde{\lambda}: {\mathcal G} \rightarrow {\mathcal G}_1$ taking $\iota$ to an element 
 $\iota' \in {\mathcal G}_1$, it is necessary and sufficient to have:
 \begin{enumerate}
  \item $(\iota')^2 = \lambda(z)$.
  \item $\lambda(\iota n \iota^{-1}) = \iota' \lambda(n) \iota'^{-1}.$
   \end{enumerate}
\end{lemma}

 \section{Uniqueness} \label{unique}

 In this section we will prove that the non-uniqueness in extending
 $A$-parameters $W_\C \times \SL_2(\C)\rightarrow \wG$ to $W_\R \times \SL_2(\C)\rightarrow{}^L\sG$ is limited to the group
$H^1(W_\R/\C^\times, \wZ)$ where
  $\wZ$ is the center of $\wG$, which is trivial for example if $\sG$ is simply connected. 
 
Here is a  result where $\mathsf{G}$ being simply connected comes useful. The proof of this lemma is due to the referee, which is more direct than that of the authors
who proved the lemma using the dual groups and a better-known result  
  for a simply connected group.

\begin{lemma} \label{connected} Let $\wG$ be any reductive group with connected center, and let $\wL$ be any Levi subgroup in $\wG$. Then the center $Z(\wL)$ of $\wL$ is a connected algebraic group.  
 
 \end{lemma}
\begin{proof} Observe that if $Z(\wG)$ is the center of a connected reductive group such as $\wG$, with a Borel subgroup $\wB$ containing a maximal torus $\wT$ and $\Delta  \subset X^\star(\wT)$,  the set of simple roots of $\wT$ on $\wB$, then $Z(\wG)$  is connected if and only if 
  \[X^\star(Z(\wG)) = X^\star(\wT)/\langle \Delta \rangle,\]
  is torsion free, where $\langle \Delta \rangle$ is the subgroup of  $X^\star(\wT)$ generated by  $\Delta  \subset X^\star(\wT)$.

  Similarly, if $\wL=\wL_S$ is defined as the centralizer of the connected component of identity of the simple roots in $S\subset \Delta$, then   \[X^\star(Z(\wL)) = X^\star(\wT)/\langle S \rangle,\]
  which is clearly torsion free as it sits in the following exact sequence for which both the end terms are torsion free:
  \[ 0 \rightarrow \langle \Delta \rangle /\langle S \rangle \rightarrow
  X^\star(\wT)/\langle S \rangle \rightarrow  X^\star(\wT)/\langle \Delta \rangle \rightarrow 0. \] 
  \end{proof}

The following proposition will play a key role in proving uniqueness of $A$-parameters with the infinitesimal character
of a finite dimensional representation.

 \begin{prop} \label{uniqueness} Let $\wL$ be a standard Levi subgroup in
${}^L\sG = \wG \rtimes W_\R$ with $\wG$ an adjoint semi-simple group.
   Let $\varpi_{\wG} = \omega_{\wG} \cdot j$ be the element in $\wG \cdot j$ constructed in Proposition 1.
   Assume $\varpi_{\wG}$  normalizes $\wL$, therefore $\langle \varpi_{\wG} \rangle$ acts on  $Z(\wL)$.
   Then, 
 \[H^1(\langle \varpi_{\wG} \rangle, Z(\wL)) = 0.\] 
\end{prop}
 \begin{proof}
    Let $\iota_{\wG}$ be the {\it opposition involution} (a pinned automorphism of $(\wG,\wB,\wT)$) which has the property 
  \[\omega_{\wG} \circ \iota_{\wG}(t) = t^{-1}, \forall t \in \wT,\tag{1}\]
where  $\omega_{\wG}$ is  the longest element in the Weyl group of
$({\wG}, \wB,\wT)$.

For the Levi subgroup $\wL_S$  corresponding  to a subset $S \subset \Delta$ of simple roots, the center of $\wL_S$, call it $Z(\wL_S)$, is known to be connected 
 by Lemma \ref{connected}, and is equal to $\{ t \in \wT| \alpha(t) = 1,  \forall \alpha \in S\}$. Therefore, 
the natural map \[Z(\wL_S) \longrightarrow \prod \Gm, \] product taken over
 simple roots in $\Delta-S$, is an isomorphism.  The Levi subgroup  $\wL_S$  is invariant under $\varpi_{\wG}$ if and only if the 
 set $S$, and hence $\Delta-S$, is invariant under $\iota_{\wG}\cdot j$.
 Therefore $Z(\wL_S) \cong \prod \Gm$, a certain  product indexed by the  simple roots in 
 $\Delta-S$ on which $\iota_{\wG} \cdot j$ 
 operates as it operates on $\Delta-S$. Thus $Z(\wL_S)$ is a product of certain copies of  $\C^\times \times \C^\times$ on which $\iota_{\wG}\cdot j$ operates by 
 permuting the two co-ordinates $(z_1,z_2) \rightarrow (z_2,z_1)$, and certain copies  of $\C^\times$ on which $\iota_{\wG} \cdot j $ operates by identity. By the property (1) recalled above, we have:
 $\omega_{\wG} \circ j (t) = (\iota_{\wG} \circ j)(t^{-1}), \forall t \in \wT.$ Hence the action of $\varpi_{\wG}$ on $Z(\wL_S)$ 
  is a product of certain copies of $\C^\times \times \C^\times$ on which $\varpi_{\wG}$ operates by $(z_1,z_2) \rightarrow (z_2^{-1},z_1^{-1})$, and certain copies of $\C^\times$ on 
  which $\varpi_{\wG}$ operates by $t\rightarrow t^{-1}$. 
  For both the actions, $H^1(\langle \varpi_{\wG} \rangle,\C^\times \times \C^\times) = 0$ and $H^1(\langle \varpi_{\wG} \rangle,\C^\times) = 0$, completing
  the proof of the proposition.
 \end{proof}

The following standard lemma, cf. Serre \cite{Se} I \S5.8(a),  answers  the question about 
possible lifts of an $A$-parameter from $\C^\times \times \SL_2(\C)$ to $W_\R \times \SL_2(\C)$.
Note that for an $A$-parameter $\sigma: W_\R \times \SL_2(\C) \rightarrow {}^L\sG$ 
whose restriction to $\C^\times \times \SL_2(\C)$ is given by ${\sigma}_0: \C^\times \times \SL_2(\C) \rightarrow {}^L\sG$  gives rise to an
action of $W_\R \times \SL_2(\C)$ and $\C^\times \times \SL_2(\C)$ on $\wG$. We denote $\wG$ with these
actions of $W_\R \times \SL_2(\C)$ and $\C^\times \times \SL_2(\C)$ by $\wG_\sigma$ and $\wG_{\sigma_0}$,
and will  consider the corresponding cohomology sets:  
$H^1(W_\R \times \SL_2(\C), \widehat{\G}_\sigma) $ and $H^1(\C^\times \times \SL_2(\C),\widehat{\G}_{\sigma_0})$.

\begin{lemma}\label{fibers}
  For a fixed $A$-parameter $\sigma: W_\R \times \SL_2(\C) \rightarrow {}^L\sG$
  with its restriction to $\C^\times \times \SL_2(\C)$ given by ${\sigma}_0: \C^\times \times \SL_2(\C) \rightarrow {}^L\sG$,
we have an exact sequence of pointed sets, with $\alpha$ an injective map of sets, and where $\sigma_0$ can be considered as the base-point 
of the pointed set $H^1_{\sigma_0}(\C^\times \times \SL_2(\C),\widehat{\G}) ^{{\rm Gal}(\C/\R) }$:

\begin{eqnarray*}
1 &\longrightarrow& H^1({\rm Gal}(\C/\R), \widehat{\G}^{\sigma_0(\C^\times \times \SL_2(\C))}
) \stackrel{\alpha}\longrightarrow H^1(W_\R \times \SL_2(\C), \widehat{\G}_\sigma) \\
&\stackrel{\rm res}\longrightarrow& H^1(\C^\times \times \SL_2(\C),\widehat{\G}_{\sigma_0}) ^{{\rm Gal}(\C/\R) }.\end{eqnarray*}
\end{lemma}

In this lemma, $\widehat{\G}^{\sigma_0(\C^\times \times \SL_2(\C))}$ is the center $Z(\wL)$ of $\wL$.

\begin{prop} \label{general}
  For $\sG$ a connected real reductive group, given a cohomological $A$-parameter $\sigma: \C^\times \times \SL_2(\C)  \rightarrow {}^L\sG$,
  the conjugacy classes of its extension to $W_\R \times \SL_2(\C)$
  has a transitive action of $H^1(W_\R/\C^\times, \wZ)$ where
  $\wZ$ is the center of $\wG$.
\end{prop}
\begin{proof}
  By Lemma \ref{fibers}, we know that the set of conjugacy classes of extensions of $\sigma$ to $W_\R \times \SL_2(\C)$ is a principal
  homogeneous space for $H^1(W_\R/\C^\times, Z(\wL))$ where $\wL$ is the Levi subgroup of $\wG$ naturally associated with the
  parameter $\sigma$. By Proposition \ref{uniqueness}, and in the notation of that Proposition,
   $H^1(\langle \varpi_{\wG} \rangle, Z(\wL)/\wZ) = 0$ since $ Z(\wL)/\wZ$ is the
  center of a Levi in the adjoint dual group $\wG/\wZ$. We are now done by the cohomology exact sequence
  corresponding to the short exact sequence of $\langle \varpi_{\wG} \rangle $-groups:
\[1 \rightarrow \wZ
  \rightarrow Z(\wL) \rightarrow Z(\wL)/\wZ \rightarrow 1, \tag{*}\]
  and noting that the action of $ \langle \varpi_{\wG} \rangle $ on  $\wZ$, the center of $\wG$, is the same
  as the action of $W_\R/\C^\times$ on $ \wZ$.
\end{proof}

\begin{remark}\label{twisting} 
  By a result of Langlands, cf. 
  Appendix A by Labesse and Lapid in \cite{LM},
  there is a natural surjective
  map from $H^1(W_\R, \wZ)$ to
  the group of characters of $\sG(\R)$ which is an isomorphism if $\sG(\R)$ is quasi-split. As $\C^\times$ is a divisible group, $\Hom(\C^\times, \wZ(\C))$ is torsion free, and therefore all the torsion elements in
  $H^1(W_\R, \wZ)$ come from $H^1(W_\R/\C^\times, \wZ)$. Thus finite order characters of $\sG(\R)$
  come from  $H^1(W_\R/\C^\times, \wZ)$, in particular if $\sG(\R)$ is quasi-split, then $\sG(\R)$
  is a connected Lie group if and only if  $H^1(W_\R/\C^\times, \wZ)=0$, in which case all the inner forms
  of $\sG(\R)$ are also connected. Using Proposition \ref{uniqueness}  in the cohomology sequence associated with the exact sequence (*) of Proposition \ref{general}, it follows that if the quasi-split form of $\sG(\R)$ is connected, then so is the group $\sL(\R)$
  whose $L$-group (as constructed in Remark \ref{VZcharacter})
is ${}^L \sL = \wL \rtimes \langle \omega_{\wL}\omega_{\wG} \cdot j\rangle$.

Further, using the previous paragraph, one can interpret Proposition \ref{general} as the assertion
that for a general reductive group, cohomological
  $W_\R \times \SL_2(\C)$ parameters which are equal when restricted to $\C^\times \times \SL_2(\C)$ are twists of each other
  by finite order characters of $\sG(\R)$.
  \end{remark}

\vspace{.4cm}
\noindent{\bf Notation:} Given the transitive action of $H^1(W_\R/\C^\times, \wZ)$
on  the conjugacy classes of  extension of $\sigma: \C^\times \times \SL_2(\C)  \rightarrow {}^L\sG$ to $W_\R \times \SL_2(\C)$, let $Q_\sigma$ be the
quotient of $H^1(W_\R/\C^\times, \wZ)$ through which
$H^1(W_\R/\C^\times, \wZ)$ acts (now simply transitively)
on  the conjugacy classes of  extension of $\sigma: \C^\times \times \SL_2(\C)  \rightarrow {}^L\sG$ to $W_\R \times \SL_2(\C)$.
\vspace{.4cm}

\section{Cohomological $A$-parameters} \label{parameters}

\begin{defn} (Cohomological $A$-parameter) \label{cohomologicalCparameter}
  Let $\sG$ be a real reductive group. By a {\it cohomological $A$-parameter} for $\sG$, we will
  mean an $A$-parameter $\tilde{\Lambda}:  W_\R \times \SL_2(\C) \rightarrow {}^L\sG$ such that
  the infinitesimal character of $\tilde{\Lambda}$  is that of a finite dimensional algebraic
  representation of $\sG(\C)$.
\end{defn}

\begin{remark}
  The relevance of this definition to cohomological representations of $\sG(\R)$
  will become clear in the next section when we review the work of Adams-Johnson.   
\end{remark}

\begin{defn} \label{sap} 
(Self-associate parabolic) Let $\sG$ be a real reductive group with ${}^L \sG = \wG\rtimes W_\R$ its $L$-group
  with $W_\R = \C^\times \cdot \langle j \rangle$. The $L$-group ${}^L \sG$  comes equipped with a maximal torus $\wT$ and a Borel subgroup $\wB$, both invariant under $W_\R$.  A standard parabolic $\wP$ in $\wG$ (i.e, containing $\wB$) will be 
called {\it self-associate} in ${}^L\sG$ 
if $\wP$ is conjugate by the element  $ \omega_{\wG} \cdot j \in \wG \cdot j$ (where $\omega_{\wG}$ is the longest element in the Weyl group of $\wT$
for the ordering induced by $\wB$) to its opposite, i.e., to $\wP^-$; equivalently, $\iota_{\wG}\cdot j$
preserves $\wP$ where $\iota_{\wG}$ is the
{\it opposition involution}, a pinned automorphism of $\wG$
preserving $(\wB,\wT)$,  and acting on $\wT$ as $\omega_{\wG}(t^{-1})$. Standard
Levi subgroup of a self-associate standard parabolic of ${}^L\sG$ will be called  a {\it self-associate Levi} of  ${}^L\sG$ 
 \end{defn}

There is another related definition for the coefficient system $V$ considered in this paper, as these
are the only coefficient systems for which there are cohomological unitary representations by
Proposition 6.12, Chapter II of \cite{BW}.

\begin{defn} (Self-associate coefficient system)
  Let $\sG$ be a real reductive group with $\theta$, a Cartan involution on $\sG(\R)$, extended to
  an algebraic automorphism of $\sG(\C)$, also denoted as $\theta$.
  A finite dimensional irreducible representation $V$ of $\sG(\C)$ will be said to be {\it self-associate}
  if $V^\theta \cong V$. \end{defn}

  \begin{remark} \label{dualgroup}
  It is well-known, see for instance section 6 of \cite{AV}, that for a real reductive group
  $\sG(\R)$ with
  a Cartan involution $\theta$ giving rise to a pinned automorphism
  ${}^\vee \theta$ of $\wG$,  we have  ${}^\vee \theta
  = (-\omega_{\wG})\cdot j = \iota_{\wG}\cdot j$ where $\iota_{\wG}$ is the opposition involution.
  Thus the condition on a standard Levi or a standard parabolic of $\wG$ to be self-associate is just being invariant under $ {}^\vee \theta$.
  \end{remark}

 \begin{remark}\label{dscase}
   If $\sG$ is an absolutely  simple group over $\R$ with a discrete series representation for $\sG(\R)$, then all the standard
   parabolics
   in $\wG$ are self-associate as $\sG(\R)$ has a compact Cartan subgroup.
It is easy to see that $\sG(\R)$ has a compact Cartan subgroup
   if and only if $\sG(\R)$ is an inner form of its
   compact form whose $L$-group is defined by the action of $\Gal(\C/\R)$ on $\wG$ by the { opposition involution},
   preserving $(\wB,\wT)$ and acting on $\wT$ as $\omega_{\wG}(t^{-1})$.
   Similarly, if $\sG(\R)$ has a discrete series representation, then 
   all the representations of $\sG(\C)$ are self-associate since $\sG(\R)$ has a compact Cartan subgroup if and only if the Cartan involution on $\sG(\R)$ is an inner automorphism on $\sG(\C)$. Thus the notion
   of self-associate, either for a parabolic in $\wG$, or for a representation of $\sG(\C)$, is
   non-trivial only when $\sG(\R)$ does not have a discrete series representation, thus $\sG(\R)$
   is an inner form of  either a split group of  type $A_n, D_{2n+1}, E_6$, or $\sG(\R)$ is an
   outerform of the split $D_{2n}$.
  For each of these groups except $D_4$, we have ${\Out}(\wG) = \Z/2$, and in these cases a  self-associate 
parabolic in $\wG$ is one which corresponds to a subset $S \subset \Delta$ of simple roots 
invariant under the non-trivial diagram automorphism; for $D_4$, it is the same involution which defines
the outerform.
 \end{remark}

The next proposition justifies the usage of self-associate parabolics and self-associate Levi subgroups in $L$-groups when dealing with cohomological $A$-parameters $\tilde{\Lambda}:  W_\R \times \SL_2(\C) \rightarrow {}^L\sG$.

\begin{prop}
  Suppose  ${}^L\sG = \wG \rtimes W_\R$ is an $L$-group, $\phi: \C^\times/\R^+ \rightarrow \wG$ is a homomorphism,
  defining $\wL$ as the centralizer of  $\phi (\C^\times)$ in $\wG$, and $\wP$ as the subgroup of  $\wG$
  on which the action of   $\C^\times$ is via the characters $(z/\bar{z})^n, n \geq 0$, which we assume after conjugating $\phi$ by an element of $\wG$ are standard Levi and standard parabolic in $\wG$. Similarly define $\wP^-$
  as the subgroup of  $\wG$
  on which the action of   $\C^\times$ is via characters $(z/\bar{z})^n, n \leq 0$.
  Then if there is an element $g_0$ 
  in $\wG \cdot j$ for which  $g_0 \phi (z) g_0^{-1}= \phi(z^{-1})$ for $z \in \C^\times$, 
  $\wL$ must be  a self-associate
  Levi subgroup in $\wG$, in which case,  $\wP$ is automatically a self-associate parabolic.
\end{prop}
\begin{proof}
  Since $g_0 \phi (z  ) g_0^{-1}= \phi(z^{-1})$ for $z \in \C^\times$,  conjugation by $g_0$ preserves
$\wL$ and takes $\wP$ to $\wP^-$. Since  any two
  pairs consisting of a Borel subgroup and a torus contained in it are conjugate in any reductive group, in
  particular for $\wL$, there is an element $l \in \wL$ such that conjugation by
  $l\cdot g_0$ preserves the standard maximal
  torus $\wT$ in $\wL$, and takes the standard Borel subgroup in $\wL$
  to its opposite Borel subgroup in $\wL$.
  Conjugation by $l\cdot g_0$ still takes $\wP$ to $\wP^-$, but further takes $\wT$ to $\wT$ and
  $\wB $ to $\wB^-$, thus up to an element in $\wT$, $ l\cdot g_0 $ must be $\omega_{\wG}\cdot j$. Hence, $(\omega_{\wG}\cdot j)\wP (\omega_{\wG}\cdot j)^{-1}=
  \wP^{-}$,  proving the proposition.
\end{proof}

Here is one of the main theorems of this work  classifying $A$-parameters of $\sG(\R)$ with the infinitesimal character of
a finite dimensional algebraic representation of $\sG(\C)$. 

\begin{thm} \label{cohomo}
  Let $\sG$ be a connected reductive group over $\R$ with $L$-group ${}^L \sG= \wG\cdot W_\R$. 
Given a finite dimensional algebraic
representation $F_\lambda$ of $\sG(\C)$ with highest dominant weight $\lambda \in X^\star(\sT) = X_\star(\wT)$,
there is a bijective correspondence,
  denoted $\wL \longleftrightarrow \Lambda_{\wL}$,  between
  \begin{itemize}
  \item[(1)]   
   standard Levi subgroups $\wL \subset \wG$ 
    for which
    $\langle \lambda, \alpha_i \rangle =0$ for all simple roots $\alpha_i$ of $\wL$,
  \item [(2)] the conjugacy classes of $A$-parameter $\Lambda:  W_\C \times \SL_2(\C) \rightarrow {}^L\sG$
    such that
  the infinitesimal character of $\Lambda$  is that of  $F_\lambda$.
\end{itemize}
  In this bijective correspondence, the restriction of $\Lambda_{\wL}$ to $W_\C=\C^\times \subset W_\R$
lands inside $Z(\wL) \subset \wT$, and 
  is given by:
    \[ \tag{a} \lambda_1: z \rightarrow (z/\bar{z})^{\lambda +\check{\rho}_{\wG}-\check{\rho}_{\wL}}. \]
    The restriction of $\Lambda_{\wL}$
    to $\SL_2(\C)$ lands inside $\wL$ as a principal $\SL_2(\C)$, and the restriction of $\Lambda_{\wL}$ 
    to the diagonal torus $(z,z^{-1})$ in $\SL_2(\C)$, lands inside $\wT$, and is given by
  \[\tag{b} \lambda_2: z \rightarrow z^{2 \check{\rho}_{\wL}}. \]

  If either $\wL$ is not self-associate,
  or $\omega_{\wG}\cdot j (\lambda)\not =-\lambda$,
  the corresponding 
  $A$-parameter $\Lambda:  W_\C \times \SL_2(\C) \rightarrow {}^L\sG$ does not extend
  to $\tilde{\Lambda}:  W_\R \times \SL_2(\C) \rightarrow {}^L\sG$. Therefore, assume now that 
    $\omega_{\wG}\cdot j (\lambda)=-\lambda$. Then
 an $A$-parameter $\Lambda:  W_\C \times \SL_2(\C) \rightarrow {}^L\sG$ as in (2) for  a standard Levi subgroups $\wL \subset \wG$ invariant under $\omega_{\wG} \cdot j \in \wG \cdot j$
  extends to an $A$-parameter  $\tilde{\Lambda}:  W_\R \times \SL_2(\C) \rightarrow {}^L\sG$, which is unique for example when $\sG$ is simply connected.
  In general, an $A$-parameter $\Lambda:  W_\C \times \SL_2(\C) \rightarrow {}^L\sG$ as in (2) for  a standard Levi subgroups $\wL \subset \wG$ invariant under $\omega_{\wG} \cdot j \in \wG \cdot j$
  extends to an $A$-parameter  $\tilde{\Lambda}:  W_\R \times \SL_2(\C) \rightarrow {}^L\sG$, sending $j \in W_\R$ to
  $\omega_{\wL}\omega_{\wG} \cdot j \in \wG \rtimes W_\R$ where  $\omega_{\wL} \in \wL$,
  not uniquely determined by $\wL$, was introduced in Corollary \ref{rem}. The group $Q_\sigma$,
  a finite elementary abelian 2-group, which was introduced at the end of \S\ref{unique}), 
  acts simply transitively
  on the set of such extensions.
\end{thm}
  
\begin{proof}
  We begin by noting that   by Proposition \ref{Aux1},
  given an $A$-parameter $\Lambda: W_\C \times \SL_2(\C) \rightarrow {}^L \sG$ whose infinitesimal
  character is that of $F_\lambda$, then if the centralizer of $\Lambda(W_\C)$ is defined to be $\wL$,
  $\Lambda: \SL_2(\C) \rightarrow \wL$ is a principal $\SL_2(\C)$ in $\wL$, which we can assume
  by conjugating $\Lambda$ by $\wG$ to be a standard Levi subgroup of $\wG$. Furthermore,
   by Proposition \ref{Aux1},
  $\Lambda$ restricted to $W_\C=\C^\times$ is  as in equation (a), 
  and $\Lambda$ restricted to the diagonal of $\SL_2(\C)$  is  as in equation (b).

    Next, we prove the converse, i.e. given a standard Levi subgroups $\wL \subset \wG$
        for which
  $\langle \lambda, \alpha_i \rangle =0$ for all simple roots $\alpha_i$ of $\wL$, we construct the desired
  $A$-parameter $\Lambda_{\wL}:  W_\C \times \SL_2(\C) \rightarrow {}^L\sG$.
For this, define $\lambda_1$ on  $W_\C=\C^\times$ as in equation (a). Observe  that
  \[Z_{\wG}(\lambda_1(\C^\times)) = \wL.\]
  This follows because
  \[ \langle \lambda +\check{\rho}_{\wG}-\check{\rho}_{\wL},
  \alpha_i \rangle = 0
  \hspace{.5 cm} {\rm ~for ~} \alpha_i {\rm ~simple}  \hspace{.5 cm} {\rm ~if~ and~ only~ if~} \alpha_i \in \wL,  \]
  which is a consequence of
  \begin{enumerate}
    \item $\langle \lambda, \alpha_i \rangle =0$ for all simple roots $\alpha_i$ of $\wL$,
    \item
  $\langle \check{\rho}_{\wG}, \alpha_i \rangle = \langle \check{\rho}_{\wL}, \alpha_i \rangle =1$
      for all simple roots $\alpha_i$ of $\wL$,
    \item    $\langle\check{\rho}_{\wG}-\check{\rho}_{\wL}, \alpha_i \rangle > 0$ if
      $\alpha_i$ is a simple root in $\wG$ but   $\alpha_i \not \in \wL$. (To see this, observe that
      since ${\rho}_{\wL}$ written as a sum of positive simple roots does not involve $\alpha_i$, we must have
      $\langle {\rho}_{\wL}, \alpha_i \rangle \leq 0$.)
\end{enumerate}

Fix any $\Lambda_2: \SL_2(\C) \rightarrow \wL$ whose image is a principal $\SL_2(\C)$ in
$\wL$. Then $\Lambda_2$ restricted to the diagonal subgroup $(z,z^{-1})$ inside $\SL_2(\C)$ is $\lambda_2$ as given in equation (b) of the theorem by
Proposition \ref{J-M}.

Using $\lambda_1$ on $W_\C=\C^\times$, and $\Lambda_2$ on $\SL_2(\C)$,  we have constructed an $A$-parameter $\Lambda_{\wL}: \C^\times \times \SL_2(\C) \rightarrow {}^L \sG$ whose
infinitesimal character is that of $F_\lambda$, completing the proof of the first part of the Theorem.

Next we argue that $\Lambda_{\wL}: \C^\times \times \SL_2(\C) \rightarrow {}^L \sG$
can be extended using Lemma \ref{trivial} to
$W_\R \times \SL_2(\C) \rightarrow {}^L \sG$ if $\wL \subset \wG$ is invariant under $\omega_{\wG} \cdot j \in \wG \cdot j$.

 We will now find it convenient to modify  the definition of an $A$-parameter, and not use $A$-parameters
 using the group $ W_\R \times \SL_2(\C)$, but instead  use the group  $A'_\R = \langle  A_\C, j \rangle$ where  $A_\C= \C^\times \times \SL_2(\C)$
 is a normal subgroup of $A'_\R$ of index 2 and the element $ j $ outside $A_\C$
  satisfies  $j^2=(-1, -1) \in \C^\times \times \SL_2(\C)$; further,
   \[j(z,g) j^{-1} = (\bar{z}, sgs^{-1}) ,\]
  where   $s= \left ( \begin{array}{cc} 0 & 1 \\  -1 & 0 \end{array}\right )$.
 
 It is easy to see that   $\psi: W_\R \times \SL_2(\C) \rightarrow  A'_\R$ 
 with $\psi(z,g) =(z, sgs^{-1})$ for $(z,g) \in A_\C$ and $\psi(j) = (1,s)\cdot j$ is an isomorphism of groups.
  Thus, $A$-parameters  using either of
 the two groups  $W_\R \times \SL_2(\C)$, or $A'_\R $,
 are the same, but
 $A'_\R$ will be useful now, as we extend $\Lambda_{\wL}: \C^\times \times \SL_2(\C) \rightarrow {}^L \sG$
 to
$A'_\R  \rightarrow {}^L \sG$ if $\wL \subset \wG$ is invariant under $\omega_{\wG} \cdot j \in \wG \cdot j$.

  Since  we are given that $\wL$ is invariant under conjugation by $\omega_{\wG} \cdot j
  \in \wG \cdot j$,
by Lemma \ref{jacobson}, there exists $\Lambda_2: \SL_2(\C) \rightarrow \wL$ whose image is a principal $\SL_2(\C)$ in
$\wL$ (and invariant under conjugation by $\omega_{\wG} \cdot j \in \wG \cdot j$)
such that $\Lambda_2$ restricted to the diagonal subgroup $(z,z^{-1})$ inside $\SL_2(\C)$ lands inside $\wT$, and is 
$\lambda_2$ as given in equation (b) of the theorem. We change the original $\Lambda_{\wL}: \C^\times \times \SL_2(\C) \rightarrow {}^L \sG$
by a conjugation in $\wL$ so that $\Lambda_{\wL}$ restricted to $\SL_2(\C)$ is $\Lambda_2$.
We now extend $\Lambda_{\wL}: \C^\times \times \SL_2(\C) \rightarrow {}^L \sG$ to $A'_\R$ 
by sending $j \in W_\R$ to $\varpi_{\wG} = \omega_{\wG} \cdot j$, and check now that the conditions of  Lemma \ref{trivial}
are satisfied.

  Condition 1 of Lemma \ref{trivial} (for ${\mathcal G}= A'_\R$ containing ${\mathcal N}=W_\C$ as a subgroup of index 2)
 will be satisfied  because of Proposition \ref{central} and since
 $\varpi_\sG^2= \epsilon $ by Proposition \ref{J-M}. Suffices then to check condition 2 of Lemma \ref{trivial}
 for ${\mathcal N}=\C^\times  \times \SL_2(\C)$. First we check the condition on $\C^\times$ (or rather $\C^\times/\R^+$) on which $j$
 operates by
 $z \rightarrow z^{-1}$, thus to check the condition  2 of Lemma \ref{trivial} on $\C^\times$,  it suffices to check the equivalent condition
 \[\omega_{\wG} \cdot j(\lambda +\check{\rho}_{\wG}-\check{\rho}_{\wL}) = -(\lambda +\check{\rho}_{\wG}-\check{\rho}_{\wL}), \tag{EC}\]
 which follows because
\begin{eqnarray}
  \omega_{\wG} \cdot j(\check{\rho}_{\wG}) & = &  -\check{\rho}_{\wG}, \\
  \omega_{\wG} \cdot j(\check{\rho}_{\wL}) & = &  -\check{\rho}_{\wL},\\
  \omega_{\wG} \cdot j(\lambda) & = &  -\lambda.
\end{eqnarray}
Here, observe that the first equality is obvious, and the second follows because $\wL$ is invariant under 
$ \varpi_{\wG} $. Thus the condition (EC) is satisfied if and only if the condition
(3) is satisfied. This proves one of the assertions in the theorem that if 
$\lambda$ does not satisfy condition (3), then the parameter
${\Lambda}:  W_\C \times \SL_2(\C) \rightarrow {}^L\sG$, cannot be extended to
$\tilde{\Lambda}:  W_\R \times \SL_2(\C) \rightarrow {}^L\sG$. Assume thus that 
$\omega_{\wG} \cdot j(\lambda)  =   -\lambda$,   so
condition 1 of Lemma \ref{trivial} (for ${\mathcal G}= A'_\R$ containing ${\mathcal N}=W_\C$ as a subgroup of index 2)
will be satisfied, and condition 2 of Lemma \ref{trivial} for $\SL_2(\C)$  is part of Lemma \ref{jacobson},
allowing to extend the
parameter $\Lambda_{\wL}: \C^\times \times \SL_2(\C) \rightarrow {}^L \sG$ to $A'_\R$,

Having extended the
parameter $\Lambda_{\wL}: \C^\times \times \SL_2(\C) \rightarrow {}^L \sG$ to $A'_\R$,
we can now use the isomorphism
between $A'_\R$ and $\SL_2(\C) \times W_\R$ 
to extend $\Lambda_{\wL}: \C^\times \times \SL_2(\C) \rightarrow {}^L \sG$ to
$\tilde{\Lambda}_{\wL}: W_\R \times \SL_2(\C) \rightarrow {}^L \sG$, which following the isomorphism 
between $A'_\R$ and $\SL_2(\C) \times W_\R$ discussed earlier,
sends $j \in W_\R$  to the element $\omega_{\wL}\omega_{\wG} \cdot j \in \wG \rtimes W_\R$
where $\omega_{\wL}$ is as in Remark \ref{rem}.
(By Remark,   \ref{rem} $(\omega_{\wL}\omega_{\wG} \cdot j)^2 =  2\hat{\rho}_{\wL}(-1) \cdot 2\hat{\rho}_{\wG}(-1)$,
hence by Proposition \ref{central}, it is clear that 
$j \in W_\R \rightarrow \omega_{\wL}\omega_{\wG} \cdot j \in \wG \rtimes W_\R$ defines an extension
of $\Lambda_{\wL}: \C^\times \times \SL_2(\C) \rightarrow {}^L \sG$ to $W_\R \times \SL_2(\C) \rightarrow {}^L \sG$.)

The simply transitive action of the group $Q_\sigma$ on the set of extensions of the $A$-parameter $\Lambda: \C^\times \times \SL_2(\C) \rightarrow \wG$ constructed using $\lambda_1,\lambda_2$ in (a) and (b) to  $\tilde{\Lambda}: W_\R \times \SL_2(\C) \rightarrow {}^L \sG$ is a consequence of Proposition \ref{general}.
\end{proof}

\begin{remark} \label{VZcharacter}
  In the previous theorem, an $A$-parameter
  $\Lambda:  W_\C \times \SL_2(\C) \rightarrow {}^L\sG$ associated to a
  standard Levi subgroups $\wL \subset \wG$ invariant under $\omega_{\wG} \cdot j \in \wG \cdot j$
  extends to an $A$-parameter  $\tilde{\Lambda}:  W_\R \times \SL_2(\C) \rightarrow {}^L\sG$.
    In this, the image of   $\tilde{\Lambda}$ in
  $ {}^L\sG$ is not necessarily contained in the Levi subgroup ${}^LL$ of $ {}^L\sG$ (even if ${}^LL$ was a Levi subgroup of $ {}^L\sG$).
  By Remark \ref{rem},   $\omega_{\wL}\omega_{\wG} \cdot j$
  considered as an
  automorphism
  of $\wL$ is a pinned automorphism of $\wL$ of order $\leq 2$.
This allows us to  define a reductive group $\sL$ over $\R$ (called the twisted Levi subgroup of $\sG(\R)$, well-defined as an innerclass of groups) by
demanding that its $L$-group be
${}^L \sL = \wL \rtimes \langle \omega_{\wL}\omega_{\wG} \cdot j\rangle$.
 \end{remark}

\begin{example} \label{cg} (Classical groups) Let $\sG$ be any one of the classical groups $\Sp(n), \SO(p,q)$ over $\R$. Recall that if $\sG(\R)$ has a discrete series representation, then all parabolics in $\wG$ are self-associate, therefore the condition of being self-associate on a parabolic of $\wG$ is non-trivial
  only for the groups $\SO(p,q)$ with $p+q$ even and $(-1)^{(p-q)}= (-1)^{(p+q)/2}$
  for which self-associate parabolics  are those which are invariant under the diagram automorphism of the group. Each of these groups $ \Sp(n), \SO(p,q)$ have an
  $L$-group which comes equipped with a natural representation in $\GL_m(\C)$ preserving either symmetric or skew-symmetric bilinear form on $V=\C^m$ with Levi subgroups of the form:

$$\wM = \GL(V_1) \times \GL(V_2) \times \cdots \times \GL(V_k) \times \wG(W),$$  
where $$V= V_1+ V_2 + \cdots + V_k + W + V_k^\vee + \cdots + V_2^\vee +  V_1^\vee,$$
is a decomposition of the space $V$ in isotropic subspaces $V_i$ and $V_i^\vee$ which are in perfect pairing 
under the corresponding bilinear form, and all these are orthogonal to the non-degenerate subspace $W$.
Such a Levi subgroup corresponds to a self-associate parabolic if and only if $\dim (W) \geq 2$
in case the group involved is $\SO(p,q)$ with $p+q$ even and $(-1)^{(p-q)}= (-1)^{(p+q)/2}$.

The corresponding cohomological $A$-parameter $\sigma$ restricted to $W_\C \times \SL_2(\C)$ lands inside such a 
Levi subgroup $\wM$ with the image of $\SL_2(\C)$, a principal $\SL_2(\C)$ in $\wM$, and therefore,
the parameter $\sigma$ as an $m$-dimensional representation of $W_\R \times \SL_2(\C)$ must be:
$$\sigma = \sigma_{d_1} \otimes [n_1] + \sigma_{d_2} \otimes [n_2] + \cdots + \sigma_{d_k} \otimes [n_k] +
A[W],$$
  with $\sigma_{d_i}$, the irreducible 2-dimensional representation of $W_\R$ which contains the character $(z/\bar{z})^{d_i/2}$ when restricted to $\C^\times$, and $[n_i]$ the
  irreducible representation of $\SL_2(\C)$ of dimension $n_i=\dim (V_i)$; $A[W]$ is a representation of
  $W_\R \times \SL_2(\C)$
  which restricted to $\SL_2(\C)$ is 
  a  principal $\SL_2(\C)$ in $\wG(W)$ (which is the irreducible
  representation $[\dim(W)]$ of $\SL_2(\C)$
  in all cases unless we are dealing with $\SO(p,q)$ with $p+q$ even, in which case it is  the representation $[\dim(W) -1] + 1$;
  the integers $n_i,d_i$ are arbitrary
  with the only constraint that the infinitesimal
  character of $\sigma$ is that of a finite dimensional representation of $\sG$.
  All possible $A$-parameters $\sigma: W_\R \times \SL_2(\C) \rightarrow {}^L \sG$ with the above restriction on $\sigma|_{\SL_2(\C)}$ are allowed with the only condition
  coming from $\det(\sigma)=1$ if $\sG$ is not an even orthogonal group, and if it is, $\det(\sigma)$ must be the determinant
  of the quadratic space.
\end{example}
  
\section{Adams-Johnson packets:  Definition and parametrization} 
\label{AJpack}

In this section we review  Adams-Johnson packets ($AJ$-packets, for short) of  unitary representations of $\sG(\R)$ and their parameters.
This summarizes work of Adams-Johnson \cite{AJ} building on the Vogan-Zuckerman \cite{VZ} classification of cohomological representations, with some details on $A$-parameters clarified.  The main concern of \cite{AJ}, apart from defining the $AJ$-packets, was to prove endoscopic character relations for suitable signed combinations of characters in these packets.  These relations constitute important evidence that these are the correct $A$-packets, but since they will not be of immediate relevance to us we do not recall them.  

Adams-Johnson packets and their parameters have also been reviewed in \cite{Ar, AMR, Ta}, but under restrictive assumptions on $\mathsf{G}(\R)$ (typically, under the assumption that the derived group  of $\sG(\R)$ has discrete series).  Since it is essential for us to deal with the general case, and we have not found an account of this in the literature, we give one here. We have tried to give a self-contained account of the elementary parts of the theory.

We use the usual notation:  $\sG$ is a connected reductive group over $\R$, $G=\sG(\R)$, $\g_0$ is the Lie algebra of $G$,  $K$ is a maximal compact, $\theta$ is the Cartan involution of $G$ as well as of $\g_0$ given by $K$, $\g=\k+\p$ the corresponding
Cartan decomposition.  We also write $\theta$ for the extension of $\theta$ from $\g_0$ to $\g$ and its action on duals etc.  The following definition is basic to what follows: 

\begin{defn}
A {\it $\theta$-stable parabolic subalgebra for $\g_0$} is a parabolic subalgebra $\q \subset \g$ such that 
\begin{enumerate} 
\item $\theta(\q)=\q$ 
\item $\q$ and $\bar{\q}$ are opposite, i.e. $\q \cap \bar{\q} = \l$ is a Levi subalgebra of $\q$.   
\end{enumerate} 
(Note that this is not simply a parabolic subalgebra of $\g$ stable under $\theta$ because of the second condition.) 
 We will usually write $\q=\l+\u$ for the Levi decomposition.  
The Levi subalgebra $\l$ is defined over $\R$ and defines a connected reductive subgroup $\mathsf{L}$ of $\mathsf{G}$ with real points $L= {\rm Stab}_G(\q)$, Lie algebra $\l_0 = \l \cap \g_0$, and Cartan involution $\theta|_{\l_0}$. 
\end{defn}

The parametrization of $\theta$-stable parabolics for $\g_0$ in terms of the dual group will be reviewed in Proposition~\ref{Sigmaprop} below; here we recall enough to state Theorem \ref{AJ1} below precisely.

Let $\sT^c \subset \sB \subset \sG_\C$ be a $\theta$-stable fundamental torus defined over $\R$ and a $\theta$-stable  Borel subgroup containing it.  (The existence of a $\theta$-stable  Borel subgroup containing $\sT^c$ is proved in Proposition \ref{para}.)
The dual group $\wG$ of $\sG$ comes equipped with a Borel subgroup $\wB$ and maximal torus $\wT^c$.  Let  
$$
\mathcal{L} = 
\left\{\text{ $\wB$-standard Levi subgroups in $\wG$ invariant under $\omega_{\wG}\cdot j$ }\right\}.
$$  
Now a $\theta$-stable parabolic $\q$ for $\g_0$ is conjugate to a unique $\sB$-standard one, and this gives a Levi subgroup $\wL$ in $\wG$; the condition that $\q$ is $\theta$-stable translates into $\wL \in \mathcal{L}$. This defines a natural surjective mapping 
$$
\Sigma: \left\{\begin{array}{c} \text{$\theta$-stable parabolic} \\ \text{subalgebras for $\g_0$} \end{array}\right\} 
\Big{/} K
\longrightarrow 
\mathcal{L} 
$$
discussed in detail in Proposition~\ref{Sigmaprop}. 

Let $E$ be an irreducible finite-dimensional algebraic representation of $\sG(\C)$.  We will assume that $E$ is self-associate so that its $\sB$-highest weight $\lambda_E\in X^\star(\sT^c)$ is $\theta$-invariant.  (Recall that this is a necessary condition for the infinitesimal character of $E$ to match with that of a unitary representation, e.g. \cite[II.6.3]{BW}.) 
Let $\mathcal{L}_E$ be the following subset of $\mathcal{L}$: 
$$
\mathcal{L}_E =\left\{\ \wL\in \mathcal{L}: \langle \lambda_E, \alpha\rangle=0 \text{ for } \alpha\in \Phi(\wT^c,\wL)\ \right\}.
$$
By Proposition~\ref{Sigmaprop} below, for $E$ as above, 
$\dim(E^\u)=1$ if and only if $\Sigma(\q) \in \mathcal{L}_E$.  Thus when $\q \in \Sigma^{-1}(\wL)$ for $\wL \in \mathcal{L}_E$, we have a one-dimensional unitary representation $E^\u$ of $L$, which we  denote $\pi_L$.

The following theorem defines $AJ$-packets and summarizes their properties: 

\begin{thm}\label{AJ1} 
Let $\sG, G, K, E$ etc be as above.  For $\wL \in \mathcal{L}_E$ and $c \in H^1(\R,Z(\wL))$, define the Adams-Johnson packet 
$$
\Pi_{E,\wL,c} = \left\{\ \mathcal{R}^S_\q(\pi_L \otimes \nu_L(c)): \q \in \Sigma^{-1}(\wL)\ \right\} 
$$
where:
\begin{enumerate}

\item[--] $\pi_L:L \to \C^\times$ is the unitary character of $L$ on $E^\u$, 

\item[--] $\nu_L(c)$ is the finite-order character of $L$ given by the homomorphism $\nu_L:H^1(\R,Z(\wL)) \to \Hom_{cts}(L/L_0,\C^\times)$ of \cite{La} (see e.g. 
\cite[p.~319]{Ta} or \cite[Appendix~A]{LM}), and 

\item[--] $\mathcal{R}^S_\q(\cdot)$ is the cohomological induction functor in degree $S=\dim(\u\cap \k)$. 

\end{enumerate} 
When $c$ is the identity we will simply write $\Pi_{E,\wL,c}$ as  $\Pi_{E,\wL}$.  Then 

(1)  The representations in a packet are pairwise inequivalent, so that isomorphism classes in  $\Pi_{E,\wL,c}$ are in canonical bijection with 
$$
\Sigma^{-1}(\wL) = W(G,T^c) 
\backslash W(\mathsf{G},\mathsf{T}^c)^\theta/W(\widehat{\mathsf{L}}, \widehat{\mathsf{T}}^c)^\theta,  
$$ 
where we identify $W(\sG, \mathsf{T}^c)$ and $W(\wG,\widehat{\mathsf{T}}^c)$ to think of $W(\widehat{\mathsf{L}},\widehat{\mathsf{T}}^c)$ as a subgroup of $W(\sG,\mathsf{T}^c)$.  
The choice of $\q=\l+\u\in \Sigma^{-1}(\wL)$ fixes a bijection 
$$
\Sigma^{-1}(\wL)=\ker\left[H^1(\R,\sL) \to H^1(\R,\sG)\right]
$$ 
with a set of pure inner forms of $\sL$. 

(2) 
The packets $\Pi_{E,\wL}$  consist of irreducible unitary representations with $(\g,K)$-cohomology.  More precisely, if $R=\dim(\u\cap \p)$,
$$
H^*(\g,K, \mathcal{R}^S_\q(\pi_L)\otimes E^\vee) = H^{*-R}(\l,L\cap K, \C). 
$$
Every irreducible unitary  representation 
with $H^*(\g,K,V\otimes E^\vee)\neq  \{0\}$ 
appears in such a packet.

The remaining $AJ$-packets $\Pi_{E,\wL,c}$ are twists of the packets $\Pi_{E,\wL}$ by finite-order characters of $G/G_0$ and their union consists of all irreducible unitary representations with $H^*(\g,K_0,V\otimes E^\vee) \neq \{0\}$.

(3) If $\wL = \wT^c$, so that $\Sigma^{-1}(\wT^c)$ consists of the $\theta$-stable Borel subalgebras, then $\Pi_{E,\wT^c}$ consists of tempered representations which belong to the fundamental series of $G$.  The packet $\Pi_{E,\wT^c,c}$ consists of  twists of the fundamental series representations by a character of $G/G_0$.  Conversely, every irreducible tempered representation with $H^*(\g,K_0,V\otimes E^\vee)\neq \{0\}$ belongs to some $\Pi_{E,\wT^c,c}$.

\end{thm}

The proof of this theorem will be discussed after we discuss the relation with $A$-parameters. 

The next theorem describes the parametrization of $AJ$-packets in terms of the cohomological $A$-parameters introduced in
Definition~\ref{cohomologicalCparameter} of  the previous section.
We restrict ourselves to $G=\sG(\R)$ connected
for reasons explained below in Remark~\ref{parametergeneral}. 
 Given a cohomological $A$-parameter $\psi:W_\R \times \SL_2(\C) \to {}^L\sG$,
 the restriction of $\psi$ to $\C^\times \times \SL_2(\C)$ gives (by Theorem~\ref{cohomo})
a standard Levi subgroup $\wL$ (the standard conjugate of $ Z_{\wG}(\psi(\C^\times))$) of $\wG$ invariant under $\omega_{\wG} \cdot j$.  Suppose that $\wL \in \mathcal{L}_E$. For $\q=\l+\u$ in $\Sigma^{-1}(\wL)$, let 
$\pi_{L}:L \to \C^\times$ be the character of $L$ on $E^{\u}$ and let $\lambda=d\pi_L:\l \to \C$.  
(Note that $L$ is also connected e.g. by \cite[Lemma 5.10]{KV}, so that $\lambda$ determines $\pi_L$.) 
We define 
$$
\Pi(\psi) := \Pi_{E,\wL}=\left\{\, A_\q(\lambda): \q \in \Sigma^{-1}(\wL) 
\, \right\} 
$$
where we use the notation $A_\q(\lambda) = \mathcal{R}^S_{\q}(\pi_L)$ as in \cite{VZ}.  
With this definition, we can state the following theorem, which is just a restatement of the previous theorem:

\begin{thm}\label{AJ2}
  Let $G=\sG(\R)$ be connected.  Then for  each cohomological $A$-parameter $\psi$,  the Adams-Johnson packet $\Pi(\psi)$ of representations consists of pairwise inequivalent irreducible unitary representations.  Given $\wL$ associated with $\psi$ as above and choosing $\q = \l + \u \in \Sigma^{-1}(\wL)$ the packet is in bijection with the set of pure inner forms $\ker\left[H^1(\R,\sL) \to H^1(\R,\sG)\right]$.  The union of all such $\Pi(\psi)$ consists of all irreducible unitary representations with $(\g,K)$-cohomology with respect to a finite-dimensional representation of $\sG(\C)$.

\end{thm}

\begin{remark}
The parametrization $\psi \mapsto \Pi(\psi)$ in Theorem~\ref{AJ2} is not one-to-one in general.  For example, if $G$ is compact then any $\wL \in \mathcal{L}_E$ gives the same packet (see Example \ref{example:compact} in the next section). 
It can also fail to be one-to-one because the map $\nu_G:H^1(\R,Z(\wG)) \to \Hom_{cts}(G/G_0,\C^\times)$ need not be injective, e.g. if $G=\GL_n(\mathbb{H}), n \geq 1$.  
\end{remark}

\begin{remark}\label{parametergeneral}
Theorem \ref{AJ2} has only been stated for $\sG$ with $G$ connected, unlike Theorem \ref{AJ1}, i.e. for a general $\sG$ we do not specify a map $\psi \mapsto \Pi(\psi)$ from cohomological $A$-parameters to $AJ$-packets.  We remark on this briefly.   

Using the correspondence of Theorem~\ref{cohomo}, for a fixed $E$, 
the set $\mathcal{G}_E$ of conjugacy classes of cohomological $A$-parameters with infinitesimal character of $E$
maps to $\mathcal{L}_E$ by restricting to $\C^\times \times \SL_2(\C)$.

Then $\mathcal{G}_E\to \mathcal{L}_E$ is an $H^1(\R,Z(\wG))$-torsor by Proposition~\ref{general}.  On the other hand,  the set of $AJ$-packets of representations with infinitesimal character matching $E$ is 
$$
\mathcal{A}_E=\left\{\Pi_{E,\wL,c}: \wL \in \mathcal{L}_E, c \in H^1(\R,Z(\wL))\right\}
$$ 
(notation as in Theorem~\ref{AJ1}, we will think of these as finite sets of isomorphism classes of $(\g,K)$-modules).
The set $\mathcal{A}_E$ is  an $H^1(\R,Z(\wG))$-torsor over $\mathcal{L}_E$ 
by $\Pi_{E,\wL,c} \mapsto \wL$.  
(Recall that $H^1(\R,Z(\wG)) \twoheadrightarrow H^1(\R,Z(\wL))$ by Proposition~\ref{uniqueness}. 
The action of $H^1(\R,Z(\wG))$ on $\mathcal{A}_E$   factors through $\nu_G:H^1(\R,Z(\wG))\to \mathrm{Hom}_{cts}(G/G_0,\C^\times)$, i.e. it is by twisting by characters of $G/G_0=K/K_0$.)  Fixing a map $\psi \mapsto \Pi(\psi)$ amounts to specifying a map of torsors 
$$\mathcal{G}_E \to \mathcal{A}_E$$ 
over $\mathcal{L}_E$.  
Both torsors have natural sections:  For $\mathcal{A}_E$,  a basepoint in each fibre is given by the packet $\Pi_{E,\wL}$ of $(\g,K)$-cohomological representations in (1) of Theorem~\ref{AJ1}.
On the other hand, a section of $\mathcal{G}_E \to \mathcal{L}_E$ can be specified using arguments of Langlands and Shelstad (see \cite{Ar}, \cite[4.2.2]{Ta}, \cite[5.1]{AMR}).  
This can be used to give a map of torsors $\mathcal{G}_E \to \mathcal{A}_E$.  To check that the resulting map $\psi \mapsto \Pi(\psi)$ is the correct parametrization, it would suffice to check compatibility with the Langlands parametrization in the quasisplit case, i.e. it must be checked that $\Pi(\psi)$ contains the $L$-packet of the $L$-parameter associated with $\psi$ when $\sG$ is quasisplit. This is asserted in \cite[3.3]{AJ} without proof and we will not check it here  (not even in the connected case.) 
\end{remark}

We will now discuss the proof of Theorem~\ref{AJ1}, which will take up the rest of this section. 
Theorem~\ref{AJ1} follows from Proposition~\ref{Sigmaprop}, which parametrizes $\theta$-stable parabolic subalgebras in $\g_0$ in terms of the dual group, and Proposition~\ref{cohprop}, which proves the cohomological properties of representations in $AJ$-packets.    
(Of course, we will do little more than combine well-known facts from the literature.) 
We will first prove a version of the parametrization for subgroups of $\sG(\C)$ which correspond to $\theta$-stable parabolic subalgebras.  Thus a {\it $\theta$-stable parabolic subgroup $\sP$ for $\sG$} is a parabolic subgroup of $\sG(\C)$ which satisfies 
\begin{enumerate}
  \item $\theta(\sP)=\sP$,
  \item $\sP$ and $\overline{\sP}$ are opposite parabolics.   \end{enumerate}
In other words, $\sP$ is a subgroup of $\sG(\C)$ with the Lie algebra a $\theta$-stable parabolic for $\g_0$. 
The   Levi subgroup $\sL:= \sP \cap \overline{\sP}$ is evidently defined over $\R$ and stable under $\theta$.

\begin{prop} \label{para}
(a)  Let $\sG$ be a reductive algebraic group over $\C$ with a fixed maximal torus $\sT^c$. Then the set $X_{\sT^c}$ of
  parabolics $\sP$ of $\sG$ containing $\sT^c$ which are conjugate in $\sG(\C)$ are in a single $W(\sG,\sT^c)$-orbit,
  and are in bijective correspondence with $W(\sG,\sT^c)/W(\sL,\sT^c)$, where $\sL$ is a Levi subgroup of $\sP$ containing $\sT^c$.
  
  (b)  Let $\sG$ be a reductive algebraic group over $\R$ with a Cartan involution $\theta$ on $G=\sG(\R)$,
  and a fixed $\theta$-stable fundamental torus $\sT^c=\sT \, \sA$ over $\R$ where $\sT(\R)$ is a maximal compact torus
  in $\sG(\R)$. Then there is a $\theta$-stable Borel subgroup $\sB$ of $\sG(\C)$ containing $\sT^c$, which we fix. 
  Let $Y_{\sT^c}(\sP_S)$ be the set of $\theta$-stable parabolics $\sP$ of  $\sG$ containing $\sT^c$ which are conjugate to 
  $\sP_S \supset \sB$
  for a  set $S$ of simple roots of $\sT^c$ in $\sB(\C)$, with $\sL_S\supset \sT^c$ the Levi subgroup of $\sP_S$.
  Then  $Y_{\sT^c}(\sP_S)$ is non-empty
  if and only if the set $S$ is $\theta$-invariant.

  If $Y_{\sT^c}(\sP_S)$ is non-empty,
  $Y_{\sT^c}(\sP_S)= W(\sG,\sT^c)^{\theta}/W(\sL,\sT^c)^{\theta}$.
  The set $Y_{\sT^c}(\sP_S)$  carries a natural conjugation
  action of $W(G,T^c):=N_G(T^c)/T^c$, where
  $T^c=\sT^c(\R)$.
  The orbits of $W(G,T^c)$ on $Y_{\sT^c}(\sP_S)$ are parametrized by the double cosets \[W(G,T^c)\backslash
  W(\sG,\sT^c)^\theta/W(\sL_S,\sT^c)^\theta.\]

  (c) Given a $\theta$-stable parabolic subgroup $\sP$ for $\sG$  with Levi subgroup  $\sL= \sP \cap \overline{\sP}$, 
    let $X(\sP)$ be the set of all $\theta$-stable parabolics $\sQ$ for $\sG$ which are  conjugate to $\sP$ in $\sG(\C)$. Then  $X(\sP)$ is invariant under $K=\sG(\R)^\theta $,
  and there is a canonical bijection 
$$
K \backslash X(\sP) = W(G,T^c) 
\backslash  W(\sG,\sT^c)^\theta/W(\sL,\sT^c)^\theta.
$$

(d) The Levi subgroup  $\sL= \sP \cap \overline{\sP}$ of  a $\theta$-stable parabolic $\sP$ contains a fundamental torus $\sT^c$ of $\sG$. For a given 
$\sG(\C)$-conjugacy class of $\theta$-stable parabolics $\sP$ for $\sG$, such Levi subgroups are Galois twists of each other using a
  cocycle from $H^1(\R, \sT^c)$.  More precisely, these subgroups $\sL$ of $\sG(\R)$ are pure inner forms of each other and there is an identification 
$$
K\backslash X(\sP) = \ker\left[H^1(\R,\sL) \to H^1(\R,\sG)\right].
$$ 

\end{prop}
\begin{proof} Part (a) of the Proposition follows from the fact that inside any parabolic, any
  two maximal tori are conjugate, therefore,
  given two pairs of parabolics and maximal tori $(\sP_1,\sT_1), (\sP_2,\sT_2)$, if the parabolics $\sP_1$ and $\sP_2$ are conjugate
  in $\sG$, we can assume that there is a conjugation which takes the pair  $(\sP_1,\sT_1)$ to $ (\sP_2,\sT_2)$.
  Since any element in the normalizer of $\sT^c$ is by definition an element of the Weyl group  $W(\sG,\sT^c)$
  up to an element of $\sT^c$,
  this proves part (a).

  For the remaining parts of the proposition, we now make some general remarks first.
  \begin{enumerate}
\item  For any real reductive group $\sL(\R)$ with a Cartan involution $\theta$, the centralizer of a
  maximal compact torus in $\sL(\R)$
  is a maximal torus in $\sL(\R)$. This follows from the more elementary observation (applied to $Z_{\sL}(\sT)/\sT$
where $\sT$ is a   maximal compact torus in $\sL(\R)$, and proved most easily by a Lie algebra argument)
  that if on a real reductive group $\sH(\R)$ with a Cartan involution $\theta$, $\sH^\theta(\R)$ is finite,  then $\sH$ must be a torus. This implies
  that for the fundamental torus $\sT^c=\sT\,\sA$ in $\sG$, the centralizer of $\sT$ in $\sG$ is $\sT^c$.

\item   From the observations in the last paragraph, the root space decomposition of $\sG(\C)$ with respect to $\sT^c$ when restricted to
  $\sT$ gives rise to the root space decomposition of $\sG(\C)$ with respect to $\sT$.
    Define a positive system of roots in   $X^\star(\sT)$ to be compatible with a
    positive system of roots in  $X^\star(\sT^c)$ if under the natural surjective map on character groups
    $X^\star(\sT^c) \rightarrow X^\star(\sT)$, a root in  $X^\star(\sT^c)$ is positive if and only if
    it goes to a positive root in $X^\star(\sT)$. Clearly, a compatible system of positive roots on   $X^\star(\sT^c)$ and
    $X^\star(\sT)$ exists which we fix now.   Thus we have a Borel subgroup $\sB$ of $\sG(\C)$ defined using either of the positive systems of roots, which as $\theta$ operates by
  identity on $\sT(\R)$, preserves $\sB$, whereas since complex conjugation operates by $t\mapsto t^{-1}$ on $\sT(\R)$,
  $\bar{\sB}(\C)$ is the opposite Borel subgroup, i.e.,  $\sB \cap \bar{\sB}= \sT^c(\C)$. (This proves the existence of a
$\theta$-stable Borel subgroup of $\sG(\C)$.)
  Further, as the Cartan involution
  $\theta$ preserves $\sT^c$ and $\sB$, it permutes the simple roots of  $\sT^c$
  on $\sB$. Thus $\theta$-invariant parabolic
  subgroups $\sP$ of $\sG(\C)$ containing $\sB$ 
  are exactly of the form $\sP_S = \sP$ for any subset of the set $S$ of the simple roots of   $\sT^c$
  on $\sB$ such that $\theta(S)=S$. Clearly, for such $\theta$-invariant parabolics, $\bar{\sP}_S \cap \sP_S = \sL_S$, thus,
  $\theta$-invariant parabolics $\sP_S$  are $\theta$-stable; in fact the particular choice of $\sB$ played no role, therefore
any $\theta$-invariant parabolic
  containing $\sT^c$ is $\theta$-stable.

\item   Next, we observe that for any
  $\theta$-invariant real reductive subgroup   $\sL(\R)$ of $\sG(\R)$
  which is a Levi subgroup in a $\theta$-invariant parabolic $\sP$ in $\sG$,
  a fundamental torus in $\sL(\R)$ is a fundamental torus in $\sG(\R)$. For this,
  let $Z_\sL$ be the connected center of $\sL$, which is $\theta$-invariant as $\sL$ is $\theta$-invariant.
Let $Z_\sL=Z_cZ_s$ such that $\theta$ operates
by the identity on $Z_c$ and by inversion on $Z_s$. Now note that the centralizer of $Z_c$ and $Z_\sL$ in $\sG$ must be the same,
else there will be a root space of $Z_\sL$ on $\sP$ on which $Z_c$ acts trivially, hence acts via  $Z_\sL/Z_c$. But we are given
that $\sL(\R)$ is a Levi subgroup of a $\theta$-invariant parabolic $\sP \subset \sG$, whereas $\theta$ operates by inversion on  $Z_\sL/Z_c$,
a contradiction.
   Thus a
   $\theta$-invariant real reductive subgroup   $\sL(\R)$ of $\sG(\R)$
which is a Levi subgroup in a $\theta$-invariant parabolic $\sP$ in $\sG$
   contains, up to $\sG(\R)$ conjugacy,
  a fixed fundamental torus of $\sG(\R)$.

\item In a real reductive group $\sG(\R)$ with Cartan involution $\theta$, a fundamental Cartan is unique
  up to conjugacy by $K=\sG(\R)^\theta$. This is direct consequence of the conjugacy of maximal tori in $K$,
  and the definition of a fundamental Cartan subgroup of $\sG(\R)$
  as the centralizer of a maximal torus of $K$.

\item The image of $W(K,T)$ inside $W(\sG,\sT^c)^\theta$ is $W(G,T^c)$.
  \end{enumerate}

  In part (2) above, we have already proved the existence of a $\theta$-stable Borel subgroup $\sB$, and deduced
  when $Y_{\sT^c}$ is non-empty. 

By part (a), any two parabolic subgroups of $\sG(\C)$ containing $\sT^c$ which are conjugate in $\sG(\C)$ are conjugate by  $W(\sG,\sT^c)$. 
Now having fixed a $\theta$-stable parabolic $\sP=\sP_S$, for a set $S$ of simple roots of $\sT^c$ on $\sB$ which is
invariant under $\theta$, its conjugate $\sP^w = w\sP w^{-1}$
by  $w \in W(\sG,\sT^c)$,
is $\theta$-stable if and only if $w$ as an element of
$W(\sG,\sT^c)/W(\sL_S,\sT^c)$ is represented by an element in
$ W(\sG,\sT^c)^\theta =\{v \in W(\sG,\sT^c) | \theta(v)=v\} $. To see this,
observe that $\theta(\sP^w) = \sP^{\theta(w)}$, and therefore  $\theta(\sP^w) = \sP^{w}$, if and only if $ \theta(\tilde{w})=\tilde{w}$, where $\tilde{w}$ is the image
of $w$ in the coset space  $W(\sG,\sT^c)/W(\sL_S,\sT^c)$. Now we claim that a
$\theta$-invariant element of  $W(\sG,\sT^c)/W(\sL_S,\sT^c)$ is in fact represented by an element of  $W(\sG,\sT^c)^{\theta}$. Since $\theta$ preserves all the positive roots in $\sG$ as well as in $\sL_S$, this is immediately clear by considering the Kostant representative of $\tilde{w}$ in $W(\sG,\sT^c)$ which by the
uniqueness of the  Kostant representative, must belong to $W(\sG,\sT^c)^{\theta}$,
thus 
\[ [W(\sG,\sT^c)/W(\sL_S,\sT^c)]^\theta =
W(\sG,\sT^c)^{\theta}/W(\sL_S,\sT^c)^{\theta}.\]

It was observed at the end of part (2) that  $\theta$-invariant parabolic containing $\sT^c$ 
are $\theta$-stable, proving that 
\[Y_{\sT^c}(\sP_S)= W(\sG,\sT^c)^{\theta}/W(\sL_S,\sT^c)^{\theta}.\]

The assertion on the orbits of $W(G,T^c)$ on $Y_{\sT^c}(\sP_S)$ being parametrized by the double cosets \[W(G,T^c)\backslash
W(\sG,\sT^c)^\theta/W(\sL_S,\sT^c)^\theta,\]
is of course clear.

Now we prove  part (c) of the Proposition. By (4) above, for any $\theta$-stable parabolic $\sQ$, the Levi subgroup
$\sM= \sQ \cap \bar{\sQ}$ contains a $K$-conjugate of $\sT^c$.  Therefore by (5), $K$-conjugacy classes of such
parabolics $\sQ$ is the same as the set of $W(G,T^c)$ orbits on  $Y_{\sT^c}$ (which is the set of
$\theta$-stable parabolics $\sP$ for $\sG$
containing $\sT^c$ and belonging to a single
$\sG(\C)$-conjugacy class of parabolics of $\sG(\C)$). This proves (c) as a consequence of (b).

Now we prove (d).   Recall now the well-known fact 
  that all the automorphisms of a torus over $\R$ which is either split or compact are defined over $\R$. Since an element $w \in W(\sG,\sT^c)^\theta$ preserves the decomposition $\sT^c=\sT\, \sA$ as product of compact and split tori, the automorphism of $\sT^c$ induced by $w$ is defined over $\R$, which means that
  $w\bar{w}^{-1} \in \sT^c(\C)$. This proves that  the twisted Levi
  $\sL^{w}$ is obtained from $\sL$ by twisting by a cocycle in
  $H^1(\R,\sT^c)$.    The mapping $W(\sG,\sT^c)^\theta \to Z^1(\R,\sT^c)$ defined by $w \mapsto w\bar{w}^{-1}$ defines a mapping   $W(G,T^c)\backslash W(\sG,\sT^c) \to H^1(\R,\sT^c)$ 
with image in the kernel of $H^1(\R,\sT^c) \to H^1(\R,\sG)$.  That this gives a bijection 
$$
W(G,T^c)\backslash W(\sG,\sT^c) \to \ker\left[H^1(\R,T^c) \to H^1(\R,\sG)\right]
$$ 
is a consequence of the description by Borovoi \cite[Theorem 9]{Boro} of $H^1(\R,\sG)$ as the quotient of the set  $H^1(\R,\sT^c)$ by the action of $W(\sG,\sT^c)^\theta$ defined by $t \mapsto n^{-1}t\bar{n}$ (for $t \in \sT^c(\C)$ with $t\bar{t}=1$ and $n \in N_{\sG(\C)}(\sT^c)$).  (For more details about this action and the computation of the isotropy as $W(G,T^c)$ we refer to the proof of Theorem \ref{serre} below.) 
This is the case of $\sL=\sT^c$, the general case follows easily.     
  \end{proof}

The following lemma relates the Levi subgroups of $\theta$-stable parabolic subgroups of $\G(\C)$ of the previous proposition
in terms of the dual group.

\begin{lemma}\label{theta}
  If $\sG$ is a connected reductive group over $\R$ with $\theta$ a Cartan involution on $\sG(\R)$, and
  the dual group ${}^L\sG = \wG\cdot \langle j\rangle$
  then the twisted Levi subgroup $\sL$  of Proposition \ref{para} has for its dual group ${}^L\sL = \wL\cdot \langle \omega_{\wL}\omega_{\wG}j\rangle$,
  thus is the one which appears in Theorem \ref{cohomo}, and Remark \ref{VZcharacter} from the perspective of dual groups.
  \end{lemma}
\begin{proof}
  By Remark \ref{dualgroup}, the Cartan involution $\theta$ which operates on $(\sG,\sB,\sT^c)$, gives rise to a pinned automorphism
${}^\vee \theta$ of $\wG$
with \[{}^\vee \theta
= (-\omega_{\wG})\cdot j = \iota_{\wG}\cdot j, \tag{*}\] where $\iota_{\wG}$ is the opposition involution. Since the Cartan involution $\theta$
leaves $\sL$ invariant, let $\theta_\sL$ be $\theta$ restricted to $\sL$ which is a Cartan involution of $\sL$. The involution ${}^\vee \theta$ of $\wG$ leaves $\wL$ invariant and is ${}^\vee \theta_L$ of $\wL$. Now (*) restricted to $\sL$ gives us:
\[{}^\vee \theta_L= (-\omega_{\wG})\cdot j 
= (-\omega_{\wL} \omega_{\wL}\omega_{\wG} )\cdot j = (-\omega_{\wL})\cdot j_\sL ,\]
for $j_\sL= \omega_{\wL}\omega_{\wG} \cdot j$, proving the lemma.
\end{proof}

For $\theta$-stable parabolic subalgebras for $\g_0$, Proposition \ref{para} and Lemma \ref{theta} implies the following: 

\begin{prop}  \label{Sigmaprop}
For $\sG$ connected reductive there is a natural surjective map 
\begin{equation}
\Sigma: \left\{\begin{array}{c} \text{$\theta$-stable parabolic} \\ \text{subalgebras for $\g_0$} \end{array}\right\} 
\Big{/} K
\longrightarrow 
\mathcal{L} =\left\{\begin{array}{c} \text{standard Levi subgroups $\wL \subset \wG$}\\ \text{invariant under $\omega_{\wG}\cdot j$} \end{array} \right\} 
\notag 
\end{equation}
with the following properties:  

(1) 
If $\q \in \Sigma^{-1}(\wL)$ has associated Levi $\sL$ then the fibre is identified with 
\begin{equation}
\begin{split}
\Sigma^{-1}(L) 
= W(G,T^c)\backslash W(\sG,\sT^c)^\theta/W(\sL,\sT^c)^\theta. 
\end{split}
\notag
\end{equation}
The fibre is also identified with a set of pure inner forms of $\sL$: 
$$
\Sigma^{-1}(\wL)
= \ker\left[H^1(\R,\sL) \to H^1(\R,\sG)\right]. 
$$
For each choice of a $\theta$-stable Borel subalgebra $\b \supset \t^c$, the fibre $\Sigma^{-1}(\wL)$ contains a unique $\b$-standard element.

(2) 
If $E$ is a finite-dimensional irreducible representation of $\sG(\C)$ and $\q=\l+\u$ is $\theta$-stable for $\g_0$, then  
$$
\dim(E^\u)=1\ \Longleftrightarrow\  \Sigma(\q) \in \mathcal{L}_E = 
\left\{\ \wL\in \mathcal{L}: \langle \lambda_E, \alpha\rangle=0 \text{ for } \alpha\in \Phi(\wT^c,\wL)\ \right\}.
$$
\end{prop} 

\begin{proof} 
This is more-or-less an immediate consequence of the Proposition \ref{para}, but we spell things out explicitly as this is useful in computations.  

(1)  
The parametrization of $\theta$-stable parabolics has the following explicit form.  Fix a fundamental Cartan subalgebra $\t^c$ in $\g$ with $\t^c_0 = \t_0 \oplus \a_0$ according to the action 
(by $+1$ or $-1$, respectively) of $\theta$. 
Let $\Phi(\g,\t^c)$ be the roots, which are real-valued on $\a_0+ i \t_0$, and let $\g^\alpha$ denote the $\alpha$-root space for $\alpha \in \Phi(\g,\t^c)$.   For $\lambda_0 \in i\t_0^*$ define 
$\q=\q(\lambda_0)= \l + \u$ where 
$$
\l = \l(\lambda_0) = \t^c \oplus \sum_{\{\alpha \in \Phi(\g,\t^c):\langle \lambda_0, \check{\alpha}\rangle =0\}} \g^\alpha, \quad 
\u =\u(\lambda_0) = \sum_{\{\alpha \in \Phi(\g,\t^c):\langle \lambda_0, \check{\alpha}\rangle >0\}} \g^\alpha.
$$  
Then $\q$ is $\theta$-stable, $\theta(\l)=\l$, and $\theta(\u)=\u$.  
Conversely, since every $\theta$-stable parabolic for $\g_0$ contains a fundamental Cartan (by Proposition~\ref{para}) and since all fundamental Cartans are conjugate by $K$, 
we see that up to $K$-conjugacy, every $\theta$-stable parabolic  for $\g_0$ is of the form $\q(\lambda_0)$.

The definition of the map $\Sigma$ is the obvious one:  Fix a $\theta$-stable fundamental torus $\sT^c$ and $\theta$-stable Borel subalgebra $\sB$ of $\sG_\C$ containing it.  (The existence of such a $\sB$ was proved in the proof of the previous proposition.) 
By the previous proposition a $\theta$-stable parabolic $\q$ contains a fundamental Cartan, which, after conjugation by $K$, we may assume is $\t^c$.  Forming the dual group with respect to 
these choices there is an isomorphism of based root data  
$$ 
(X^\star(\mathsf{T}^c), \Delta_\mathsf{G}, X_\star(\mathsf{T}^c), \check{\Delta}_\mathsf{G}) \cong 
(X_\star(\widehat{\mathsf{T}}^c), \check{\Delta}_{\widehat{\mathsf{G}}}, X^\star(\widehat{\mathsf{T}}^c), \Delta_{\widehat{\mathsf{G}}}) 
$$
where $\Delta_\sG$ is the set of positive roots fixed by $\sB$ then gives a Levi $\wL$ in $\wG$ which is invariant under $\omega_{\wG}\cdot j$.  The parametrization of the fibres in terms of double cosets or in terms of Galois cohomology is contained in the previous proposition.  (We leave it to the reader to check that the map $\Sigma$ does not depend on the choices made to define it.)

(2) For this note that the condition $\dim E^\u=1$ is independent of $\q$ in $\Sigma^{-1}(Q)$, i.e. it holds for all $\q \in \Sigma^{-1}(\wL)$ if it holds for one, because the corresponding $\theta$-stable parabolics are conjugate under $\mathsf{G}(\C)$.   
It remains to check that this condition is the same as the condition $\Sigma(\q) \in \mathcal{L}_E$.  We may assume $\q$ is the $\b$-standard element of $\Sigma^{-1}(\Sigma(\q))$.  Then $E^\u=H^0(\u,E)$ as an $\l$-module is the irreducible highest-weight module for $\l$ with highest weight $\lambda_E$.  This has dimension one if and only if $\langle \lambda_E, \alpha \rangle=0$ for all roots $\alpha$ appearing in $\l$, which is precisely the condition $\Sigma(\q) \in \mathcal{L}_E$.  
\end{proof}

\medskip

We now discuss the $AJ$-packets of representations defined in Theorem~\ref{AJ1}.  The definition uses  Zuckerman's cohomological induction from $\theta$-stable parabolic subalgebras, so we quickly review some of the formal properties of these functors first. 
The basic construction of cohomological induction 
(see \cite[6.3]{V} or \cite[V.1]{KV}) associates with a $\theta$-stable parabolic subalgebra $\q=\l+\u$ for $\g_0$ and an $(\mathfrak{l},L\cap K)$-module $\pi$ a sequence of derived functor $(\g,K)$-modules 
$$
\mathcal{R}^i_\q(\pi) = \mathcal{R}_{(\q,L\cap K)}^{(\g,K), i}(\pi)
$$  
for $i\geq 0$.  (The more precise notation 
$\mathcal{R}_{(\q,L\cap K)}^{(\g,K), i}$ instead of $\mathcal{R}_\q^i$ will only be used when we wish to emphasize the source and target categories of the functor.)  These are of finite length if $\pi$ is of finite length.  The construction has the following properties: 
\begin{enumerate}
\item[(i)]
If $\q$ and $\q'$ are $K$-conjugate, so that $L'=\mathrm{Int}(k)(L)$, then $\mathcal{R}^i_\q(\pi)\cong \mathcal{R}^i_{\q'}(\pi')$ where $\pi'=\mathrm{Int}(k^{-1})\circ \pi$. 
\item[(ii)]
If $\h \subset \l$ is a Cartan subalgebra and 
$\pi$ has infinitesimal character $\chi_\pi \in W(\l,\h)\backslash \h^*$ then the infinitesimal character of $\mathcal{R}^i_\q(\pi)$ is $\chi_\pi + \rho(\u)\in W(\g,\h)\backslash \h^*$, where $\rho(\u)$ is the half-sum of roots of $\h$ appearing in $\u$. 
\end{enumerate} 
(The proof of (ii) in \cite{V} assumes $G$ connected, but 
 \cite[Cor.~5.25(b)]{KV} covers our situation.)

The cohomological properties of the representations in an $AJ$-packet are summarized in the next proposition, which completes the proof of Theorem~\ref{AJ1}, modulo some loose ends which we will tie up after the proof of the proposition.  The proofs are straightforward given the formalism of cohomological induction from \cite{KV}  (essentially Frobenius reciprocity) and  the irreducibility and unitarity theorems of \cite{KV}. 
(In the connected case they reduce to results in \cite{V,VZ}.) 
Unfortunately the discussion is made notationally complicated by the fact that we must keep track of the source and target categories of various functors.  

\begin{prop}\label{cohprop}
(1) The packets $\Pi_{E,\wL}$ consist of irreducible unitary representations with $(\g,K)$-cohomology with coefficients in $E^\vee$.  More precisely, for $R=\dim(\u\cap \p)$, 
$$
H^*(\g,K, \mathcal{R}_\q^S(\pi_L)\otimes E^\vee) = H^{*-R}(L^u/L\cap K).
$$ 
Conversely, every irreducible unitary $(\g,K)$-module $V$
with $H^*(\g,K,V\otimes E^\vee)\neq \{0\}$ belongs to some packet $\Pi_{E,\wL}$. 

(2) The packets $\Pi_{E,\wL,c}$ consist of irreducible unitary representations with $H^*(\g,K_0,V\otimes E^\vee)\neq \{0\}$.  Conversely, every irreducible unitary $(\g,K_0)$-module $V$ with $H^*(\g,K_0,V\otimes E^\vee)\neq \{0\}$ belongs to some packet $\Pi_{E,\wL,c}$. 
\end{prop} 

\begin{proof}

As usual, for  a $\theta$-stable parabolic $\q=\l+\u$ for $\g_0$ we write  $L={\rm Stab}_G(\q)$, $S=\dim(\u\cap \k)$, $R=\dim(\u\cap \p)$ etc.  The following facts will be used below:   Since ${\rm Stab}_{G_0}(\q)$ is connected (\cite[Lemma 5.10]{KV}),  $L_0 = {\rm Stab}_{G_0}(\q)=L \cap G_0$ and $(L\cap K)_0=L\cap K_0$, so that $L/L_0 \to G/G_0$  and $(L\cap K)/(L\cap K)_0 \to K/K_0$ are injective.  If $L\,G_0=G$ (i.e. $L$ meets every component of $G$) then this implies that $L/L_0=G/G_0$ and $(L\cap K)/(L\cap K)_0=K/K_0$.

(1) 
We recall some facts from \cite{KV}.   
\begin{enumerate} 

\item[(i)]
For an $(\l,L\cap K)$-module $\pi$ we have a natural isomorphism 
\begin{equation} 
\mathcal{R}^{(\g,K),i}_{(\q,L\cap K)}(\pi) = 
I_{(\g,(L\cap K)K_0)}^{(\g,K)} 
\mathcal{R}^{(\g,(L\cap K)K_0),i}_{(\q,L\cap K)}(\pi) 
\label{factoring}
\end{equation} 
for all $i$, where  
$I_{(\g,(L\cap K)K_0)}^{(\g,K)}$ is the ordinary exact induction functor between compact groups for $(L\cap K)K_0 \subset K$.  
(This is an elementary consequence of the definitions, see \cite[p.~332]{KV}.)
\item[(ii)]
If $\pi$ is an $(\l,L\cap K)$-module such that 
$\mathcal{R}_{(\q,L\cap K)}^{(\g,K),q}(\pi)=\{0\}$ for $q\neq q_0$, and $X$ is any $(\g,K)$-module then 
there is a first quadrant Frobenius reciprocity spectral sequence with 
\begin{equation}
\begin{split}
E^{p,q}_2 &= \Ext_{(\l,L\cap K)}^p\left(H_q\big{(}\u,\mathcal{F}_{(\g,K)}^{(\q,L\cap K)}(X)\big{)},\pi\otimes \wedge^{\dim(\u)}\u\right)  
\\ & \qquad \qquad \Rightarrow  \quad 
\Ext_{(\g,K)}^{p+q-q_0}\left(X, \mathcal{R}_{(\q,L\cap K)}^{(\g,K),q_0}(\pi)\right). 
\end{split}
\label{frss}
\end{equation} 
Here $\mathcal{F}_{(\q,L\cap K)}^{(\g,K)}$ is the forgetful functor, i.e. $\mathcal{F}_{(\q,L\cap K)}^{(\g,K)}(X)$ is simply the notation of \cite{KV} for $X$ considered as a $(\q,L\cap K)$-module by restriction. 
(See \cite[Corollary~5.121(b) and Theorem~5.120(b)]{KV} for (\ref{frss}) or, with slightly different indexing and notation, \cite[Corollary~6.3.4]{V}.)  
\end{enumerate}

Now consider our situation, where $\pi=\pi_L:L \to \C^\times$ is the  unitary character of $L$ on $E^\u$.  For $i\neq S$ we have $\mathcal{R}_{(\q,L\cap K)}^{(\g,K),i}(\pi)=\{0\}$ for $i \neq S$ by combining (i) and Theorems~5.35 and 5.109 of \cite{KV}, by which 
$\mathcal{R}^{(\g,(L\cap K)K_0),i}_{(\q,L\cap K)}(\pi) = \{0\}$ for $i \neq S$ because $\pi$ satisfies the positivity condition 
\begin{equation} 
\mathrm{Re}\, \langle d\pi|_{\t^c} + \rho_\q, \alpha \rangle > 0
\quad \text{ for } \alpha\in \Phi(\t^c,\u), 
\label{positivity}
\end{equation} 
which is easily verified. 
An essential hypothesis for these theorems is that the $\theta$-stable Levi in question meets every component of the group, which is tautologically true for $L\subset L\,G_0$.  

Thus (ii) gives a  spectral sequence for $X=E$.  It  collapses in this situation because the infinitesimal characters of $H_q(\u,E)$ and $\pi\otimes \wedge^{\dim(\u)}\u$ can only match for $q=\dim(\u)$, so the nonzero terms of $E_2$ are all in a single column.  With $R=\dim(\u\cap \p)$, we have (writing 
$E$ for $\mathcal{F}_{(\q,L\cap K)}^{(\g,K)}(E)$ to simplify notation):
\begin{equation}
\begin{split}
H^*\left(\g,K,\mathcal{R}^{(\g,K),S}_{(\q,L\cap K)}(\pi) \otimes E^\vee\right) 
& = \Ext^*_{(\g,K)}(E, \mathcal{R}^{(\g,K),S}_{(\q,L\cap K)}(\pi))
\\
& = \Ext^{*-R}_{(\l,L\cap K)}\left(H_{\dim(\u)}(\u,E), \pi \otimes \wedge^{\dim(\u)}\u\right)
\\
&=
H^{*-R}(\l,L\cap K, \pi\otimes (E^\u)^\vee).  
\end{split}
\notag
\end{equation}
Thus if $\pi$ is the unitary character of $L$ on $E^\u$ then 
$$
H^*\left(\g,K,\mathcal{R}^{(\g,K),S}_{(\q,L\cap K)}(\pi)\otimes E^\vee \right) = 
H^{*-R}(L^u/L\cap K), 
$$  
while if $\pi$ is a unitary character other than the one appearing in $E^\u$ then the cohomology vanishes. 
In particular, $\mathcal{R}^S_\q(\pi)$ is a representation with cohomology with coefficients in $E^\vee$.  The unitarity is a consequence of (i) above and the unitarity of $\mathcal{R}^{(\g,(L\cap K)_0K), S}_{(\q,L\cap K)}(\pi)$, which is available under the positivity condition (\ref{positivity}) by \cite[Corollary~9.70]{KV}.   

The irreducibility of $\mathcal{R}^S_\q(\pi)$ holds under (\ref{positivity}) by \cite[Corollary~8.28]{KV}.

Let us now check the converse, i.e. that any irreducible unitary  $V$ with $H^*(\g,K,V\otimes E^\vee) \neq \{0\}$ 
appears as  in a packet $\Pi_{E,\wL}$.  The restriction $\mathcal{F}_{(\g,K)}^{(\g,K_0)}V$ is $(\g,K_0$)-cohomological, hence admits a nonzero map to some $\mathcal{R}^{(\g,K_0),S}_{(\q,(L\cap K)_0)}(\lambda)$ by \cite{VZ} for connected groups.  By reciprocity for the exact induction functors 
$\mathcal{F}_{(\g,K)}^{(\g,K_0)}$ and $I_{(\g,K)}^{(\g,K_0)}$
there is a nonzero map
$$
V \to I_{(\g,K_0)}^{(\g,K)} \mathcal{R}_{(\q,L\cap K_0)}^{(\g,K),S}(\lambda) 
= I_{(\g, (L\cap K)K_0)}^{(\g,K)} 
I^{(\g, (L\cap K)K_0)}_{(\g,K_0)}
\mathcal{R}_{(\q,L\cap K_0)}^{(\g,K_0),S}(\lambda).  
$$
Now 
$$
I^{(\g, (L\cap K)K_0)}_{(\g,K_0)}
\mathcal{R}_{(\q,L\cap K_0)}^{(\g,K_0),S}(\lambda)
=
\mathcal{R}_{(\q,L\cap K)}^{(\g,(L\cap K)K_0),S} I_{(\q,L\cap K_0)}^{(\q,L\cap K)}(\lambda) 
$$
because of transitivity of induction.  (The functors $I$ are actually cohomological induction functors, only exact.) 
Now $\lambda$ is the unitary character of $L_0$ on $E^\u$, and 
the induction $I_{(\q,L\cap K_0)}^{(\q,L\cap K)}(\lambda)$ is a sum of characters, exactly one of which is the unitary character by which $L$ acts on $E^\u$.  Thus there is a unique summand which has $(\g,(L\cap K)K_0)$-cohomology (by the same computation as above), and hence $V$ must map nontrivially into that summand, which is the irreducible representation $\mathcal{R}_{(\q,L\cap K)}^{(\g,K),S}(E^\u) = I_{(\g,(L\cap K)K_0)}^{(\g,K)} \mathcal{R}^{\g,(L\cap K)K_0,S}_{(\q,L\cap K)}(E^\u)$ in the packet $\Pi_{E,\wL}$.   This concludes the proof of (1) in the proposition. 

(2) 
follows from (1) since  we have 
$$
\mathcal{R}_{(\q,L\cap K)}^{(\g,K)}(\pi) \otimes \varepsilon = \mathcal{R}_{(\q,L\cap K)}^{(\g,K)}(\pi \otimes \varepsilon|_{L/L_0})
$$  
for any finite order character $\varepsilon: G/G_0\to \C^\times$ as a consequence of the definitions; we leave this to the reader to check. 
\end{proof} 

Let us now check that 
Propositions~\ref{Sigmaprop} and \ref{cohprop} combine to prove (1)-(3) of Theorem~\ref{AJ1}.  
Let $\Pi_{E,\wL}$ be an $AJ$-packet as in the theorem.  The fact that its elements are pairwise inequivalent follows in the connected case from consideration of $K$-types (specifically, from \cite[Proposition 6.1]{VZ} or \cite[Proposition 10.24]{KV}).  For the general case we again use the factoring in (\ref{factoring}) and use the assertion of \cite[Theorem 10.44]{VZ} for the $K$-types of 
$\mathcal{R}^{(\g,(L\cap K)K_0),i}_{(\q,L\cap K)}(E^\u)$.  This proves (1), modulo the assertion that other packets are twists by characters of $G/G_0$.  This follows from the  surjectivity of $H^1(\R,Z(\wG)) \to H^1(\R,Z(\wL))$, which 
follows from  Proposition~\ref{uniqueness}. 

The assertions of (2) of the theorem are contained in Proposition~\ref{cohprop}.

For (3) about tempered packets, we note that a representation $\mathcal{R}^S_\q(\pi)$ is tempered if and only if $[\l,\l]\subset \k$ (see \cite{VZ}).    If $\wL=\wT^c$ then $\b \in \Sigma^{-1}(\wT^c)$ is $K$-conjugate to a Borel containing $\t^c$ and hence $[\t^c,\t^c]=0$, so the first statement is immediate.  For the second statement, if $\q$ is such that $[\l,\l]\subset \k$ then 
$\mathcal{R}^S_\q(\pi) = \mathcal{R}^S_{\t^c}(\pi|_{T^c})$, so that the representation already appears in a packet $\Pi_{\wT^c,c}$.

This concludes the proof of Theorem~\ref{AJ1}.

\begin{remark}
The following remarks on infinitesimal characters may help to explain the parametrization of $AJ$-packets.   The remarks here will not be used in the sequel.  

Fix a $\theta$-stable fundamental Cartan $\t^c$ and $\theta$-stable Borel $\b\supset \t^c$ as above. 
Suppose that $(\q, \pi)$ is a pair 
consisting of a $\theta$-stable parabolic $\q=\l+\u$ containing $\t^c$ and a unitary character $\pi:L \to \C^\times$.  
As recalled above, the infinitesimal character of $\mathcal{R}^i_\q(\pi)$ is 
$$
\lambda_\pi + \rho_\q \in W(\sG,\sT^c)\backslash \t^{c*}
$$ 
where $\lambda_\pi =d\pi|_{\t^c}$ and $\rho_\q=\rho(\l)+\rho(\u)$ for $\rho(\l)$ the half-sum of any positive system for $\Phi(\l,\t^c)$.  The infinitesimal character of a finite-dimensional irreducible  representation $E$ is  
$$
\lambda_E + \rho \in W(\sG,\sT^c)\backslash \t^{c*}
$$ 
where $\lambda_E \in X^*(\sT^c)$ is the $\b$-highest weight and $\rho$ is the half-sum of roots in $\Phi(\t^c,\b)$.  The infinitesimal characters of $\mathcal{R}^i_\q(\pi)$ and $E$ agree if and only if  
$\lambda_\pi + \rho_\q$ and $\lambda_E + \rho$ are $W(\sG,\sT^c)$-equivalent. 
If, furthermore, $\mathcal{R}^i_\q(\pi)$ is unitary, then $\theta(\lambda_E) = \lambda_E$ (see e.g. \cite[II.6.13]{BW} or \cite[1.3]{BC}), and assuming the positive system for $\l$ is chosen $\theta$-stable (as it may be), so that $\theta(\rho(\l))=\rho(\l)$, it 
follows  that $\lambda_\pi+\rho_\q$ and $\lambda_E+\rho$ are in fact 
$W(\sG,\sT^c)^\theta$-equivalent. 
Now consider the infinitesimal character equation, in which we think of  $\lambda_E$ as fixed and we need to solve for $\q$ and $\lambda_\pi$ to find derived functor modules $\mathcal{R}_\q^i(\pi)$ with matching infinitesimal character: 
\begin{equation}
\lambda_\pi + \rho_\q = 
\lambda_E + \rho \quad \text{ mod }\ W(\sG,\sT^c)^\theta.
\label{equation:infchar}
\end{equation} 
This has some natural `tautological' solutions whose existence depends only on $\Sigma(\q)$. 
First take $\q$ to be $\b$-standard.  Choosing the positive system for $\l$ given by $\b\cap \l$ we see that $\rho_\q = \rho$ and (\ref{equation:infchar}) has the solution $\lambda_\pi = \lambda_E$, which is valid, i.e. can be the infinitesimal character of a $\pi:L \to \C^\times$, if and only if 
$$
\langle \alpha, \lambda_E \rangle =0 \ \text{ for all }\alpha
\in \Phi(\l,\t^c)
$$   
which is the condition $\Sigma(\q) \in \mathcal{L}_E$ defined above.  
For $w \in W(\sG,\sT^c)^\theta$ and $\q=\q(w\lambda_0)$ where $\q(\lambda_0)$ is $\b$-standard then $\rho_\q = w\rho$ (for the correct choice of $\rho(\l(w\lambda_0))$) and (\ref{equation:infchar}) has the solution $\lambda_\pi=w\lambda_E$, provided that 
$\langle w\lambda_E, \alpha \rangle=0$ for all $\alpha \in \Phi(\l(w\lambda_0),\t^c)$, which is again the condition $\langle\lambda_E, \alpha\rangle =0$ for $\alpha \in \Phi(\l(\lambda_0),\t^c)$. 
(The element $w$ is only determined up to its $W(\sL,\sT^c)$-coset but the solution $\lambda_\pi = w\lambda_E$ does not depend on this.) 
Since $\Sigma(\q(w\lambda_0)) = \Sigma(\q(\lambda_0))$ by definition, this condition can be rewritten as $\langle \lambda_E, \alpha\rangle=0$ for every root $\alpha \in \Delta_{\wG}$ appearing in the Levi of $\Sigma(\q)$, i.e. $\Sigma(\q) \in \mathcal{L}_E$.

The upshot of these observations is that  for each $\wL \in \mathcal{L}_E$, there is a family of $(\g,K)$-modules indexed by $\Sigma^{-1}(\wL)$ having the infinitesimal character of $E$, namely the modules $\mathcal{R}^i_\q(E^\u)$.  
This is precisely the $AJ$-packet $\Pi_{E,\wL}$.   In the connected case (so that any $\theta$-stable Levi $L$ is connected, see \cite[Lemma~V.5.10]{KV}), this the only choice, while if $L$ is disconnected one could also twist $E^\u$ by a finite-order character of $L$, which leads to general packets.   In fact, a general theorem of Salamanca-Riba \cite[Theorem 1.8]{SR} shows (for $G$ connected) that any irreducible unitary $(\g,K)$-module which shares the infinitesimal character of a finite-dimensional representation of $\sG(\C)$ is necessary of the form $A_\q(\lambda) = \mathcal{R}^S_\q(\C_\lambda)$ for a $\theta$-stable $\q$ and unitary character $\lambda$.   Given this, we see that $AJ$-packets are the natural solution to the question of parametrizing the irreducible unitary representations with 
the infinitesimal character of a finite-dimensional representation of $\sG(\C)$ in terms of data in the dual group.

\end{remark}

\section{Adams-Johnson packets:  Examples} \label{AJexam}

This section discusses examples of Adams-Johnson packets for some groups, with a view to illustrating the results of the previous section and some results in later sections.  These examples are based on computations with the parametrization of $\theta$-stable parabolics (cf. Proposition \ref{Sigmaprop} and its proof) 
and we do not give all details in each case.

A basic feature of the parametrization of cohomological representations by $\theta$-stable parabolic subalgebras (for groups other than $GL(n,\R)$) is that it has redundancies: It may happen that $A_{\q_1}$ and $A_{\q_2}$ are isomorphic while $\Sigma(\q_1)\neq \Sigma(\q_2)$ (so that, in particular, $\q_1$ and $\q_2$ are not $K$-conjugate).  
This forces different $AJ$-packets of representations to overlap, and sometimes even coincide. 
This is already seen in the example of  compact groups (see Example \ref{example:compact} below), but this degeneracy propagates to noncompact groups.   In fact there is a simple general criterion for when this happens \cite[Proposition 1.11]{SR}:  The derived functor modules $\mathcal{R}^S_{\q}(\pi)$ and $\mathcal{R}^S_{\q'}(\pi')$, for $\pi, \pi'$ unitary characters of $L, L'$ satisfying the necessary positivity condition (\ref{positivity}) to contribute to $AJ$-packets, are isomorphic if and only if (1) $\u\cap \p = \u'\cap \p$ and (2) $\pi|_{L\cap L'} = \pi'|_{L \cap L'}$. 
This criterion, together with the explicit parametrization of $\theta$-stable parabolics modulo $K$-conjugacy,  are used in the examples below.

\begin{example}\label{example:compact} 
(The compact case)  
Let $\sG$ be connected, reductive,  and anisotropic over $\R$, so that $G=\sG(\R)$ is a compact connected group.  
Let $\t_0$ be a Cartan subalgebra and $\b \supset \t$ a Borel subalgebra of $\g$.  Since $\sG$ is compact, $\theta$ is the identity automorphism, and every $\q \supset \b$ (including $\b$ itself) is a $\theta$-stable subalgebra for $\g_0$  (complex conjugation takes each root to its negative,  so that the opposition condition is vacuously true; alternately we can use Prop.~11 or Remark~\ref{dscase} on the group level).  Thus any standard Levi  subgroup of $\wG$ is $\omega_{\wG}\cdot j$-invariant, i.e. belongs to $\mathcal{L}$.   The condition $\dim E^\u=1$ picks out a set of standard Levi subgroups $\mathcal{L}_E$ in $\wG$ and hence a set of  parabolic subalgebras for $\g_0$ invariant under conjugation by $\sG(\C)$.  For each $\wL \in  \mathcal{L}_E$  in $\wG$ the fibre $\Sigma^{-1}(\wL)$ is a singleton, and choosing $\p = \mathfrak{m}+\mathfrak{n}\in \Sigma^{-1}(\wL)$ gives the singleton $AJ$-packet 
$$
\Pi(\psi)=\{E=\mathcal{R}^{\dim \n}_\mathfrak{p}(E^\n)\} 
$$  
(for any $\psi$ with $\psi_\C$ corresponding to $\wL$). 
This construction of $E$ by cohomological induction from the one-dimensional representation of the Levi on $E^\n$ is essentially Kostant's  Borel-Weil theorem (see \cite[IV.9, IV.11]{KV}), especially \cite[Prop.~4.173]{KV}).  

Note that the representation $E$ has many $A$-parameters.  This degeneracy is necessary to account for global  automorphic forms on groups which are compact at the real places.  (Otherwise, automorphic representations on $\sG(\A)$ with $\sG(\R)$ compact would be tempered!)

\end{example}

\begin{example}\label{dscase} 
(The discrete series case)  Suppose that the derived group of $\sG$ has a discrete series.  In this case every standard Levi subgroup of $\wG$ is  $\omega_{\wG}\cdot j$-invariant because this element acts as $t \rightarrow t^{-1}$ on $\wT$.  Thus, up to twisting by characters of $G/G_0$, the $AJ$-packets are parametrized by the parabolics of the dual group. 

The criterion for coincidence of different $A_{\q}$'s above shows that if $\mathsf{G}(\R)$ has discrete series,  and $\q$ is such that $\mathsf{L}(\R)$ is compact, then $A_\q$ is a discrete series representation, though $\Sigma(\q)$ need not be the maximal torus $\wT^c$ in $\widehat{\mathsf{G}}$.   
On the other hand, the same discrete series representation must appear as $A_{\b}$ for a suitable $\theta$-stable Borel subalgebra $\b$, for which $\Sigma(\b)$ is the maximal torus $\wT^c$ in $\wG$.

\end{example}

\begin{example} \label{spehq} (The Speh representation)  Let  $G=\GL(2n,\R)$. 
A simple example which might help orient the reader in  Proposition~\ref{Sigmaprop} is the $\theta$-stable parabolic for $\GL(2n,\R)$ which gives the Speh representation.  The Dynkin diagram of $\g$ with respect to a  fixed fundamental $\theta$-stable Cartan and $\theta$-stable Borel subgroup is the diagram $A_{2n-1}$, with $\theta$ acting by  $\alpha_i \leftrightarrow \alpha_{2n-i}$ for $1 \leq i \leq n$.  (In this case $\gamma$ is trivial, as it is for any (inner form of a) split group, and $\iota=\theta$.)   The standard $\theta$-stable parabolic subalgebra $\mathfrak{q}$ with Levi $\mathfrak{l}$ containing exactly the simple roots $\{\alpha_i:i \neq n\}$ has $\mathfrak{l} \cong \mathfrak{gl}(n,\C)^2, \mathfrak{l}_0\cong \mathfrak{gl}(n,\C)$, and $\mathsf{L} = {\rm R}_{\C/\R}\GL(n)$.  (For explicit matrix descriptions of these we refer to Example 1 on p.~586 of \cite{KV}.) 
The dual group is $\widehat{\mathsf{G}} = \GL(2n,\C)$, and for the standard choice of Borel subgroup,  $\Sigma(\q)$ is the $(n,n)$ parabolic and $\widehat{\mathsf{L}}=\GL(n,\C)^2$ is the $(n,n)$ Levi.  The representation $\mathcal{R}^S_\q(\C)$ is the Speh representation, which is cohomological for trivial coefficients and constitutes a singleton packet, and its  $A$-parameter is $\sigma_n \otimes [n]$.  \end{example}

\begin{example} Let $G=\GL(N,\R)$. 
The $\theta$-stable parabolic subalgebras $\q$  for $\GL(N,\R)$ for general $N$ are classified in terms of $\theta$-stable subsets of the simple roots with respect to a fundamental Cartan $\t^c$ and $\theta$-stable Borel $\b$ (see Proposition~\ref{para}).  Since the $\theta$-stable Borel is unique up to $K$-conjugacy for $\GL(N,\R)$, all $AJ$-packets are necessarily singletons. 
Up to $K$-conjugacy, a $\theta$-stable parabolic $\q$ is given by an ordered partition 
$$
(n_d, \cdots, n_1, n_0, n_1, \cdots, n_d)
$$ 
with $n_0 + 2\sum_{k\geq 1} n_k=N$ and the Levi is 
$$
\sL= \GL(n_0) \times \prod_{k\geq 1} {\rm Res}_{\C/\R}\GL(n_k)
$$ 
with $\wL = \GL(n_0,\C) \times \prod_{i\geq q} \GL(n_i,\C)^2$. 
The corresponding $AJ$-packet $\Pi_{\C,\wL}$ contains the single representation $\mathcal{R}^S_\q(\C)$ (or more generally, $\Pi_{E,\wL}=\{\mathcal{R}^S_\q(E^\u)\}$ if $\wL \in \mathcal{L}_E$), which is irreducible in this case.
The $AJ$-packets are, as expected, singletons, and agree with $L$-packets. 

For $\GL(N,\R)$, the problem of parametrizing packets by $A$-parameters (mentioned in Remark \ref{parametergeneral}) has an obvious solution, which we will expand on in Section~\ref{gln} below. 
\end{example}

\begin{example}   
  Let 
  $G = \GL(n,\qH)$, where $\qH$ denotes Hamilton's quaternions.
This is an inner form of $\GL(2n,\R)$ so  ${}^L\sG=W_\R \times \GL(2n,\C)$.     
The number of self-associate standard Levi subgroups in $\wG=\GL(2n,\C)$  is $2^n$, corresponding to the set of  self-dual ordered partitions of $2n$.  As in the case of $\GL(n,\R)$,  there is a unique $\theta$-stable Borel subalgebra up to $K$-conjugacy.  This ensures that the fibres of $\Sigma$, and hence the $AJ$-packets, are singletons.  

The $\theta$-stable standard Levi subgroups are represented by the subgroups \[\sL(\R) = \GL_{n_1}(\C) \times \cdots \times \GL_{n_k}(\C) \times \GL_m(\qH),\]
  with $n_1+\cdots+ n_k+m=n$, where we allow $m$ to be zero.  
  Let \[\Si^1 \subset \underbrace{\C^\times \times \cdots \times \C^\times}_{(n-m)- {\rm times}}  \subset \underbrace{\qH^\times \times \cdots \times \qH^\times}_{(n-m)- {\rm times}}
  \times \underbrace{\qH^\times \times \cdots \times \qH^\times}_{m- {\rm times}},\]
  where $\Si^1 \subset \underbrace{\C^\times \times \cdots \times \C^\times}_{(n-m) = (n_1+n_2+\cdots+n_k) -{\rm times}} ,$ via the embedding
  \[z\rightarrow (  \underbrace{z \times \cdots \times z}_{n_1- {\rm times}}
  \times  \underbrace{z^3 \times \cdots \times z^3}_{n_2- {\rm times}} \times  \underbrace{z^5 \times \cdots \times z^5}_{n_3- {\rm times}},\cdots)\] 
  and $\underbrace{\C^\times \times \cdots \times \C^\times}_{(n-m)- {\rm times}}  \subset \underbrace{\qH^\times \times \cdots \times \qH^\times}_{(n-m)- {\rm times}}
  \times \underbrace{\qH^\times \times \cdots \times \qH^\times}_{m-{\rm times}},$ via the standard embedding of $\C^\times$ in $\qH^\times$ (and trivial in the last $m$-copies of $\qH^\times$).  It is easy to see that for the embedding of $\Si^1 \subset \GL_n(\qH)$ so constructed, the centralizer of $\Si^1$ is the subgroup $\sL(\R) = \GL_{n_1}(\C) \times \cdots \times \GL_{n_k}(\C) \times \GL_m(\qH),$ and the subgroup $\sP$ of $\sG(\C)$ on which   $\Si^1$ operates via $z^d, d \geq 0$ has the properties
  (1) and (2) desired above.

 This enumeration of $\theta$-stable standard Levi subgroups shows that they fall into two types, those with real points of the form 
$\GL(n_1,\C) \times \cdots \times \GL(n_k,\C)$
where $\sum n_i = n$ and those with real points of the form 
$\GL(n_1,\C) \times \cdots \times \GL(n_k,\C) \times \GL(m,\qH)$
where $\sum n_i + m=n$ and $m>0$.   There are $2^{n-1}$ of each type.
However, when one considers the representations $A_\q=\mathcal{R}^S_\q(\C)$ of  $G = \GL(n,\qH)$,  there are coincidences.  In \cite[Section 3]{SW} it is checked that the module of the first kind for the partition $(n_1,\dots,n_{k-1},1)$, i.e. with $n_k=1$ and that of the second kind with partition $(n_1, \cdots, n_{k-1},1)$ with $m=1$ are isomorphic, and these are the only coincidences.  (The simplest example of this is when $n=2$, when the Levi $\C^\times \times \C^\times$ and the Levi $\C^\times \times \mathbb{H}^\times$ give the same representation, essentially because the two Levis differ by a compact-mod-centre factor.)  Thus there are $2^{n-1}+2^{n-2}$ isomorphism classes of representations and $2^n$ nonempty singleton $AJ$-packets, and necessarily there must be coincidences among 
$AJ$-packets.

In particular, one sees that there is an $AJ$-packet which consists of a single tempered representation but which is not the tempered cohomological packet.  This is the packet given by the Levi $\C^\times \times \cdots \times \C^\times \times \mathbb{H}^\times$, which consists of the single fundamental series representation,  which is also the packet given by $\C^\times \times \cdots \times \C^\times$, i.e. by the tempered parameter. 
(Once again, this sort of degeneracy was already seen for compact groups, where all nonempty packets coincide.)

\end{example}

\begin{example} (Unitary groups)  Let $G=\U(p,q)$ with $p+q=N$.  
The $L$-group is $\wG=\GL(N,\C)$ and since $G$ contains a compact maximal torus we know that all parabolics in $\wG$ give cohomological $A$-parameters and $AJ$-packets. 
Let $T=\U(1)^N$ be the diagonal torus.   It is a simple matter to enumerate the $\theta$-stable parabolic subalgebras $\q$ containing $\t$ using the fact that they are $\q=\q(\lambda_0)$ for $\lambda_0 \in i\t_0^*$ (for the notation as in the proof of Proposition~\ref{Sigmaprop}).
We will not go into detailed computations, but simply note the final result:  The $K$-conjugacy classes of $\theta$-stable parabolic subalgebras correspond bijectively to pairs of ordered expressions in nonnegative integers 
\begin{equation}
\begin{array}{l}
p = p_1 +\cdots + p_r
\\
q= q_1+ \cdots + q_r
\end{array}
\quad \text{with $p_i + q_i \neq 0$ for all $1\leq i \leq r$}. 
\label{Upqpartitions} 
\end{equation}
By an abuse of terminology (since zero entries are allowed), we will refer to these as pairs of partitions. 
The Levi $L$ corresponding to the pair $((p_i),(q_i))$ is $\prod_{i=1}^r \U(p_i, q_i)$ and the map $\Sigma$ 
sends the corresponding $K$-conjugacy class of $\theta$-stable parabolics to the Levi $\prod_{i=1}^r \GL(p_i+q_i, \C)$ in $\wG=\GL(N,\C)$, equivalently, to the ordered partition $(p_1+q_1)+ \cdots +(p_r+q_r)=N$.

In general, a pair of partitions as in (\ref{Upqpartitions}) gives the same cohomological representation as a pair with the property that neither partition has two adjacent zeroes.   
For example, 
$\begin{array}{l}2 = 1+1+0 \\1= 0+0+1\end{array}$ and 
$\begin{array}{l}2 = 2+0 \\1= 0+1\end{array}$ give the same representation of $\U(2,1)$.   (This follows from the fact that if $\q=\l+\u$ and $\q'=\l'+\u'$ are the $\theta$-stable parabolics corresponding to two such pairs of partitions, then $\q \subset \q'$ and $\u$ and $\u'$ contain the same imaginary noncompact roots, i.e. $\u\cap \p = \u'\cap \p$ and hence $A_\q\cong A_{\q'}$.)  
    This reduces the number of pair of partitions to consider if one is only interested in enumerating isomorphism classes of cohomological representations.  However,
    it is essential to consider all expressions if one wants to parametrize the $AJ$-packets: 
The $AJ$-packet corresponding to an ordered partition $n_1+\cdots+n_k = N$ is in bijection with all pairs of partitions as in (\ref{Upqpartitions}) with the associated partition of $N$ being $(n_1,\dots,n_k)$ for $n_i=p_i+q_i$.

The discrete series representations of $G=\U(p,q)$ are exhausted by the pairs of partitions in which all $p_i$ and $q_i$ belong to $\{0,1\}$ and $p_i+q_i=1$ for all $i$.  They are therefore $\binom{N}{p}=\binom{N}{q}$ in number.  More generally (by the preceding observation), a representation given by a pair of partitions is discrete series if and only if the partition satisfies $\sum_ip_iq_i=0$.

\end{example}

\begin{example} \label{Un1example}
Let $G=\U(n,1)$.  The cohomological representations of $\U(n,1)$ are particularly simple in terms of the partitions (\ref{Upqpartitions}) above, since at most one $q_i$ can be nonzero, and then it is equal to one.  
Using the observation in the previous example, we see that to get all isomorphism classes of cohomological representations it suffices to consider pairs of the form 
\begin{equation}
\begin{array}{l}
n = a+b+c
\\
1= 0+1+0
\end{array}
\label{Un1partitions} 
\end{equation}
where we allow $a=0$ (respectively, $c=0$), in which case this is understood as 
$\begin{array}{l} n = b+c\\1= 1+0 \end{array}$ 
(respectively, as $\begin{array}{l} n = a+b\\1= 0+1 \end{array}$). 
The pairs with $b=0$ are discrete series representations and there  are $n+1$ of them, forming a single packet.  The remaining representations are $n(n+1)/2$ in number and  pairwise inequivalent. 

This fits well with the classification of cohomological representations in \cite[IV.4]{BW}:  For each pair $(i,j)$ of nonnegative integers there is a unique cohomological representation $J_{ij}$ for which the nonvanishing cohomology groups are 
$$
H^{i+l,j+l}(\g,K,J_{ij}) = \C 
\quad \text{ for } 0\leq l \leq n-(i+j), 
$$
where the bigrading is the Hodge bigrading coming from the complex structure on $\g/\k$.  
The discrete series are $J_{ij}$ with $i+j=n$, while the trivial representation is $J_{00}$.  (The notation in \cite{BW} does not allow for $i+j=n$ and the discrete series are denoted $D_k$, but this notation is more natural.) 
The relation between the two classifications is given by 
$$
\begin{array}{l}
n = a+b+c
\\
1= 0+1+0
\end{array}\longleftrightarrow J_{ca}.
$$ 
The description of the $J_{ij}$ as Langlands quotients shows that each $J_{ij}$ for $i+j <n$ is a singleton $L$-packet.   

In this case there is a simple way to think about $AJ$-packets using the bigrading of cohomology and  the Hodge diamond, by thinking of each representation as being a point in the Hodge diamond, with $J_{ij}$ corresponding to $(i,j)$.  Thus the discrete series correspond to the middle row of vertices $(n,0), (n-1,1), \dots, (0,n)$ and the trivial representation $J_{00}$ corresponds to the bottom vertex. 
Then the $AJ$-packet corresponding to an ordered  partition $n_1 + \cdots + n_k = n+1$ consists of those $J_{ab}$ for which, when one looks at the Hodge diamond of $J_{ab}$ with lower vertex at $(a,b)$, their intersection with the middle row $\{(n,0),\cdots, (0,n)\}$ (corresponding to the discrete series) gives a decomposition of that row into non-intersecting subsets which are intervals of length $n_1, n_2, \dots, n_k$.

For example, if $n=5$ we have the following picture for the partition $3+1+2=6$, where we have only drawn the lower half of the Hodge diamonds for simplicity: 
$$
\begin{tikzpicture}[scale=1.5]
    \foreach \x in {-2.5,...,2.5}
    \foreach \y in {-.5,...,1.5}
    {
  \fill[black] (\x,\y) circle (0.04cm);
    }       
    \foreach \x in {-2,...,2}
    \foreach \y in {-1,...,1}
    {
  \fill[black] (\x,\y) circle (0.04cm);
    }      
\draw[white,fill=white] (.5,-1.05)--(3.05,-1.05)--(3.05,1.55);
\draw[white,fill=white] (-.5,-1.05)--(-3.05,-1.05)--(-3.05,1.55);
     \draw[dashed,black] (-1.5,.25) -- (-2.8,1.6);
     \draw[dashed,black!50] (-1.5,.25) -- (-.2,1.6);
     \draw[dashed,black] (.5,1.25) -- (.2,1.6);
     \draw[dashed,black] (.5,1.25) -- (.8,1.6);
     \draw[dashed,black] (2,.75) -- (1.2,1.6);
     \draw[dashed,black] (2,.75) -- (2.8,1.6);     
	\node[scale=.75] at (-2.5, 1.75) {$(5,0)$};	
	\node[scale=.75] at (-1.5, 1.75) {$(4,1)$};	
	\node[scale=.75] at (-.5, 1.75) {$(3,2)$};	
	\node[scale=.75] at (.5, 1.75) {$(2,3)$};	
	\node[scale=.75] at (1.5, 1.75) {$(1,4)$};	
	\node[scale=.75] at (2.5, 1.75) {$(0,5)$};	
	\node[scale=.75] at (-1.5, .75) {$(3,0)$};	
	\node[scale=.75] at (2.0, 1.25) {$(0,4)$};	
	\node[scale=.75] at (0,-.75) {$(0,0)$};	
\draw[black,fill=black] (-1.5,0.5) circle (0.06cm); 
\draw[black,fill=black] (0.5,1.5) circle (0.06cm); 
\draw[black,fill=black] (2,1) circle (0.06cm); 
\end{tikzpicture}
$$
The $AJ$-packet is $\{J_{30}, J_{23}, J_{04}\}$, of which one ($J_{23}$) is a discrete series representation.

We will refrain from writing out a rule for this as the picture involving the Hodge diamond is quite appealing.  The justification for this picture of $AJ$-packets will be given in Section~\ref{coho} below, see especially Theorem~\ref{mirror}.   The following fact is a consequence:  The $A$-packet  corresponding to an ordered partition of $n+1$ contains a discrete series representation if and only if the partition contains a $1$ (see Remark~\ref{Un1mirror}). 

\end{example}

\begin{example}
We give an example of an $AJ$-packet for the non-quasisplit group $\U(3,1)$ which contains two $L$-packets.   Of course, the discussion in the previous example shows clearly that an $AJ$-packet can contain more than one $L$-packet, but in this example we see that there is no preferred choice. 
In $\U(3,1)$, take the packet corresponding to the partition $4=2+2$ (i.e. the standard parabolic in $\wG$ with Levi $\GL(2,\C) \times \GL(2,\C)$).   There are two representations in the $AJ$-packet, which correspond to the pairs of partitions  
$$
{\begin{array}{c} 3=2+1 \\ 1=0+1 \end{array}}
\quad \text{and} \quad \small{\begin{array}{c} 3=1+2 \\ 1=1+0 \end{array}}, 
$$ 
each of which constitutes a singleton $L$-packet.   There is no preferred choice of $L$-packet in the $AJ$-packet.  

\end{example}

\section{Tempered cohomological representations} \label{cohom-para}
Although most cohomological representations are nontempered, there are always tempered cohomological
representations for any reductive group $\sG(\R)$.  Their parameters particularly  simple to describe as the
following theorem asserts which is an immediate consequence of Theorem \ref{cohomo} on noting that tempered 
cohomological representations have tempered Langlands parameter with the infinitesimal
character of a finite dimensional representation of $\sG(\C)$, therefore they have unique
Langlands parameter (up to     twisting by $H^1(\R,Z(\wG))$), by Theorem \ref{cohomo} (in which to have tempered $A$-parameter, the only option is to have
$\wL=\wT$).
The existence of $AJ$-packets of tempered representations is addressed in (3) of Theorem~\ref{AJ1}.
This can also be seen directly using the fact that the fundamental series representations are induced
from the discrete series packet on the fundamental Levi (the induction is irreducible) and
generalities about the cohomology of induced representations (\cite[III.3]{BW}).

\begin{thm} \label{tempered}
 Let $\mathsf{G}$ be a  connected reductive algebraic group defined over $\R$. 
 Then for any fixed coefficient system $V\cong V^\theta$, a finite dimensional irreducible representation of
 $\sG(\C)$, there is a cohomological tempered $L$-parameter $\sigma$ for $\mathsf{G}(\R)$, and
 up to     twisting by $H^1(\R,Z(\wG))$,
 there is at most one
 cohomological tempered $L$-parameter $\sigma$ for $\mathsf{G}(\R)$.
The cohomological tempered $L$-parameter $\sigma$ for the trivial coefficient system is
  given by the following commutative diagram:
$$
\xymatrix{
W_\R \ar[r]^-{\sigma} 
\ar[rd]^-{\sigma_1} & {}^L\sG\\
& \SL_2(\C)\times W_\R. \ar[u]_-{\Lambda}
}
$$
  Here $\Lambda: W_\R \times \SL_2(\C) \rightarrow {}^L \sG$ 
 restricted to $\SL_2(\C)$ is the principal 
 $\SL_2(\C) $ inside $\wG$
 (associated with the pinning on $\wG$, hence commuting with the action of $W_\R$ on $\wG$), 
and  the parameter $\sigma_1: W_\R \rightarrow \SL_2(\C) \times W_\R$ corresponds to the lowest discrete
series representation $D_2$ of $\PGL_2(\R)$, and is the 2 dimensional irreducible representation of
$W_\R$ given by ${\rm Ind}_{\C^\times}^{ W_\R} (z/\bar{z})^{1/2}$. 
\end{thm}

The same proof works to prove Theorem \ref{coh-para} once we have noted that $\Lambda$ and $T(\Lambda)$ have the same
infinitesimal character.

\section{Complex groups} \label{Cgroup}

 Complex groups $\sG(\C)$ where $\sG$ is a reductive group over $\C$ can be treated as a real group via restriction of scalars. In this section
 we discuss how the classification of cohomological representations of real reductive groups of earlier sections translates for $\sG(\C)$.
 One of the issues that we discuss here is that unlike real groups where there is a requirement of parabolics being ``self-associate'',  there
 is no such requirement for complex groups. (Of course, one could also say that $W_\C$ does not have the `symmetry' forced by $j \in W_\R$, but
 we want to argue via real groups!) Cohomological representations of complex groups are rather well understood due to the work
 of Enright [En] which we recall at the end of the section, making it more precise.

 For the formalism of $L$-groups, it is better to begin with a real reductive group $\sG$ (defined over $\R$),  whose complex points
 give rise to a fixed reductive group over $\C$. There are many choices of real reductive groups which give rise to the same
 group over $\C$. For example we could take a split group over $\R$, or we could also take a compact group over $\R$.

 Given an $L$-group ${}^L \sG = \wG \rtimes W_\R$, there is the notion of restriction of scalars of this $L$-group, to be denoted
 as $R_{\C/\R}({}^L \sG) = (\wG \times \wG) \rtimes W_\R$ where the action of $W_\R$ on $\wG \times \wG$ is via
 $\Gal(\C/\R) = \langle \sigma \rangle $
 which operates as
 $$\sigma (g_1,g_2) = (\sigma g_2, \sigma g_1 ), \hspace{1cm} (g_1,g_2) \in \wG \times \wG ,$$
 where $\sigma g_1, \sigma g_2$ denotes the action of $W_\R$ via 
 $\Gal(\C/\R) = \langle \sigma \rangle $ on $\wG$.

 Let's prove the following well-known lemma, a form of  Shapiro's lemma in this context.

 \begin{lemma} \label{shapiro}
   The set of admissible homomorphisms from $W_\R$ to $R_{\C/\R}({}^L \sG) = (\wG \times \wG) \rtimes W_\R$ up to equivalence
   is in a natural bijective correspondence with
   the set of admissible homomorphisms from $W_\C=\C^\times$ to $ \wG$ up to equivalence; the
   natural bijection being just the restriction from $W_\R$ to $W_\C$, and projecting from $\wG \times \wG $ to $\wG$ by the first factor. 
 \end{lemma}
 \begin{proof}It suffices to give an `inverse' to the restriction map, i.e., given a homomorphism 
$\phi: \C^\times \rightarrow \wG$, to construct an admissible homomorphism $\Phi: W_\R \rightarrow R_{\C/\R}({}^L \sG) = (\wG \times \wG) \rtimes W_\R$
   whose restriction to $\C^\times$ (and projected to the first factor) is the map $\phi$ we started with.

   The map $\Phi$ is defined as 
   \begin{eqnarray*}
   \Phi:  W_\R & \lrta &  (\wG \times \wG) \rtimes W_\R \\
   z \in \C^\times & \lrta & (\phi(z), \sigma \phi(\bar{z})) \cdot z \\ 
   j & \lrta  & (\phi(-1), 1) \cdot j,
   \end{eqnarray*}
and one checks easily that we have an admissible homomorphism
$ \Phi:  W_\R  \rightarrow   (\wG \times \wG) \rtimes W_\R$ that one wanted.
\end{proof}   

 It follows from  Lemma \ref{shapiro} that   the set of admissible homomorphisms (up to equivalence)
 from $W_\R$ to $R_{\C/\R}({}^L \sG) = (\wG \times \wG) \rtimes W_\R$
 is independent of the real group $\sG(\R)$ one begins with to define $\sG(\C)$. We also note the following
 well-known corollary on the infinitesimal character of representations of complex groups.

\begin{cor}
  If $\sG$ is a split reductive group over $\R$, and $\pi$ an irreducible representation of $\sG(\C)$ with
  Langlands parameter $\sigma: \C^\times \rightarrow \wG$ given by:
  \[\sigma(z) = z^{\lambda} \bar{z}^\mu
          {\rm ~~ where~~} \lambda , \mu \in X_\star(\wT)\otimes \C  {\rm ~~with~~} \lambda- \mu \in X_\star(\wT), \]
  then the infinitesimal
 character of the representation $\pi$ is the pair $\lambda , \mu \in X_\star(\wT)\otimes \C$ up to $W \times W$-conjugacy, where $W$ is
 the Weyl group  of $\sG(\R)$.   \end{cor}
 \begin{prop}
   The set of conjugacy classes of self-associate parabolics in  $R_{\C/\R}({}^L \sG) = (\wG \times \wG) \rtimes W_\R$
   is in bijective correspondence with the set of all parabolics in $\wG$. 
   \end{prop}
 \begin{proof} A parabolic in $\wG \times \wG$ is of the form $P= P_1 \times P_2$. For this parabolic, we have $jPj = (\sigma P_2,\sigma P_1)$,
   and therefore if $P$ has to be conjugate to its opposite by an element in $ (\wG \times \wG) \cdot j$, then $ \sigma P_2$
   must be a conjugate of $P_1^{-}$ and  $ \sigma P_1$
   a conjugate of $P_2^{-}$; the second condition is clearly a consequence of the first, and thus any parabolic $P$
   in $ \wG \times \wG$  which is conjugate to its opposite, can be written  as $P= P_1 \times (\sigma P_1)^{-}$, and
   conversely, such parabolics are self-associate, proving the proposition. \end{proof}

 Recall that
 for a complex group $\sG(\C)$, the irreducible coefficient systems $V$ for which there are unitary
 representations with cohomology,
 are of the form $V= V_\lambda \otimes \bar{V}_\lambda^\vee$ where $V_\lambda$ is an algebraic highest weight module for $\sG(\C)$
 with highest weight $\lambda$
 and $\bar{V}_\lambda$ denotes its complex conjugate representation. (This is the class of irreducible
 representations of $\sG(\C)$ invariant under a  Cartan involution $\theta$ on $\sG(\C)$. To see
 this,  observe that if  $V_\lambda$ is an algebraic highest weight module for $\sG(\C)$ with highest weight $\lambda$, then we can assume that the Cartan involution on $\sG(\C)$ is the restriction 
 of the Cartan involution  $g\rightarrow {}^t\bar{g}^{-1}$ on $ \GL(V_\lambda)$.)

 We note that in $V_\lambda$, there is a highest weight vector $v_\lambda$ on which a fixed Borel subgroup $\sB(\C)$
 containing a fixed maximal torus $\sT(\C)$  acts by the character $\lambda: \sB(\C) \rightarrow \C^\times$. However, the stabilizer of the line $\langle v_\lambda \rangle$ inside the projective space $\PP(V_\lambda)$ could be bigger than $\sB$; call the stabilizer $P_\lambda \supset \sB$.   
 The parabolic group $P_\lambda$ comes equipped with a character $P_\lambda \rightarrow \C^\times$ (by which  $P_\lambda$  operates on
 $v_\lambda$). By abuse of notation, we call this character on $P_\lambda$ also to be $\lambda$ (as $\lambda$ restricted
 to $\sT(\C)$ determines the character on $P_\lambda$). If $P$ is a parabolic contained in $P_\lambda$, then $\lambda$ restricts
 to give a character on $P(\C)$ too.
  Thus the character
 $\lambda: \sT(\C) \rightarrow \C^\times$ defines a character on all $P \subset P_\lambda$, which we will
 denote by $\lambda_P: P(\C) \rightarrow \C^\times$.

 Here is the classification of cohomological representations  for complex groups. Unlike
 the results for other real groups which are in terms of $A$-parameters,  this theorem is in terms
 of  irreducible admissible representations of the group as the $AJ$-packets are singletons (although general
 $A$-packets for groups other than $\GL_n(\C)$ need not be).
 This theorem is proved in \cite{En}, see also Theorem 7.2 in \cite{Sc}, although a knowledgeable reader may find
 our formulation different from those in these sources.

\begin{thm} \label{En}
 Let $\sG(\C)$ be any connected reductive group over $\C$, and $V$ a fixed coefficient system
  (a finite dimensional irreducible representation of $\sG(\C)$) of the form
  $V= V_\lambda \otimes \bar{V}^\vee_\lambda$ where $V_\lambda$ is an algebraic highest weight module
  with highest $\lambda$. 
  Then there is a bijective map between irreducible unitary cohomological representations of $\mathsf{G}(\C)$
  with coefficients in $V$ and the  set of standard
    parabolics in $\sG(\C)$ contained in  $P_\lambda$. For a parabolic $P \subset P_\lambda$, with the modulus
    function $\delta_P: P(\C) \rightarrow \R^+$, the associated unitary cohomological representation is
    $${\rm Ind}^{\sG(\C)}_{P(\C)}[ D(\delta_P) \cdot \lambda_P(z/\bar{z})],$$
    where $D(\delta_P)$ is the unitary character on $P$ constructed below, and $\lambda_P$
    is the character on $P(\C)$ constructed above. In particular, there is a non-tempered
    cohomological representation if and only if $\lambda$ is not regular.
\end{thm}

In the rest of the section, we construct the unitary character $D(\delta_P)$ on $P$ from the positive  character $\delta_P$ used above, done in greater generality of any complex reductive group. 

First, for a complex torus $\sT(\C) = X_\star(\sT) \otimes \C^\times$, with character group $X^\star(\sT)$, and co-character group $X_\star(\sT)$,
note that any topological character $\chi$  on $\sT(\C)$ is of the form
$$ t \in \sT(\C) \longrightarrow t^\lambda \cdot \overline{t^{\mu}}, \hspace{.7 cm}
\lambda, \mu \in X^\star(\sT) \otimes \C, \hspace{.7 cm}
( \lambda -\mu) \in X^\star(\sT).$$
A characters  $\chi: \sT(\C) \rightarrow \C^\times$ has the property that
$\chi (\sT(\C)) \subset \R^+$ precisely when  $\lambda = \mu$.
For such characters $\chi$ with  $\lambda = \mu$, define another character $D(\chi)$ by
$$ t \in \sT(\C) \longrightarrow t^\lambda/ \overline{t^{\lambda}}, \hspace{1 cm}  \lambda  \in X^\star(\sT) \otimes \C,$$
where for the definition of $D(\chi)$ to be meaningful, we must have:
$2 \lambda \in X^\star(\sT)$.
Positive characters $\chi$ on $\sT$ for which $D(\chi)$ is meaningful, i.e.,  $2 \lambda \in X^\star(\sT)$, will be called
{\it positive algebraic characters} on $\sT$.

The resulting association gives a bijective correspondence between positive algebraic characters on $\sT(\C)$, and all unitary
characters on $\sT(\C)$, denoted as 
$$\chi\longleftrightarrow D(\chi).$$

Now, let $\sG$ be a general complex reductive group with maximal toral quotient $\sT$. Any topological homomorphism $\sG(\C) \rightarrow \C^\times$ factors through $\sT(\C)$ for which we can do the previous construction to get a
duality among certain  positive algebraic characters on $\sG(\C)$ and all unitary characters on $\sG(\C)$, which
also will be denoted by
$$\chi\longleftrightarrow D(\chi).$$ 

\section{Cohomological parameters for $\GL_n(\R)$, $\GL_n(\C)$, $\U_{p,q}(\R)$} \label{gln}

This section merely summarizes the  cohomological parameters for the groups $\GL_n(\R)$,  
$\GL_n(\C)$, $\U_{p,q}(\R)$.

We begin by giving a complete classification of cohomological (unitary) parameters for $\GL_n(\R)$.
Cohomological representations for $\GL_n(\R)$ were constructed  by Speh \cite{Sp} for trivial coefficients. A precise reference for general coefficients is not easy to find in the literature, but  is of course a consequence of the work of Vogan-Zuckerman \cite{VZ}, and was
already discussed from this perspective in Example \ref{spehq} of section \ref{AJpack}.

We will denote the unique $m$-dimensional irreducible representation of $\SL_2(\C)$ by $[m]$. Observe that any finite dimensional semi-simple representation of $ W_\R \times \SL_2(\C)$ is a direct sum 
of tensor product of irreducible representations of $W_\R$ and $\SL_2(\C)$. 

We note that finite dimensional algebraic representations of $\GL_n(\C)$ are the highest weight modules of the form
$V=V_{\underline{\lambda}}$ for an $n$-tuple of integers,
$$\underline{\lambda} = \{\lambda_1 \geq  \lambda_2 \geq  \cdots \geq  \lambda_n \}.$$
One knows that if there is a unitary
cohomological representation of $\GL_n(\R)$ with coefficient in $V_{\underline{\lambda}}$,
then $\underline{\lambda}$ is self-associate in the sense that $\lambda_i + \lambda_{n+1-i} =0$ for all $i$.
Observe that this condition implies that
$$ w(\underline{\lambda})= \lambda_1 +   \lambda_2 +  \cdots +  \lambda_n =0,$$
so $V_{\underline{\lambda}}$ has trivial central character, and hence the cohomological
representations of $\GL_n(\R)$ that we will consider will also have trivial central character. As mentioned
in the introduction, this misses out certain interesting unitary cohomological representations which from
the point of this paper are best handled by considering them on $\SL_n(\R)$ where they will have non-trivial
central character.

We denote by $ \rho_n$, half the sum of positive roots:
$$\rho_n= \{(n-1)/2 > (n-3)/2 >\cdots > -(n-1)/2\}.$$

\begin{thm} \label{Speh} Cohomological representations for $\GL_n(\R)$ (for the finite dimensional coefficient system $V=V_\lambda \cong V^\theta \cong V^\vee$) have $L$-parameters underlying $A$-parameters
  $\sigma $
    where $\sigma$ is  the representation 
of $W_\R \times \SL_2(\C)$:
$$\sigma = \sum_d \sigma_d \otimes [m_d + 1] \oplus \omega_\R^{\{0,1\}}[a],$$
where $a\geq 0$ is an integer with $n\equiv a \bmod 2$, and the index 
$d$ runs over a set of (distinct) positive integers with $m_d \geq 0$, and such that the infinitesimal character of $\sigma $ is the same as that of $V$, i.e.
\begin{eqnarray*} 
 &  \sum_d & (\nu^{d/2}+ \nu^{-d/2})(\nu^{m_d/2} +\nu^{(m_d -2)/2} + \cdots \nu^{-(m_d/2)}) + (\nu^{(a-1)/2}  + \cdots + \nu^{-(a-1)/2)})  \\
  & = & [\nu^{\lambda_1+ (n-1)/2}  + \nu^{\lambda_2 + (n-3)/2}  + \cdots + \nu^{\lambda_n -(n-1)/2}],
 \end{eqnarray*}
an equality among elements of the group ring $\Z[\nu^{\frac{1}{2}\Z}]$, where
we omit the term $(\nu^{(a-1)/2}  + \cdots + \nu^{-(a-1)/2)})$ on the left if $a=0$. 
\end{thm}

  We next consider the group  
 $ \GL_n(\C)$ which could be considered either as a real reductive group, or as a complex reductive group. It is simpler to think of it as a complex reductive group,
  allowing us  to use $W_\C$ and $W_\C \times \SL_2(\C)$ instead of the more complicated $W_\R$ and $W_\R \times \SL_2(\C)$.

  For $\GL_n(\C)$, cohomological representations with the infinitesimal character that of the trivial representation
  are in bijective correspondence 
 with standard parabolics in $\GL_n(\C)$, thus with ordered
 partitions of $n$ as $(n_1,n_2,\cdots, n_k)$ with $n =n_1+n_2+\cdots +n_k$. To simplify notation, in what follows, 
 we will only consider cohomological representations of $\GL_n(\C)$ 
 with the infinitesimal character that of the trivial representation.

 For the ordered partition $\frak{p} = (n_1,n_2,\cdots, n_k)$ of $n$,
 with $n_i \not = 0$ for all $i$,
 defining the standard parabolic $P_{\frak{p}}$,
 the associated cohomological representation $\pi_{\frak{p}}$ of $\GL_n(\C)$ has, by equation (a) in the statement of Theorem \ref{cohomo}, the $A$-parameter $\sigma_{\frak{p}}$ given
 by (as a representation of $\C^\times \times \SL_2(\C)$):
  \[\sigma_{\frak{p}} = \sum_i (z/\bar{z})^{d_i} [n_i],\]
 where 
 \begin{eqnarray*}
 (\star) \begin{cases}  d_1 & =  \dfrac{n-n_1}{2} \\
     d_2 & =  \dfrac{n-(2n_1+n_2)}{2} \\
    d_3 & =  \dfrac{n-(2n_1+2n_2+n_3)}{2} \\
     \cdot  & =  \cdots \\
      \cdot  & =  \cdots \\
      d_k & =  \dfrac{n-(2n_1+2n_2+ \cdots + 2n_{k-1}+n_k)}{2}.
      \end{cases}
 \end{eqnarray*}

 In particular, a cohomological representation of $\GL_n(\C)$ is induced from a unitary one dimensional representation of a parabolic subgroup of $\GL_n(\C)$, a feature that this
 group has in common with all other complex reductive groups (and such unitary inductions are irreducible for all complex groups).

 Cohomological $A$-parameters of $\GL_n(\R)$ basechange to cohomological $A$-parameters of $\GL_n(\C)$, the resulting map is
 one-to-one or two-to-one (as we discuss more precisely below), and onto the set of
 self-conjugate partitions, i.e., ordered partitions of $n$ as $n=n_1+n_2+\cdots + n_k$,  with 
  $$(n_1,n_2,\cdots, n_k) = (n_k,n_{k-1}, \cdots, n_1).$$
 
 Given a self-conjugate (ordered) partition $(n_1,n_2,\cdots, n_k)$ of $n$, with $n_i \not = 0$ for all $i$, the cohomological representation of $\GL_n(\C)$ corresponding to 
 the representation $\sum_i (z/\bar{z})^{d_i} [n_i]$ of $\C^\times \times \SL_2(\C)$ (with $d_i$ given as above in equations $(\star)$) is self-dual
 with $d_i = -d_{k+1-i}$.
 If $k$ is even, there is the unique representation of $W_\R \times \SL_2(\C)$ whose restriction to 
 $\C^\times \times \SL_2(\C)$ is $\sum_i (z/\bar{z})^{d_i} [n_i],$ which is:
 $$ \sum_{i =1}^{i= k/2} {\rm Ind}_{W_\C}^{W_\R}(z/\bar{z})^{d_i}  \otimes [n_i] = \sum_{i =1}^{i= k/2} \sigma_{d_i} \otimes [n_i].$$
 If $k$ is odd, then $d_{(k+1)/2} = 0$, and there are two  representations of $W_\R \times \SL_2(\C)$ whose restriction to 
 $\C^\times \times \SL_2(\C)$ is $\sum_i (z/\bar{z})^{d_i} [n_i],$ which are:
$$ \sum_{i =1}^{i= (k-1)/2} {\rm Ind}_{W_\C}^{W_\R}(z/\bar{z})^{d_i}  \otimes [n_i] + \omega_\R^{\{0,1\}}[n_{(k+1)/2}] = \sum_{i =1}^{i= (k-1)/2} \sigma_{d_i} \otimes [n_i]+ \omega_\R^{\{0,1\} }[n_{(k+1)/2}] .$$

 We next consider the group $\U(p,q)$. Recall from \cite{GGP} that parameters for $\U(p,q)$ are nothing but conjugate-selfdual parameters for $\GL_n(\C)$ of parity 1 if $p+q$ is odd and parity 
 $-1$ if $p+q$ is even. Since parity of the conjugate-selfdual representation  $(z/\bar{z})^{d_i}$ of $W_\C=\C^\times$, where $2d_i$ is an integer,  is 
 $(-1)^{2d_i}$, and since parity of the representation $[n]$ of $\SL_2(\C)$ is $(-1)^{n-1}$, the way $d_i$ are defined above in equations $(\star)$,
 one sees that all the representations 
  $\sum_i (z/\bar{z})^{d_i} [n_i]$ of $\C^\times \times \SL_2(\C)$ (with $d_i$ given as above) are conjugate-selfdual of parity $(-1)^{n-1}$ which makes them 
  parameters for $\U(p,q)$. We note this as a corollary.
  
  \begin{cor}
    The basechange map from $\U(p,q)$ to $\GL_n(\C)$ gives a bijective map from 
    the set of cohomological parameters of  $\U(p,q)$ to the set of cohomological parameters of  $\GL_n(\C)$ (for any
    coefficient system).
             \end{cor}

\section{Some explicit low rank examples} \label{more-exam}

We give a few illustrative examples of cohomological parameters of $\sG(\R)$ with the infinitesimal character that of the trivial representation of $\sG(\C)$ beginning with $\sG(\R)= \GL_n(\R)$. In all cases, if $\sigma$ is a
cohomological parameter of $\GL_n(\R)$, then so is $\sigma \otimes \omega_\R$, although we list only one of them below.

\begin{enumerate}
 \item  $\GL_2(\R)$: There are two cohomological parameters  corresponding to self-conjugate partitions: $\{1,1\}$ and $\{2\}$. The corresponding
 representations are the trivial representation with parameter $\nu^{1/2} + \nu^{-1/2}$, and the discrete series $D_2$ with parameter $\sigma_1$.
 
 \item  $\GL_3(\R)$: There are two cohomological parameters corresponding to self-conjugate partitions: $\{1,1, 1\}$ and $\{3\}$.  The corresponding
 representations are the trivial representation, with $L$-parameter $\nu + 1 + \nu^{-1}$, and the tempered representation $1 \times D_3$ with parameter $1+ \sigma_2$.

\item  $\GL_4(\R)$:  There are 4 cohomological parameters corresponding to self-conjugate partitions: $\{4\}$,  $\{2,2\}$, $\{1,2,1\}$,  $\{1,1, 1, 1\}$.
  The corresponding
 representations are the representations with $A$-parameters $[4]$, $\sigma_3 + [2]$, $\sigma_2 \otimes [2]$, $\sigma_3 + \sigma_1$.
 
\item $\GL_5(\R)$: There are 4 cohomological parameters corresponding to self-conjugate partitions: $\{5\}$, $\{2,1, 2\}$,  $\{1,3,1\}$, $\{1,1, 1, 1, 1\}$.
  The corresponding
 representations are the representations with $A$-parameters $[5]$, $\sigma_4 + [3]$, $\sigma_3 \otimes [2] + 1$, $\sigma_4 + \sigma_2 + 1$.
 \end{enumerate}
 
 Next, we give a few examples of classical groups. 
 
 \begin{enumerate}
 \item  $\Sp_4(\R)$ with $L$-group $\SO(5, \C)$, thus the cohomological $A$-parameters correspond
 to the list above for $\GL_5(\R)$ which are (after correcting for the sign character to ensure we are in $\SO(5,\C)$ and 
 not ${\rm O}(5,\C)$)  
 $[5]$, $\sigma_4 + \omega_\R[3]$, $\sigma_3 \otimes [2] + 1$, $\sigma_4 + \sigma_2 + 1$.
 
\item  $\SO(2,3)(\R)$ with $L$-group $\Sp(4, \C)$, thus the cohomological $A$-parameters corresponds the list above 
for $\GL_4(\R)$ consisting of $[4]$, $\sigma_3 + [2]$, $\sigma_2 \otimes [2]$, $\sigma_3 + \sigma_1$. (These parameters are all  symplectic.)
 
 \item $\U(2,1)$, their cohomological parameters are cohomological parameters of representations of 
 $\GL_3(\C)$ (which are automatically conjugate-orthogonal). These are representations of  $W_\C \times \SL_2(\C)$ of the form: $[3], z/\bar{z}+ \bar{z}/z +1,  (z/\bar{z})^{-1}+ ( z/\bar{z})^{1/2} [2],  (z/\bar{z})+ ( z/\bar{z})^{-1/2} [2]$.

 \end{enumerate}

 \section{$(\g,K)$-cohomology} \label {coho}

 In this section we show that the total dimension of the $(\g,K)$-cohomology groups of representations, summed over the representations in an $AJ$-packet, is independent of the packet and has a simple expression (Theorem~\ref{packetsum}).  We also evaluate the sum of these expressions over pure inner forms, which is particularly simple (Theorem~\ref{serre}).
 Finally, when $G/K$ has an invariant complex structure we refine the numerical statement of Theorem~\ref{packetsum} to a statement about (mirror) Hodge polynomials (Theorem~\ref{mirror}).

 The following theorem may be known to experts; it was known to O. Ta\"ibi in the equal-rank case.  We will formulate it for $(\g,K)$-cohomological packets, leaving the case of general $(\g,K_0)$-cohomological packets to the reader.

\begin{thm}\label{packetsum}
  Let ${G}=\sG(\R)$ be the  group of real points of a connected reductive real algebraic group $\sG$,  $K \subset G$ a maximal compact subgroup, and $G^u$ the compact real form of $G$.  Let $\sT^c$ be a $\theta$-stable fundamental torus in $\sG$ and write $\sT^c=\sT\,\mathsf{A}$ with $\sT$ maximally anisotropic. 
Let $E$ be a finite-dimensional algebraic representation of $\sG(\C)$ which is self-associate. Then for a nonempty Adams-Johnson packet $\Pi$ of representations with $(\g,K)$-cohomology  with respect to $E$, the sum
$$
\sum_{\pi \in \Pi} \sum_{i\geq 0} 
\dim_\C H^i(\g,K, \pi \otimes E) 
$$
is independent of the packet, and equals  
$$
\sum_i \dim_\C H^i(\g,K,\C) = \sum_i \dim_\C H^i(G^u/K) = 
2^d\left |\frac{W(\sG,\sT^c)^\theta}{W(G,T^c)}\right|
$$
where $d = \dim \mathsf{A}= {\rm rank} (\sG) - {\rm rank}(K)$ is the  discrete series defect, $W(\sG,\sT^c)^\theta=\{w\in W(\sG,\sT^c): w\theta=\theta w\}$, and $W(G,T^c)$ is the subgroup of elements with a representative in $G$ (equivalently, it is the image of $W(K,T)$ in $W(\sG,\sT^c)$). 
\end{thm}

\begin{proof} 
We first give the proof under the assumption that $\sG$ is semisimple and simply-connected and that the connected group $G=\sG(\R)$ has discrete series.  Let $\widehat{X}_G:=G^u/K$ be the compact dual symmetric space. Let ${\sf T}$ be a maximal torus with $T={\sf T}(\R)$ compact.   In this case, using the parametrization of $AJ$-packets recalled earlier (in Section~\ref{AJpack}), the identity we seek to prove is that for a standard parabolic $Q$ in $\widehat\sG$ and $\theta$-stable parabolic subalgebra $\mathfrak{q} \in \Sigma^{-1}(Q)$ with Levi subalgebra $\mathfrak{l}=\mathfrak{q}\cap \bar{\mathfrak{q}}$ and associated subgroup $\sL \subset \sG$, we have 
$$
\sum_{w \in W(G,T)\backslash W(\sG,{\sf T})/W(\sL,{\sf T})} \sum_{i\geq 0} 
\dim_\C H^i(\widehat{X}_{L_w})  
= \sum_{i\geq 0} \dim_\C H^i(\widehat{X}_G). 
$$
Here $\sL_w$ is the Levi subgroup given by $\q(w\lambda)$ if $\q=\q(\lambda)$ for $\lambda \in i\t_0^*$, and $\widehat{X}_{L_w}$ denotes the compact dual symmetric space of $L_w=\sL_w(\R)$ (which is also connected and has compact centre), and we have used the fact that the graded vector spaces $H^*(\g,K, \mathcal{R}^S_{\q_w}(\lambda)\otimes E)$ and $H^*(\widehat{X}_{L_w})$ coincide up to a degree shift (Theorem 3.3 of \cite{VZ} for $V=\C$ and Theorem 5.5 of loc. cit. in general, note that $\mathcal{R}^S_\q(\C_\lambda)$ is denoted $A_\q(\lambda)$ in loc. cit.).     
It is a well-known  classical fact (and  a special case of Theorem \ref{symmetric} below) that 
$$
\sum_{i\geq 0} \dim_\C H^i(\widehat{X}_G) = 
\left|\frac{W(\sG,{\sf T})}{W(G,T)}\right|
$$ 
because 
$\widehat{X}_G$ is an equal-rank compact symmetric space. 
(Note that $W(K,T) = W(G,T)$).  Since the derived group of $\sL_w$ is also simply-connected and $L_w$ has discrete series, we know the same for $L_w$, i.e. 
$\sum_{i\geq 0} \dim_\C H^i(\widehat{X}_{L_w})= 
\left|\frac{W(\sL_w,{\sf T})}{W(L_w,T)}\right|$.  So what we want to prove can be rewritten as 
$$
\sum_{w \in W(G,T)\backslash W(\sG,{\sf T})/W(\sL,{\sf T})} 
\left|\frac{W(\sL_w,{\sf T})}{W(L_w,T)}\right|
=
\left|\frac{W(\sG,{\sf T})}{W(G,T)}\right|. 
$$
But now this follows directly by looking at the action of $W(\sL, {\sf T})$ on $W(G,T)\backslash W(\sG,{\sf T})$ and using the facts that,  as subgroups of $W(\sG, {\sf T})$, 
$W(\sL_w,{\sf T}) = wW(\sL, {\sf T})w^{-1}$  and $W(L, T) = W(\sL,{\sf T})\cap W(G,T)$.

For the general case we will use Theorem \ref{symmetric} below, together with the computation in Proposition~\ref{cohprop}.  
Let $\wL\in \mathcal{L}_E$, i.e. $\wL$ is a standard Levi of $\wG$ invariant under $\omega_{\wG}\cdot j$ and satisfies the condition in Theorem~\ref{AJ1} for $E$. 
Choosing a basepoint  $\q \in \Sigma^{-1}(\wL)$, the packet $\Pi_{E,\wL}$ of $(\g,K)$-modules with cohomology is parametrized (cf. Theorem \ref{AJ1}) by the double cosets 
$$
W(\sL, \sT^c)^{\theta}\backslash W(\sG, \sT^c)^{\theta}/W(G,T^c). 
$$
For $w \in W(\sG, \sT^c)^\theta$, the representation $\mathcal{R}^S_{\q_w}(\pi_{L_w})$ in the packet has cohomology 
$$
H^*(\g,K, \mathcal{R}_{\q_w}^{S_w}(\pi_{L_w})\otimes E^\vee) = H^{*-R_w}(L_w^u/L_w\cap K) 
$$ 
(cf. Proposition~\ref{cohprop}). 
Applying Theorem \ref{symmetric} below with $\sG$ replaced by $\sL_w$ (note that $W(K,T) \to W(\sG,\sT^c)$ may not be injective, but its image is $W(G,T^c)$ because $\sT^c$ is fundamental, and the same applies to $\sL_w$),  the identity to be proved boils down to 
$$
\sum_{w \in W(G,T^c)\backslash W(\sG,\sT^c)^\theta/W(\sL,\sT^c)^\theta} 
\left|\frac{W(\sL_w,\sT^c)^\theta}{W(L_w,T^c)}\right| 
=
\left|\frac{W(\sG,\sT^c)^\theta}{W(G,T^c)}\right| 
$$
where the factor $2^{\dim {\sf A}}=2^d$ appearing on both sides has been dropped. 
This follows easily from the facts that 
$W(\sL_w,\sT^c)^\theta = wW(\sL, \sT^c)^\theta w^{-1}$  and 
$W(L_w, T^c) = W(\sL_w,\sT^c)\cap W(G,T^c)$. 
 \end{proof}

It remains to prove Theorem~\ref{symmetric} below, which computes $\sum_{i\geq 0}H^i(G^u/K)$ for the compact symmetric space $G^u/K$, but before that we show that adding up contributions 
of compact symmetric spaces which are pure inner forms of each other gives a very simple answer. 

   \begin{thm} \label{serre}
     Let $\G_c$, $c \in H^1(\Gal(\C/\R), \G(\C))$,
     be the set of pure inner forms of a compact connected Lie group $\G$
          with maximal compact subgroups $K_c = \G_c(\R) \cap \G$. Then the following holds:
          $$\sum_{i,c}  \dim H^i(\G/K_c,\C) = \sum_{c} \frac{|W_\G|}{|W_{K_c}|} = 2^r,$$
     where $r$ is the rank of $\G$.

     Let  $\G_c$, $c \in H^1(\Gal(\C/\R), \G(\C))$,
     be the set of pure inner forms of a quasi-split real Lie group $\G(\R)$
          with maximal compact subgroups $K_c \subset \G_c(\R) $.
               Let $\theta$ be a Cartan involution
on $\G(\C)$ preserving $\G(\R)$ and a fundamental torus $\T_f(\R) \subset \G(\R)$, hence acting on the Weyl group of $\G(\C)$ with respect to the torus $\T_f$.  
Let    $\T_f(\R) = (\R^\times)^a \times (\C^\times)^b \times (\Si)^e$ (written uniquely in this form as for any real torus). Then for the natural map $\iota:W_{K_c} \rightarrow W_\G$ of Weyl groups, 
$$     \sum_{c}\left | \frac{W_\G} {\iota(W_{K_c})} \right |
= |H^1(\Gal(\C/\R), \T_f(\C))| = 2^e,$$
and therefore by Theorem \ref{symmetric} below for $\G^u = \G(\C)^\theta,$ a maximal compact subgroup of $\G(\C)$: 
$$\sum_{i,c}  \dim H^i(\G^u/K_c,\C) = 2^{a+b+e}.$$
   \end{thm}

   \begin{proof}
     This theorem is essentially a simple consequence of Theorem \ref{symmetric} and the way  
     real Lie groups are classified   in terms of Galois cohomology $H^1(\Gal(\C/\R), \G(\C))$ which classifies pure inner forms of a group $\G(\R)$. We split our proof in the two (non-mutually exclusive) cases:

     \begin{enumerate}
     \item The case  of pure inner forms of compact Lie groups,
       see Theorem 6, Chapter III, \S 4, of \cite{Se},
              according to which there is a natural identification 
     $$\T[2]/W_\G \longleftrightarrow H^1(\Gal(\C/\R), \G(\C)).$$
            \item      The case of pure inner forms of a quasi-split reductive group $\G(\R)$,
            where  the analogous result is due to Borovoi, Theorem 9 of \cite{Boro}, according to which 
there is a surjective map
$$ H^1(\Gal(\C/\R), \T_f(\C)) \rightarrow H^1(\Gal(\C/\R), \G(\C)),$$
  the fibers of which are exactly the orbits of  $W_{\G(\C)}(\R)$ acting naturally on the set 
  $H^1(\Gal(\C/\R), \T_f(\C))$, where the Weyl group is  defined using the fundamental
  maximal torus $\T_f$. (It may be noted that  for the Cartan involution $\theta$ on $\G(\R)$ which normalizes $T_f(\R)$, $W^{\theta}_{\G(\C)} = W_{\G(\C)}(\R)$.)
\end{enumerate}

              We first take up the case of  pure inner forms of a compact Lie group,     recalling  the explicit identification of $\T[2]/W_\G$ with $H^1(\Gal(\C/\R), \G(\C))$ which is needed to prove our theorem. 

     Fix $\T$ to be a maximal compact torus in $\G$, a compact connected Lie group. Let $\G(\C)$ be a fixed complexification of $\G$ with complex conjugation denoted by $g \mapsto \bar{g}$ such that
     $\G = \G(\R) \subset \G(\C)$ is the subgroup of fixed points of this complex conjugation. Thus $g \mapsto \bar{g}$ is the Cartan involution on $\G(\C)$. We will have occasion to use the well-known observation that
     if $H \subset \G(\C)$ is any real Lie subgroup invariant under a Cartan involution on $\G(\C)$, in this case $g \mapsto \bar{g}$, then the restriction of the Cartan involution on $\G(\C)$ to $H$ is a Cartan involution on $H$, in particular, $H \cap \G$ is a maximal compact subgroup of $H$.

     For any $c \in \T[2]$, define a real form $\G_c \subset \G(\C)$ by,
     $$ \G_c = \{ g \in \G(\C) | \bar{g} = c^{-1} gc \}.$$

     The subgroup $\G_c$ of $\G(\C)$, is the real points of a pure inner form of $\G$ which by abuse of notation will also be denoted by $\G_c$. The group $\G_c$ is invariant under
     complex conjugation by $ g \mapsto \bar{g} $ which is a Cartan involution on $\G_c$ with maximal compact $$K_c = \G_c \cap \G,$$
     with $\T \subset K_c$.
     
The mapping $c \mapsto \G_c$ thus constructed from $\T[2]$ to pure inner forms of $\G$ is a surjection, and for $c_1,c_2 \in \T[2]$, $\G_{c_1} $ is the same pure inner form as $\G_{c_2}$ if and only
     if ${c_1}$ and $c_2$ are in $W_\G$ orbit of each other.

     Clearly, $$W_{K_c}= N_{K_c}(T)/T = \{w \in W_\G| cwc^{-1}=w\} =  \{w \in W_\G| w(c)=c \},$$ 
therefore, the stabilizer of $c \in \T[2]$ in $W_\G$ is exactly $W_{K_c}$, proving the theorem in this case:
     $$2^r = |\T[2]| = \sum_{c} \frac{|W_\G|}{|W_{K_c}|}.$$

We next proceed to the case of quasi-split group $\G(\R)$ containing the fundamental torus $\T_f(\R)$ using the theorem of  Borovoi recalled earlier instead of that of Serre.
The situation of the quasi-split group $\G(\R)$ can be summarized in the following diagram where $\tau$ is the complex conjugation
on $\G(\C)$ for the real structure given by $\G(\R)$, and $\theta$ is the Cartan involution
on $\G(\C)$ with fixed points $\G^\theta = \G^u$ such that $\G^u\cap \G(\R)$ is a maximal compact subgroup of $\G(\R)$. The involutions $\tau$ and $\theta$ commute, and both being
{\it anti-holomorphic}, $\tau \circ \theta = \theta \circ \tau$ is a holomorphic involution,
and in fact an algebraic
automorphism of $\G(\C)$.

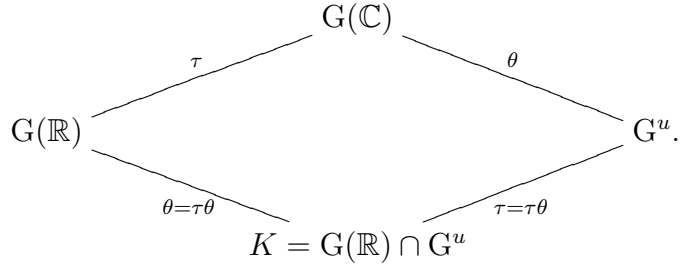
$$
\xymatrix{
& & \G(\C)\ar@{-}[dll]_{\tau}\ar@{-}[drr]^{\theta} &  & \\
\G(\R) \ar@{-}[drr]_{\theta=\tau\theta} & & & & \G^u\ar@{-}[dll]^{\tau=\tau\theta}. \\
 & & K=\G(\R) \cap \G^u  & &}
$$

Since $\Gal(\C/\R) = \Z/2$, a cocycle  in $H^1(\Gal(\C/\R), \T_f(\C))$
is determined by an
element $c \in T_f(\C)$ with $c\bar{c}=1$, so in what follows, when we talk of a cocycle in $H^1(\Gal(\C/\R), \T_f(\C))$, we just mean an element of $T_f(\C)$;
similarly for   $H^1(\Gal(\C/\R), M(\C))$ for any
algebraic group $M$ over $\R$.

Just as in case 1 of pure inner forms of a compact group, for $c \in H^1(\Gal(\C/\R), \T_f(\C))$
represented by an element $c \in T_f(\C)$ (with $c\bar{c}=1$),
defines a real form $\G_c \subset \G(\C)$ by,
     $$ \G_c = \{ g \in \G(\C) | \bar{g} = c^{-1} g{c} \},$$
where the complex conjugation $\bar{g} =\tau(g)$ used is the one which defines the quasi-split group $\G(\R)$.
If $\T \subset \T_f(\R)$ is the maximal compact connected subgroup of $T_f(\R)$, then it is easy to
see that the mapping  $H^1(\Gal(\C/\R), \T(\C)) \rightarrow H^1(\Gal(\C/\R), \T_f(\C))$ is surjective,
and therefore we may assume that $c \in \T[2]$, in particular, $\theta(c) =c$.

The subgroup $\G_c$ of $\G(\C)$ is the real points of a pure inner form of $\G$. Since $\theta, \tau$ commute,
it follows that $\G_c$ is invariant under
     the involution  $ g \mapsto \theta(g) $ which is a Cartan involution on $\G_c$ with maximal compact $$K_c = \G_c(\R)  \cap \G^\theta(\C).$$

     The situation of the group $\G_c(\R)$ can be summarized in the following diagram where $\tau_c(g) = c \tau(g) c^{-1}$ is the complex conjugation
on $\G(\C)$ for the real structure given by $\G_c(\R)$, and $\theta$ is the Cartan involution
on $\G(\C)$ with fixed points $\G^u$. The involutions $\tau_c$ and $\theta$ commute, and both being
{\it anti-holomorphic}, $\tau_c \circ \theta = \theta \circ \tau_c$ is a holomorphic involution, and in fact an algebraic
automorphism of $\G(\C)$.

$$
\xymatrix{
& & \G(\C)\ar@{-}[dll]_{\tau_c}\ar@{-}[drr]^{\theta} &  & \\
\G_c(\R) \ar@{-}[drr]_{\theta=\tau_c\theta} & & & & \G^u\ar@{-}[dll]^{\tau_c=\tau_c\theta}. \\
 & & K_c=\G_c(\R) \cap \G^u  & &}
$$

     By the theorem of Borovoi mentioned earlier,
     the mapping $c \mapsto \G_c$  from the set $ H^1(\Gal(\C/\R), \T_f(\C))$
     to the set of pure inner forms of $\G$ is a surjection, and for $ c_1,c_2 \in H^1(\Gal(\C/\R), \T_f(\C))$,
     $\G_{c_1} $ is the same pure inner form as $\G_{c_2}$ if and only
if ${c_1}$ and $c_2$ are in the same $W^\theta_\G$-orbit. Thus to prove the theorem for the quasi-split group $\G(\R)$, it suffices to prove that for an element $c \in H^1(\Gal(\C/\R), \T_f(\C))$,
the subgroup $$\{ w \in W^\theta_\G| w(c) = c \},$$ 
is nothing but $\iota(W_{K_c})$, where $W_{K_c}$ is the Weyl group of the maximal compact subgroup
$K_c=\G^u \cap \G_c(\R)$ of  $\G_c(\R)$, and $\iota: W_{K_c} \rightarrow W_\G$ is the natural map of Weyl groups. 

     Observe that the action of $n \in N_{\G(\C)}(T_f)$
     on an element $c \in H^1(\Gal(\C/\R), \T_f(\C))$ is by $c\mapsto n^{-1} c \bar{n}$, and therefore, the element  $n \in N_{\G(\C)}(T_f)(\C)$ acts trivially on $c \in H^1(\Gal(\C/\R), \T_f(\C))$
     if and only if there exists $t \in \T_f(\C)$ such that:
     $$ n^{-1} c \bar{n} = tc\bar{t}^{-1},$$
     which can be rewritten as:
     $$\overline{nt} = c^{-1}(nt) c,$$
     therefore by the definition of $\G_c(\R)$, $nt \in \G_c(\R)$, hence $n$ is represented by an
     element of $\iota(W_{K_c})$ (using the well-known result that the normalizer of a fundamental
     torus $T_f(\R)$ in any reductive group such as $\G_c(\R)$  is represented by an element in
     the unique maximal compact
     subgroup of $\G_c(\R)$ containing the maximal compact subgroup of $\T_f(\R)$). Conversely, the same computation shows that any element in $\iota(W_{K_c})$
     fixes $c  \in H^1(\Gal(\C/\R), \T_f(\C))$,
     proving one of the equalities in the theorem in this case:
     $$ \sum_{c} \left | \frac{W_\G} {\iota(W_{K_c})} \right |
     = |H^1(\Gal(\C/\R), \T_f(\C))| = 2^e.$$

     Finally, we note that by Theorem \ref{symmetric} below,
     $$      \sum_i \dim H^i(\G^u/ K_c, \C) = \sum_i \dim H^i(\g , {K_c}, \C)
     = 2^{a+b} \cdot  \left |W_{\G}^\theta /\iota(W_{K_c}) \right |, $$
     completing the proof of the theorem.
   \end{proof}

   The following theorem calculates the sum of dimensions of cohomology groups of compact symmetric spaces. We cannot imagine
   it has not been known, but we have certainly not found any reference to it. (Of course, when $\theta$ is an inner automorphism of ${\mathcal G}$, i.e., the equal-rank case, this is very well known.)

 \begin{thm} \label{symmetric}
   Let ${\mathcal G}$ be a compact connected Lie group, $\theta$ an automorphism of ${\mathcal G}$ of order 2, and ${\mathcal K}$ a subgroup of finite index of $ {\mathcal G}^\theta$, 
   the fixed point subgroup of ${\mathcal G}$. Let ${\mathcal T}$ be a maximal torus in ${\mathcal K}$ with centralizer $Z_{\mathcal G}({\mathcal T}) = S$
   and $d = \dim(S)-\dim({\mathcal T})$.   Then, $W_{\mathcal G}= N_{\mathcal G}(S)/S$ carries an action of $\theta$ such that for $W_{\mathcal K} = N_{\mathcal K}({\mathcal T})/{\mathcal T}$,  there is a natural map $\iota: W_{\mathcal K} \rightarrow W_{\mathcal G}^\theta$ (which may not be injective), such that 
   $$\sum_i \dim H^i({\mathcal G}/ {\mathcal K}, \C) = \sum_i \dim H^i(\g , {\mathcal K}, \C)
   = 2^d \cdot  \left |W_{\mathcal G}^\theta /\iota(W_{\mathcal K}) \right |. $$
      \end{thm}
 \begin{proof} 
As is well known,
$H^i({\mathcal G}/ {\mathcal K}, \C) \cong  H^i(\g , {\mathcal K}, \C)$, proving the first equality in the assertion of the theorem. We will calculate
$\sum_i \dim H^i({\mathcal G}/ {\mathcal K}, \C)$  using the Lefschetz fixed point formula as in the classic topological proof of conjugacy of maximal tori in compact groups by Hopf-Samelson and Weil.

   The (global) Cartan involution  $\theta$ gives a self-map on $X = {\mathcal G}/{\mathcal K}$.  The decomposition $\g=\mathfrak{k}+\mathfrak{p}$ (into the $\theta$-eigenspaces on the complexified Lie algebra $\g$ of ${\mathcal G}$) and  the formula $H^i(\g,{\mathcal K},\C) = \Hom_{\mathcal K}(\wedge^i \mathfrak{p},\C)$ shows that the action $\theta$ on $H^*({X}, \C)$ is the identity
   in even degrees and $-1$ in odd degrees, so that the Lefschetz number of $\theta$ is the sum
   $\sum_i \dim_\C H^i({X}, \C)$ we want to compute.  
      For any $t \in {\mathcal T}$, the self-map $\theta_t$ of ${X}$ given by
      $\theta_t(g{\mathcal K}) := \theta(tg{\mathcal K}) = t\theta(g{\mathcal K})$ is homotopic to $\theta$ and hence has the same Lefschetz number as $\theta$.  

   Choose $t \in {\mathcal T}$ so that $\overline{\langle t\rangle}={\mathcal T}$, i.e. $t$ generates ${\mathcal T}$ topologically.   Let $F={X}^{\mathcal T}$ be the fixed-point locus of ${\mathcal T}$ or, equivalently, the fixed-point locus of  multiplication by $t$ on $X={\mathcal G}/{\mathcal K}$.  We claim that the fixed point set of $\theta_t$ is $F^\theta$.   For this, observe first that if $\theta_t(g{\mathcal K}) = g{\mathcal K}$, i.e., $x=g^{-1}t\theta(g) \in {\mathcal K}$, 
then $x \theta(x) = g^{-1}t^2 g$ also belongs to ${\mathcal K}$, and thus $g^{-1}{\mathcal T}g \subset {\mathcal K}$, i.e., $g{\mathcal K} \in ({\mathcal G}/{\mathcal K})^{\mathcal T}=F$, therefore $g{\mathcal K} \in ({\mathcal G}/{\mathcal K})^{\mathcal T}$ which being $\theta_t$ invariant,  belongs to  $F^\theta$.

We need to understand $F^\theta$, for which we begin by understanding $F = X^{\mathcal T}$. Clearly a point $x = g{\mathcal K}  \in {\mathcal G}/{\mathcal K}$ is ${\mathcal T}$ invariant if and only if $g^{-1}{\mathcal T}g \subset {\mathcal K}$, which then is a maximal torus in ${\mathcal K}$, and hence by conjugacy of maximal tori in ${\mathcal K}$, we have: $g^{-1}{\mathcal T}g = k{\mathcal T}k^{-1}$, i.e., $gk \in N_{\mathcal G}({\mathcal T})$. Thus,  $$F = X^{\mathcal T} = [N_{\mathcal G}({\mathcal T})\cdot {\mathcal K}]/{\mathcal K} =
N_{\mathcal G}({\mathcal T})/N_{\mathcal K}({\mathcal T}).$$

Next, observe that $S=Z_{\mathcal G}({\mathcal T}) \subset N_{\mathcal G}({\mathcal T})$, and further,  $ N_{\mathcal G}({\mathcal T}) \subset  N_{\mathcal G}( Z_{\mathcal G}({\mathcal T})) =  N_{\mathcal G}(S)$, thus $N_{\mathcal G}({\mathcal T})$
is that subgroup of $N_{\mathcal G}(S)$ which under conjugation takes the subgroup ${\mathcal T}$ into itself. Since ${\mathcal T}$ is the connected
component of identity of $S^\theta$, it follows that   $N_{\mathcal G}({\mathcal T})$ is the  inverse image of   $W_{\mathcal G}^\theta$ under the quotient map
$N_{\mathcal G}(S) \rightarrow   W_{\mathcal G}$. Further, there is a canonical $\iota: W_{\mathcal K} \rightarrow W_{\mathcal G}$ landing inside $W_{\mathcal G}^\theta$ which fails to be an  injection precisely when $Z_{\mathcal K}({\mathcal T})  $ is strictly larger than ${\mathcal T}$, a possibility only for disconnected ${\mathcal K}$.

We now need to calculate the fixed points of $\theta$ on $F$,  which we have realized as a quotient of two groups. In general, computing the fixed points on a quotients of two groups
tends to be a difficult question, but it is easy enough in this case. For this we first observe the following lemma:

\begin{lemma}
  Suppose $E$ is a set with a fixed point free action of a group $A$, and a commuting action of
  an involutive  automorphism $\theta$ on $E$ and on $A$. Assume that $B$ is the quotient of
  $E$ by $A$ with the quotient map $\pi: E\rightarrow B$. Assume that
  $H^1(\langle \theta \rangle, A) = 1$. Then $A^\theta$ has a fixed point free action on $E^\theta$ with quotient $B^\theta$. In particular, if $E^\theta$ is a finite set,
  then $$|E^\theta|  = |A^\theta| \cdot |B^\theta|.$$ 
\end{lemma}
\begin{proof} The main step is to prove that the induced map  from $E^\theta$ to  $B^\theta$ is surjective. Suppose $\pi(e) \in B^\theta$. Therefore, by the $\theta$-invariance
  of the map $\pi: E\rightarrow B$,
  $$\theta(e) = a\cdot e,$$
  for some $a\in A$. Since $\theta^2=1$, and since $A$ operates on $E$ without fixed points, $a\cdot \theta(a)=1$. Therefore as  $H^1(\langle \theta \rangle, A) = 1$,
  $a = \alpha^{-1} \cdot \theta(\alpha)$ for some $\alpha \in A$. Now it is easy to see that $\theta(\alpha) \cdot e$ is $\theta$ invariant, proving that the induced
  map  from $E^\theta$ to  $B^\theta$ is surjective, whose fibers are easily seen to be $A^\theta$. \end{proof}

This lemma will be applied, in the notation of the present theorem to $E=  N_{\mathcal G}({\mathcal T})/N_{\mathcal K}({\mathcal T})$, and $A= S/S^\theta$ with $A$ operating on $E$ by natural multiplication on the left (the action is easily seen to  be fixed point free), and $\theta = \theta_t$.
First of all, since $\theta$ operates on the compact torus $S/S^\theta$ by $t\mapsto t^{-1}$,   $H^1(\langle \theta \rangle,  S/S^\theta) = 1$,
and  $\left|(S/S^\theta)^\theta\right| = 2^d$. By the previous lemma,
$$
\left|\left(
\frac{N_{\mathcal G}({\mathcal T})}
{N_{\mathcal K}({\mathcal T})}
\right)^\theta\right|=
2^d \cdot 
\left| \left(
\frac{N_{\mathcal G}({\mathcal T})}
{S\cdot N_{\mathcal K}({\mathcal T})} 
\right)^\theta\right|.
$$
We know that the action of $\theta$ on $ N_{\mathcal G}({\mathcal T})/\left(S\cdot N_{\mathcal K}({\mathcal T})\right)$ is trivial.  The exact sequence 
$$1 \rightarrow     \frac{ N_{\mathcal K}({\mathcal T})}{S \cap N_{\mathcal K}({\mathcal T})}  =
\frac{S\cdot N_{\mathcal K}({\mathcal T})}{S}  \rightarrow   \frac{  N_{\mathcal G}({\mathcal T})}{S}  \rightarrow
\frac{  N_{\mathcal G}({\mathcal T})}{S\cdot N_{\mathcal K}({\mathcal T})  }  \rightarrow 1,$$
shows that 
$$ \left |\frac{  N_{\mathcal G}({\mathcal T})}{S\cdot N_{\mathcal K}({\mathcal T})  } \right | = 
\left|\frac{W_{\mathcal G}^\theta}{\iota(W_{\mathcal K})}\right|.$$
 Thus we find that $F^\theta$ is a finite set of cardinality
$$2^d \cdot  \left |W_{\mathcal G}^\theta /\iota(W_{\mathcal K}) \right |. $$
Note that since $F^\theta = (X^{\mathcal T})^\theta =  X^{{\mathcal T} \rtimes \langle \theta\rangle}$ is the fixed-point locus of a compact group acting on a smooth manifold, the action can be linearized at each fixed point by the well-known Bochner linearization theorem (see e.g. page 96 of \cite{DK}).
It follows that the fixed point locus consists of a finite set of isolated fixed points in the differentiable sense also, i.e. it is a closed submanifold of dimension zero.  To conclude the proof of the theorem, we must prove that
the Lefschetz number of the diffeomorphism $\theta_t$ on $X = {\mathcal G}/{\mathcal K}$ at each of its fixed points has
the same sign, in our case positive, and this can be computed by looking at the tangent space.

Let $\omega_0$ be an isolated fixed point of the action of $\theta_t$ on $X = {\mathcal G}/{\mathcal K}$ discussed above. We will translate the fixed point $\omega_0$ to identity using the diffeomorphisms
$L_g:  {\mathcal G}/{\mathcal K} \rightarrow {\mathcal G}/{\mathcal K}$ defined by $L_g(x {\mathcal K} ) = gx {\mathcal K}$.
Thus, consider the diffeomorphism $L_{\omega_0^{-1}} \circ \theta_t \circ L_{\omega_0}$  on ${\mathcal G}/{\mathcal K}$:
  $$L_{\omega_0^{-1}} \circ \theta_t \circ L_{\omega_0}: x {\mathcal K} \rightarrow  \omega_0^{-1} t \theta(\omega_0) \theta(x)  {\mathcal K}.$$
    Since $\omega_0$ is a  fixed point of the action of $\theta_t$ on $ {\mathcal G}/{\mathcal K}$,
    $$\omega_0^{-1} t \theta(\omega_0) = k_0 \in K.$$
    Therefore,     the diffeomorphism $L_{\omega_0^{-1}} \circ \theta_t \circ L_{\omega_0}$  on ${\mathcal G}/{\mathcal K}$ is the descent to  ${\mathcal G}/{\mathcal K}$
      of the automorphism
      $$g \mapsto k_0 \theta(g) k_0^{-1},$$
      of ${\mathcal G}$.

      The Lefschetz number of the diffeomorphism $\theta_t$ of $X = {\mathcal G}/{\mathcal K}$ at $\omega_0$ is the same as
      the Lefschetz number of the diffeomorphism $L_{\omega_0^{-1}} \circ \theta_t \circ L_{\omega_0}$  of $X = {\mathcal G}/{\mathcal K}$ (taking the identity element to the identity element)
      at the identity element, which is what we will calculate now.
      
      The automorphism $ \varphi_{\omega_0}= L_{\omega_0^{-1}} \circ \theta_t \circ L_{\omega_0}:  g \mapsto k_0 \theta(g) k_0^{-1} $ of  ${\mathcal G}$
preserves $\p$ which can be treated as the tangent space of ${\mathcal G}/{\mathcal K}$ at the identity coset $e{\mathcal K}$. 
Since $k_0 \in {\mathcal K}$ belongs to a compact group,  $\p$ decomposes as a sum of eigenspaces for the action of $k_0$,
$$ \p = \p_1 + \p_{-1} + \sum_{\alpha} (\p_\alpha + \p_{\bar{\alpha}}),$$
where $\alpha \in \Si^1$ with $\alpha \not = \pm 1$, and ${\bar{\alpha}}$ is the complex conjugate of $\alpha$ appearing with the same
multiplicity as $\alpha$ (since the action of ${\mathcal K}$, in particular of $k_0$, is on a real vector space underlying $\p$). Further, since
$\theta$ operates by $-1$ on $\p$,
and the automorphism $ \varphi_{\omega_0}:  g \mapsto k_0 \theta(g) k_0^{-1} $ has isolated fixed point at the identity element  of ${\mathcal G}/{\mathcal K}$,
$\p_{-1}=0$.
Using again that $\theta$ operates by $-1$ on $\p$,
it follows that for the action of $(1 -\varphi_{\omega_0})$ on $\p$,
$$\det(1 -\varphi_{\omega_0}) = 2^{e} \cdot \prod_{\alpha \not = \pm 1}(1+\alpha)(1+\overline{\alpha}) > 0,$$
(where $e$ is the dimension of $\p_1$) completing the proof of the theorem.
 \end{proof}

 \begin{remark}
 The cohomology of compact symmetric spaces  ${\mathcal G}/{\mathcal K}$ (in the notation of Theorem \ref{symmetric}) is very well known and well studied when the rank of the group ${\mathcal G}$ is the same as that of
 $ {\mathcal K}$, such as in Borel's thesis. However, precise results on cohomology of a compact symmetric space  ${\mathcal G}/{\mathcal K}$ in the unequal rank case seem less considered. Here is the result for
$ {\mathcal G} = \U_n$ that we have been able to locate (still very classic). Presumably there are similar precise results on cohomology of other compact symmetric spaces  ${\mathcal G}/{\mathcal K}$ in the unequal rank case, we do not know.
 \end{remark}

 \begin{prop} \label{MT}
\begin{enumerate}
\item
  $H^*([\U_{n} \times \U_{n}] / \Delta (\U_{n}), \C)  \cong H^*(\U_{n}, \C) \cong \Lambda(e_1, e_3, \cdots, e_{2n-1}),$
   where $\Lambda(e_1, e_3, \cdots, e_{2n-1})$ denotes the exterior algebra on odd degree generators $e_{2i-1}$.

\item   $H^*(\gl_{2n+1}(\C) , \SO_{2n+1}, \C) \cong H^*(\U_{2n+1}  / \SO_{2n+1}, \C)  \cong
   \Lambda(e_1, e_5, \cdots, e_{4n+1}),$

   \item
   $H^*(\gl_{2n}(\C) , \SO_{2n}, \C) \cong H^*(\U_{2n}  / \SO_{2n}, \C)  \cong
\Lambda(e_1, e_5, \cdots, e_{4n-3}) \otimes \Delta(e_{2n}),$
where $ \Delta(e_{2n})$ represents the two dimensional algebra generated by $e_{2n}$ in degree $2n$, the so-called Euler class. 
   \item
   $H^*(\gl_{2n}(\C) , \OO_{2n}, \C) \cong H^*(\U_{2n}  / \OO_{2n}, \C)  \cong
\Lambda(e_1, e_5, \cdots, e_{4n-3}).$
\item $H^*(\gl_{2n}(\C) , \Sp_{2n}, \C) \cong H^*(\U_{2n}  / \Sp_{2n}, \C)  \cong
\Lambda(e_1, e_5, \cdots, e_{4n-3}).$

\end {enumerate}

   \end{prop}
 \begin{proof}
   The first part  on $\GL_{n}(\C)$, i.e.,
   $H^*(\U_{n}, \C) \cong \Lambda(e_1, e_3, \cdots, e_{2n-1}),$ is a well-known computation on cohomology of compact Lie groups.   For the cohomology of $\U_n/\SO_n$, and $\U_{2n}/\Sp_{2n}$, we refer to \cite[III.6]{MT}.  To deduce (4) from (3) note that the Euler class $e_{2n}$ is the $n$th Chern class of the tautological bundle on $\U_{2n}/\SO_{2n}$ and the action of $\OO_{2n}/\SO_{2n}$ on it reverses the orientation of the tautological bundle, so that it acts by $-1$ on $e_{2n}$. 
   \end{proof}

\begin{remark}
Theorem \ref{symmetric} and Proposition \ref{MT} suggest that (in the notation of Theorem \ref{symmetric}) the cohomology of ${\mathcal G}/{\mathcal K}$ carries an action of an exterior algebra of dimension $d = \dim(S)-\dim({\mathcal T})$, not necessarily generated in degree one.

\end{remark}

\begin{example} 
The assertion of Theorem \ref{packetsum} can be checked  directly in the case of $\sG=\GL(N)$, using the classical facts summarized in Proposition \ref{MT}.

The $\theta$-stable parabolic subalgebras for $\mathfrak{gl}(N,\R)$ modulo $K$-conjugacy are in bijection with self-dual ordered partitions  of $N$ (cf. Example \ref{spehq}).  The partition 
$$
(n_d, \dots, n_1, n_0, n_1, \dots, n_d)
$$ 
with $n_0+\sum_{i\geq 1} 2n_i = N$ gives a $\theta$-stable parabolic subalgebra $\q$ with Levi 
$$
L \cong \GL(n_0, \R) \times \prod_{i\geq 1} \GL(n_i,\C). 
$$   
The cohomology $H^*(\g,K, A_\q)$ is, up to a degree shift, 
$$
H^*(\l, K\cap L, \C) = 
H^*(\mathfrak{gl}(n_0, \R), \OO_{n_0}, \C) \otimes  
\bigotimes_{i\geq 1}
H^*(\mathfrak{gl}(n_i, \C), \U_{n_i}, \C) 
$$ 
(e.g. by the computation in Proposition~\ref{cohprop}). 
By the previous proposition the $n_0$-term is an exterior algebra on $\lfloor n_0/2\rfloor$ generators (there is one generator for each odd integer $\leq n_0$).  For $i\geq 1$ the $n_i$-term is an exterior algebra on $n_i$ generators (there is one generator for each odd integer $\leq 2n_i$).  
Since $n_0$ and $N$ have the same parity, we see that for any self-dual ordered partition $H^*(\l,K\cap L, \C)$ is an exterior algebra on $\lfloor N/2\rfloor$ generators.  (The precise degrees of the generators depend on the partition.)   Thus $\sum_{i\geq 0}\dim_\C H^i(\l,K \cap L,\C)= 2^{\lfloor N/2\rfloor}$ is independent of the chosen partition.  Since the $AJ$-packets for $\GL(N,\R)$ are singletons, this gives the assertion of Theorem \ref{packetsum} for $\GL(N)$ for $(\g,K)$-cohomology. 

Now consider $(\g,K_0)$-cohomology.   
For $\q$ associated with $(n_d, \dots, n_0, \cdots, n_d)$ the cohomology $H^*(\g,K_0, A_\q)$ is, up to a degree shift, $H^*(\l, K_0\cap L, \C)$ (in this case we are in the situation of Theorem 3.3. of \cite{VZ}), and this is given by 
$$
H^*(\l, K\cap L, \C) = 
H^*(\mathfrak{gl}(n_0, \R), \SO_{n_0}, \C) \otimes  
\bigotimes_{i\geq 1}
H^*(\mathfrak{gl}(n_i, \C), \U_{n_i}, \C). 
$$ 
Since $n_0$ and $N$ have the same parity the case of odd $N$ is  exactly the same as before, i.e. $H^*(\mathfrak{l},K_0\cap L, \C)$ is an exterior algebra on $\lfloor N/2\rfloor$ generators for any $\sL$, so let us assume from now that $N$ is even.  
Since $n_0$ is even, by the previous proposition the $n_0$-factor is an exterior algebra on $n_0/2$ generators in odd degrees, plus one more generator, the Euler class $e_{2n_0}$ in degree $n_0$.  So  $H^*(\l, K_0\cap L,\C)$ is an exterior algebra on $N/2 +1$ generators.  If $n_0=0$ it is an exterior algebra on $N/2$ generators by the proposition, but we must now take into account the fact that the restriction to $(\g,K_0)$ of the $A_\q$ for $\GL(N,\R)$ corresponding to the partition has two summands, both of which are cohomological for $(\g,K_0)$ (indeed, they constitute the tempered cohomological $L$-packet for $(\g,K_0)$).  Summing the contribution from these two representations gives $2^{N/2}+2^{N/2}=2^{(N/2) +1}$.  

\end{example}

Finally, we give a refinement of Theorem~\ref{packetsum} using Hodge polynomials in the case $G/K$ has an invariant complex structure.   Let us assume for simplicity that $\sG$ is semisimple, $G=\sG(\R)$ is connected, and $K \subset G$ is a maximal compact, and that $G/K$ has a $G$-invariant complex structure.    We restrict ourselves to trivial coefficients although the general case is easily dealt with.

For a nonnegatively bigraded vector space 
$V=\bigoplus_{p,q\geq 0} V^{p,q}$,  define 
$$
P_V(u,v) = \sum_{p,q} (\dim V^{p,q})\, u^pv^q \quad \in \Z[u,v].
$$
If $H$ is a real Hodge structure then this is the Hodge polynomial and 
Hodge symmetry implies that $P_H(u,v)=P_H(v,u)$. 
For a real Hodge structure $H$ which additionally satisfies a duality $\dim H^{p,q} = \dim H^{n-q,n-p}$, we define the {\it dual} or {\it mirror} Hodge polynomial by 
\begin{equation}
P^\#_H(u,v) := \sum_{p,q} (\dim H^{n-p,q})\,u^p v^q = u^nP_H(u^{-1},v). 
\label{mirrorpoly}
\end{equation}
For example, if $H=H^*(X)$ is the total cohomology of a compact K\"ahler manifold $X$ of dimension $n$, then this is the Hodge polynomial of a mirror manifold of $X$, if one exists, since the Hodge diamond is reflected across the line $p=n/2$.  

For a unitary cohomological representation $A_\q=\mathcal{R}^S_\q(\C)$, where $\q=\l+\u$, there is a bigrading in $H^*(\g,K,A_\q)$ coming from $\g/\k=\p=\p^+\oplus \p^-$ given by the complex structure. 
We then have  
$$
P_{H^*(\g,K,A_\q)}(u,v) = u^{R^+}v^{R^-}\,P_{H^*(\widehat{X}_L)}(u,v) 
$$
where $R^\pm = \dim(\u\cap \p^\pm)$ (see \cite[Prop.~6.19]{VZ}). 
This leads us to define 
$$
P^\#_{H^*(\g,K,A_\q)}(u,v) := u^{R^+}v^{R^-}\,P^\#_{H^*(\widehat{X}_L)}(u,v). 
$$
(Note that $H^*(\g,K,A_\q)$ is not a real Hodge structure except in certain special cases, so that $P^\#_{H^*(\g,K,A_\q)}$ is not already defined by (\ref{mirrorpoly}) in general.  In the cases where $R^+=R^-$ it so happens that 
$H^*(\g,K,A_\q)$ is a real Hodge structure, and there is potentially another meaning for $P^\#_{H^*(\g,K,A_\q)}$, which is {\it not} the one we are taking.) 
If $\pi$ is a discrete series representation then $P_{H^*(\g,K,\pi)}=P^\#_{H^*(\g,K,\pi)}=u^pv^q$  where $p,q$ with $p+q=\dim_\C G/K$ is the unique bidegree with $H^{p,q}(\g,K,\pi)\neq \{0\}$. 

With these definitions we have the following, which reduces to Theorem~\ref{packetsum}  (for $G/K$ with an invariant complex structure) on setting $u=v=1$.  The special case of the discrete series packet appears as 
\cite[Corollary 8.2]{Gr}.

\begin{thm} \label{mirror}
Suppose that $G$ is connected and $G/K$ has a $G$-invariant complex structure. 
Then for an Adams-Johnson packet $\Pi$ of $(\g,K)$-cohomological representations, the sum 
$$
\sum_{\pi \in \Pi} P^\#_{H^*(\g,K,\pi)}(u,v)  
$$ 
is independent of the packet, and hence is given by 
$$
P^\#_{H^*(G^u/K)}(u,v)=\sum_{0\leq i\leq 2n} \dim H^i(G^u/K) \,u^{n-i/2}v^{i/2} 
$$
where $n = \dim_\C G/K$. 
\end{thm} 

\begin{proof}
This follows by the same proof as was given for Theorem~\ref{packetsum} in the special case of $G$ having discrete series by substituting \cite[Proposition~6.19]{VZ} for \cite[Theorem~5.5]{VZ} and minding ones $u^p$s and $v^q$s.  We leave the details to the reader. 
\end{proof}

\begin{example}
As an example of the previous theorem, consider the case $\sG=\Sp(4)$.  There are eight representations with cohomology with trivial coefficients, which we can index by the bidegree in which their minimal-degree cohomology appears: 
$$
\pi^{30}, \pi^{21}, \pi^{12}, \pi^{03} 
\qquad 
\pi^{20}, \pi^{11}, \pi^{02} 
\qquad 
\pi^{00}=\C.
$$  
The nontrivial $AJ$-packets are 
$\{\pi^{30}, \pi^{21}, \pi^{12}, \pi^{03}\}$ (the discrete series $L$-packet), $\{\pi^{20},\pi^{02}\}$ (also an $L$-packet), and $\{\pi^{30},\pi^{11},\pi^{03}\}$ (the packet built on the $L$-packet $\{\pi^{11}\}$). 
The sum of mirror Hodge polynomials is 
$$
u^3+u^2v+uv^2+v^3=u^2(u+v)+(u+v)v^2=u^3+uv(u+v)+v^3, 
$$ 
with each expression corresponding to one of these packets. 
\end{example}

\begin{example} \label{Un1mirror} 
For $\U(n,1)$, Theorem \ref{mirror} justifies the picture of $AJ$-packets described in Example~\ref{Un1example} in Section~\ref{AJexam}.  We see that an $AJ$-packet for $\U(n,1)$ contains a discrete series representation if and only if the corresponding partition of $N=n+1$ contains the factor $1$.  Indeed, these facts follow immediately by considering of the mirror Hodge polynomial.

\end{example}

\begin{remark}
For groups $G$ with quaternionic discrete series there should be an interesting statement in terms of $\SU(2)$-representations 
analogous to Theorem~\ref{mirror}. 
\end{remark}

\begin{remark}
There is a natural generalization of Theorem \ref{mirror} for general groups in terms of decompositions of the set of fundamental series representations which we will discuss elsewhere. 

\end{remark}

\textbf{Acknowledgements.} The authors thank C. Moeglin, D. Renard, O. Ta\"ibi and Sandeep Varma
for useful comments, and D. Vogan for his encouragement.
The authors thank the referee for a very meticulous and excellent job, catching many inaccuracies in the earlier version.
The first author thanks DAE, India,  for support through the grant 
PIC 12-R\&D-TFR-5.01-0500.
The second author  thanks  SERB, India for its support
through the JC Bose
Fellowship, JBR/2020/000006. 
His work was also supported by a grant of
the Government of the Russian Federation
for the state support of scientific research carried out
under the  agreement 14.W03.31.0030 dated 15.02.2018.

\vspace{.5cm}
\noindent

\noindent
AN: Tata Insititute of Fundamental Research,
Colaba,
Mumbai-5.

\noindent
Email: {\tt arvind@math.tifr.res.in}.
\vspace{.5cm}

\noindent DP:  Indian Institute of Technology Bombay, Mumbai. 

\noindent Email: {\tt prasad.dipendra@gmail.com}
\end{document}